\documentclass[a4paper,fleqn]{cas-sc}

\usepackage[authoryear,longnamesfirst]{natbib}

\newtheorem{thm}{Theorem}[section]
\newtheorem{lemma}[thm]{Lemma}
\newtheorem{prop}[thm]{Proposition}

\newproof{proof}{Proof}

\renewenvironment{proof}[1][Proof]{
\medskip {\noindent \bfseries #1. }}{\hfill \qed \medskip }

\newtheorem{definition}[thm]{Definition}
\newtheorem{example}{Example}[section]

\newtheorem{remark}{Remark}

\usepackage{lipsum}
\usepackage{amsfonts}
\usepackage{graphicx}
\usepackage{epstopdf}
\usepackage{algorithm}
\usepackage{algpseudocode}
\ifpdf
  \DeclareGraphicsExtensions{.eps,.pdf,.png,.jpg}
\else
  \DeclareGraphicsExtensions{.eps}
\fi

\usepackage{amsopn}
\usepackage{mathtools}
\usepackage{amsmath}
\usepackage{multirow}
\usepackage{enumitem}
\usepackage{pdflscape}

\newcommand{\calC}{\mathcal{C}}
\newcommand{\calD}{\mathcal{D}}
\newcommand{\calO}{\mathcal{O}}
\newcommand{\calL}{\mathcal{L}}
\newcommand{\calS}{\mathcal{S}}
\newcommand{\calN}{\mathcal{N}}
\newcommand{\calH}{\mathcal{H}}
\newcommand{\calM}{\mathcal{M}}
\newcommand{\calV}{\mathcal{V}}
\newcommand{\calP}{\mathcal{P}}
\newcommand{\calG}{\mathcal{G}}
\newcommand{\calR}{\mathcal{R}}
\newcommand{\real}{\mathbb{R}}

\newcommand{\bbL}{\mathbb{L}}

\newcommand{\bbP}{\mathbb{P}}
\newcommand{\bbN}{\mathbb{N}}
\newcommand{\bbE}{\mathbb{E}}
\newcommand{\bbD}{\mathbb{D}}
\newcommand{\AI}{\text{AI}}
\newcommand{\calE}{\mathcal{E}}
\newcommand{\bbF}{\mathbb{F}}
\newcommand{\chol}{\text{Chol}}
\newcommand{\bfzero}{\mathbf{0}}

\newcommand{\tr}{\textsf{tr}}
\newcommand{\calU}{\mathcal{U}}
\newcommand{\parentheses}[1]{\left(#1\right)}
\newcommand{\set}[1]{\left\{#1\right\}}
\newcommand{\floor}[1]{\lfloor #1 \rfloor}
\newcommand{\PICSE}{\texttt{PICSE}}
\newcommand{\KMLE}{\texttt{KMLE}}
\newcommand{\CSE}{\texttt{CSE}}

\DeclareMathOperator{\argmin}{argmin}
\DeclareMathOperator{\Exp}{Exp}
\DeclareMathOperator{\Hess}{Hess}
\DeclareMathOperator{\grad}{grad}
\DeclareMathOperator{\diag}{diag}
\DeclareMathOperator{\sym}{\textsf{sym}}
\DeclareMathOperator{\skewed}{\textsf{skew}}

\DeclareMathOperator{\rank}{\text{rank}}

\makeatletter

\long\def\count@stringtoks#1{\tc@earg\count@toks{\string#1}}
\long\def\count@toks#1{\the\numexpr-1\count@@toks#1.\tc@endcnt}
\long\def\count@@toks#1#2\tc@endcnt{+1\tc@ifempty{#2}{\relax}{\count@@toks#2\tc@endcnt}}
\def\tc@ifempty#1{\tc@testxifx{\expandafter\relax\detokenize{#1}\relax}}
\long\def\tc@earg#1#2{\expandafter#1\expandafter{#2}}
\long\def\tctestifnum#1{\tctestifcon{\ifnum#1\relax}}
\long\def\tctestifcon#1{#1\expandafter\tc@exfirst\else\expandafter\tc@exsecond\fi}
\long\def\tc@testxifx{\tc@earg\tctestifx}
\long\def\tctestifx#1{\tctestifcon{\ifx#1}}
\long\def\tc@exfirst#1#2{#1}
\long\def\tc@exsecond#1#2{#2}
\makeatother

\makeatletter
\newcommand*{\rom}[1]{\expandafter\@slowromancap\romannumeral #1@}
\makeatother
\allowdisplaybreaks

\def\tsc#1{\csdef{#1}{\textsc{\lowercase{#1}}\xspace}}
\tsc{WGM}
\tsc{QE}
\tsc{EP}
\tsc{PMS}
\tsc{BEC}
\tsc{DE}

\begin{document}
\let\WriteBookmarks\relax
\def\floatpagepagefraction{1}
\def\textpagefraction{.001}
\shorttitle{Core covariance geometry}
\shortauthors{B. Sung}

\title [mode = title]{Covariance estimation for matrix-variate data via fixed-rank core covariance geometry}       

\author[1]{Bongjung Sung}[orcid=0000-0002-2464-9977]
\cormark[1]
\fnmark[1]
\ead{bongjung.sung@duke.edu}

\affiliation[1]{organization={Department of Statistical Science, Duke University},
            addressline={214 Old Chemistry}, 
            city={Durham},
            postcode={27708}, 
            state={North Carolina},
            country={U.S.A}}

\cortext[cor1]{Corresponding author}

\begin{abstract}
 We study the geometry of the fixed-rank core covariance manifold arising from the Kronecker-core decomposition of covariance matrices. As shown in Hoff, McCormack, and Zhang (2023), every covariance matrix $\Sigma$ of $p_1\times p_2$ matrix-variate data uniquely decomposes into a separable component $K$ and a core component $C$. Such a decomposition also exists for rank-$r$ $\Sigma$ if $p_1/p_2+p_2/p_1<r$, with $C$ sharing the same rank. If this core $C$ exhibits a partial-isotropy structure, then a partial-isotropy rank-$r$ core is a non-trivial convex combination of a rank-$r$ core and $I_p$ for $p:=p_1p_2$, where the weight on $I_p$ measures the deviation of $\Sigma$ from separability. This motivates studying the geometry of the space of rank-$r$ cores, $\calC_{p_1,p_2,r}^+$. We show that $\calC_{p_1,p_2,r}^+$ is a smooth manifold, except for a measure-zero subset associated with canonical decomposability. When $r=p$, $\calC_{p_1,p_2}^{++}:=\calC_{p_1,p_2,p}^+$ is itself a smooth manifold. The geometric properties, including smoothness of the positive definite cone via separability and the Riemannian gradient and Hessian operator relevant to $\calC_{p_1,p_2,r}^+$, are also derived. As an application, we propose a partial-isotropy core shrinkage estimator for matrix-variate data.  
\end{abstract}

\begin{keywords}
canonical decomposability \sep core covariance matrix \sep fixed-rank \sep 
Kronecker-core decomposition \sep Riemannian manifold optimization \sep
separable covariance matrix
\end{keywords}

\maketitle

\section{Introduction}\label{sec1}

Matrix-variate data arise in a range of modern applications, e.g., microarray data \cite{allen2012}, phonetic data \cite{pigoli2014}, and audio data \cite{warden2018}. A fundamental task for making statistical inference with such data is to estimate the population covariance matrix. For matrix-variate data, a commonly used covariance model is the separable (Kronecker) covariance model \cite{dawid1981}. Namely, for a zero-mean $p_1\times p_2$ random matrix $Y$, its $p_1p_2\times p_1p_2$ covariance matrix $\Sigma$ is formulated as 
\begin{align}\label{sec1.eq1}
    \Sigma=V[Y]\equiv V[\text{vec}(Y)]=\Sigma_2\otimes \Sigma_1,
\end{align}
where $\Sigma_1\in\calS_{p_1}^{++}$ and $\Sigma_2\in\calS_{p_2}^{++}$ correspond to row and column covariance matrices, respectively. Here $\otimes$ denotes the Kronecker product and $\calS_q^{++}$ is the set of $q\times q$ (strictly) positive definite matrices. Note that we assume $p_1,p_2\geq 2$ to emphasize the matrix structure of the data in this article, and let $p=p_1p_2$. It follows that 
\begin{align*}
    \bbE[YY^\top]=\tr(\Sigma_2)\Sigma_1,\quad \bbE[Y^\top Y]=\tr(\Sigma_1)\Sigma_2,
\end{align*}
enabling separate inference of correlation structures of row and column variables \cite{dawid1981,wang2022}. This model is commonly used due to its parsimony and interpretability, involving at most $O(p_1^2+p_2^2)$ correlations between variables. We denote the space of such separable covariance matrices by $\calS_{p_1,p_2}^{++}$.

Due to its interesting geometry and statistical interpretability, several works have studied covariance estimation over the space $\calS_{p_1,p_2}^{++}$ from both the applied math and statistical communities. Suppose $Y_1,\ldots,Y_n\overset{i.i.d.}{\sim} N_{p_1\times p_2}(0,\Sigma_2\otimes \Sigma_1)$. From applied math perspective, $\calS_{p_1,p_2}^{++}$ is a totally geodesic submanifold of $\calS_p^{++}$ under affine-invariant metric, and the negative log-likelihood is geodesically convex on $\calS_{p_1,p_2}^{++}$ under this metric \cite{wiesel2012}, which facilitates constrained estimation. The estimation over $\calS_{p_1,p_2}^{++}$ also enjoys group-invariant properties. Specifically, the action of $G:=GL_{p_1}\times GL_{p_2}$ on $\real^{p_1\times p_2}$ by $(Y,(A,B))\in \real^{p_1\times p_2}\times G \rightarrow AYB^\top\in \real^{p_1\times p_2}$ induces the action of the group of Kronecker non-singular matrices $GL_{p_1,p_2}$ on the parameter space $\calS_{p_1,p_2}^{++}$ as $(\Sigma_2\otimes \Sigma_1,B\otimes A)\in \calS_{p_1,p_2}^{++}\times GL_{p_1,p_2}\rightarrow (B\Sigma_2 B^\top)\otimes (A\Sigma_1 A^\top)\in  \calS_{p_1,p_2}^{++}$. Using the group-invariant theory, the threshold on the sample size $n$, for which the MLE of $\Sigma_2\otimes \Sigma_1$ on $\calS_{p_1,p_2}^{++}$ exists, has been analyzed \cite{derksen2021,amendola2021}. On the other hand, statisticians often make use of the maximum likelihood estimation \cite{dutilleul999} and may be more interested in structured estimation with assumptions on each factor $\Sigma_1$ and $\Sigma_2$, to improve the interpretation of correlation structures for row and column variables, respectively. These include sparsity \cite{leng2018,tsiligkaridis2013}, and banding \cite{zhang2023}. At the intersection of these two communities is the information geometry as studied by \cite{mccormack2025}. 

However, from a statistical perspective, the separability assumption on $\Sigma$ may oversimplify its correlation structure as $p$ grows, allowing at most $O(p_1^2+p_2^2)=o(p^2)$ correlations, whereas $O(p^2)$ correlations for the unstructured $\Sigma$. Hence, the separability assumption is often inappropriate, as also pointed out by \cite{gneiting2002,gneiting2007}. To relax this assumption, \cite{hoff2023} introduced the core covariance matrix. They showed that every $\Sigma\in\calS_p^{++}$ admits a unique decomposition into a separable component $K$, representing the most separable part of $\Sigma$, and a core component $C$, which is a whitened version of $\Sigma$ via the identifiable square root $K^{1/2}$ of $K$, e.g., symmetric square root and Cholesky factor. This decomposition, referred to as a Kronecker-core decomposition (KCD), represents $\Sigma$ as $K^{1/2}CK^{1/2,\top}$, i.e., $C=K^{-1/2}\Sigma K^{-1/2,\top}$. In particular, $\Sigma\in\calS_{p_1,p_2}^{++}$ if and only if $C=I_p$. Such a decomposition may also exist for rank$-r$ $\Sigma$ if $p_1/p_2+p_2/p_1<r$ \cite{drton2021,derksen2021,soloveychik2016}, with $C$ sharing the same rank as $\Sigma$ because $K^{1/2}$ is non-singular. By Proposition $5$ of \cite{hoff2023}, the dimension of the space of full-rank $C$ is $O(p^2)$, whereas that of $\calS_{p_1,p_2}^{++}$ is $O(p_1^2+p_2^2)=o(p^2)$. Thus, in a high-dimensional regime where the sample size $n$ is less than $p$, estimating $\Sigma$ is numerically and statistically unstable without any structural assumption on $C$.    

As discussed by \cite{hoff2023}, one remedy is to introduce a partial-isotropy rank$-r$ structure to $C$, which commonly arises in factor analysis \cite{basilevsky1994,bartholomew2011}. Specifically, $\Omega\in\calS_p^{++}$ has such a structure if $\lambda_1(\Omega)\geq \cdots \geq \lambda_r(\Omega)>\lambda_{r+1}(\Omega)=\cdots=\lambda_p(\Omega)>0$, which is equivalent to $\Omega=AA^\top+c I_p$ for some $A\in\real^{p\times r}$ of full-column rank and constant $c>0$. Nevertheless, \cite{hoff2023} left covariance estimation with such a partial-isotropy core as an open question, as characterizing such a core is crucial. In this article, we show that if $C$ exhibits a partial-isotropy rank$-r$ structure for fixed $r>p_1/p_2+p_2/p_1$, $C$ is a non-trivial convex combination of a rank$-r$ core and a trivial core $I_p$. This leads to the covariance model
\begin{align}\label{sec1.eq2}
    Y_1,\ldots,Y_n\overset{i.i.d.}{\sim} N_{p_1\times p_2}(0,\Sigma) \text{ for } \Sigma=K^{1/2}((1-\lambda)AA^\top+\lambda I_p)K^{1/2,\top},
\end{align}
where $\lambda\in(0,1)$, the identifiable square root $K^{1/2}$ of $K$, and $A\in\real^{p\times r}$ of full-column rank such that $AA^\top$ is a core. We refer to this model as a partial-isotropy core covariance model. As shown in Section \ref{sec2.1} and \ref{sec6}, the lambda $\lambda$ on $I_p$ quantifies how far $\Sigma$ is from being separable. 

To incorporate a partial-isotropy $C$ into the estimation of $\Sigma$ as in (\ref{sec1.eq2}), we need a proper understanding of the space of rank$-r$ cores, denoted $\calC_{p_1,p_2,r}^+$, motivating the study of its geometry. Therefore, this article is devoted to establishing the geometry of $\calC_{p_1,p_2,r}^+$ and thus constructing a shrinkage estimator based on this geometry. We shall emphasize that the geometry of $\calC_{p_1,p_2,r}^+$ itself is interesting, apart from the statistical motivation. Namely, the smoothness of this set is associated with connectivity between row and column variables in terms of the undirected bipartite graph and induces the smooth structure of the ambient space $\calS_{p}^{++}$ via separability. Moreover, the smoothness facilitates constrained estimation, exploiting the curvature of the negative log-likelihood to find the optimal $C$.

Our contributions are summarized in three main strands. First, we show that $\calC_{p_1,p_2,r}^+$ is a compact, smooth, embedded submanifold of $\calS_{p,r}^+$, the set of rank$-r$ PSD matrices. A key insight for the proof is that if $C=AA^\top \in \calC_{p_1,p_2,r}^+$ for $A=[\text{vec}(A_i),\ldots,\text{vec}(A_r)]$ with $A_i\in\real^{p_1\times p_2}$, Proposition $3$ of \cite{hoff2023} implies that 
\begin{align*}
    \sum_{i=1}^r A_iA_i^\top=p_2I_{p_1},\quad \sum_{i=1}^r A_i^\top A_i=p_1I_{p_2}. 
\end{align*}
Given $p_1/p_2+p_2/p_1<r$, we construct $\calC_{p_1,p_2,r}^+$ as the smooth image of the smooth manifold $\calD_{p_1,p_2,r}$ consisting of $\tilde{A}=(A_1,\ldots,A_r)$ satisfying the above. While the proof is straightforward when $r=p$, the rank-deficient case requires additional technical work. Namely, the canonically decomposable $\tilde{A}$, i.e., there exist non-singular matrices $(P,Q)$ such that $PA_iQ^{-1}$ is of non-trivial block-diagonal form, prevents $\calD_{p_1,p_2,r}$ from being a smooth manifold. The canonical decomposability notion arises in the study of the threshold on $r$ for which a generic $\Omega\in\calS_{p,r}^+$ admits the Kronecker MLE \cite{derksen2021,soloveychik2016} and hence the KCD. Essentially, the canonically indecomposable $\tilde{A}$ guarantees that the row variables and column variables are well-connected, in turn making the set $\calC_{p_1,p_2,r}^+$ smooth. 

The second contribution are results regarding differential geometry of $\calC_{p_1,p_2}^{++}$, $\calC_{p_1,p_2,r}$, and the quotient manifold $\calC_{p_1,p_2,r}/\calO_r$. Let $k(\Sigma)$ and $c(\Sigma)$ denote the separable and core components of $\Sigma$, respectively. With the map $f:\Sigma \in \calS_p^{++}\rightarrow (k(\Sigma),c(\Sigma))\in \calS_{p_1,p_2}^{++}\times \calC_{p_1,p_2}^{++}$, we show that $\calS_p^{++}$ is diffeomorphic to $\calS_{p_1,p_2}^{++}\times \calC_{p_1,p_2}^{++}$ via the map $f$ in Section \ref{sec5.1}. Therefore, we provide a new insight into the smooth structure of $\calS_p^{++}$ via separability. We also compute the differentials of $f$ and its inverse $f^{-1}$. Under the Euclidean metric, we derive the Riemannian gradient and Hessian operator on $\calC_{p_1,p_2}^{++}$ in Section \ref{sec5.2}. The same is done for $\calC_{p_1,p_2,r}$ under the same metric, which we employ in manifold optimization, and for the quotient manifold $\calC_{p_1,p_2,r}/\calO_r$ under the induced quotient metric in Section \ref{sec5.3}.

Finally, using the geometry of $\calC_{p_1,p_2,r}$, we propose a partial-isotropy core shrinkage estimator ($\PICSE$) in Section \ref{sec6}, assuming the covariance model in (\ref{sec1.eq2}) for the data. This answers the open question posed by \cite{hoff2023} on how a partial-isotropy core can be incorporated into estimating $\Sigma$. We provide an alternating minimization procedure of the negative log-likelihood in the parameters $(K^{1/2},A,\lambda)$ given in (\ref{sec1.eq2}) to compute $\PICSE$. In updating $A$, we leverage the curvature of the objective function on $\calC_{p_1,p_2,r}$ via second-order Riemannian manifold optimization using the results in Section \ref{sec5.3}, with some suitable retraction. In Section \ref{sec7}, we numerically illustrate that $\PICSE$ outperforms existing covariance estimators for matrix-variate data, and an ad hoc estimator.

\subsection{Organization}\label{sec1.1}

The rest of the article is organized as follows. Section \ref{sec1.2} introduces notations used throughout this article. In Section \ref{sec2}, we review some preliminaries, including the KCD, Riemannian manifolds, quotient manifolds, and algebraic geometry. In Section \ref{sec3}, we prove that $\calC_{p_1,p_2}^{++}$ is a compact, smooth, embedded submanifold of $\calS_p^{++}$. When $p_1/p_2+p_2/p_1<r<p$, it is shown that $\calC_{p_1,p_2,r}^+$ is a compact, smooth, embedded submanifold of $\calS_{p,r}^{+}$ in Section \ref{sec4} after removing the set of canonically decomposable matrices, using the proof strategy developed in Section \ref{sec3}. In Section \ref{sec5}, we establish the diffeomorphic relationship between $\calS_p^{++}$ and $\calS_{p_1,p_2}^{++}\times \calC_{p_1,p_2}^{++}$. We also derive the differential geometric quantities relevant to $\calC_{p_1,p_2}^{++}$, $\calC_{p_1,p_2,r}$, and $\calC_{p_1,p_2,r}/\calO_r$ under the Euclidean metric. Leveraging the geometry of $\calC_{p_1,p_2,r}$, the partial isotropy core shrinkage estimator ($\PICSE$) is proposed in Section \ref{sec6}, supported by numerical illustrations in Section \ref{sec7}. Section \ref{sec8} concludes the article with a discussion. All the omitted proofs are deferred to Appendix \ref{append.A}. The formulas of Euclidean derivative and Hessian operator associated with computing $\PICSE$ are provided in Appendix \ref{append.B}. 

\subsection{Notations}\label{sec1.2}
In this section, we collect the notations used in this article as follows:
\begin{itemize}
\item $\calS_p:=\set{\Sigma\in \real^{p\times p}:\Sigma=\Sigma^\top}$.
    \item $\calS_{p}^+:=\set{\Sigma\in\calS_p:\Sigma\succeq \bfzero_{p\times p} }$.
      \item $\calS_{p}^{++}:=\set{\Sigma\in\calS_{p}^+:\Sigma\succ \bfzero_{p\times p} }$
      \item $\bar{\calS}_p^{++}:=\set{\Sigma\in\calS_p^{++}:\tr\parentheses{\Sigma}=1}$ and $\bbP(\calS_{p}^{++}):=\set{\Sigma\in\calS_p^{++}:|\Sigma|=1}$.
    \item $\calS_{p,r}^+:=\set{\Sigma\in \calS_p^+:\text{rank}(\Sigma)=r}$. Note that $\calS_{p}^{++}\equiv \calS_{p,p}^+$.
    
    \item $\calS_{p_1,p_2}^{++}:=\set{\Sigma_2\otimes \Sigma_1:\Sigma_1\in \calS_{p_1}^{++},\Sigma_2\in\calS_{p_2}^{++}}$ for the Kronecker product $\otimes$.
    \item $\calC_{p_1,p_2}^{++}:=\set{C\in\calS_p^{++}: k(C)=I_p}$.
    \item $\calL_p:=\set{L\in\real^{p\times p}:L_{ij}=0\text{ for }i>j}$.
    \item $\calL_p^{++}:=\set{L\in\calL_p: L_{ii}>0}$ and $\bbP(\calL_p^{++})=\set{L\in\calL_p^{++}:|L|=1}$.
    \item For given $\Sigma\in\calS_p^{++}$, $\calL(\Sigma)\in\calL_p^{++}$ denotes its unique Cholesky factor.
    \item $\calL_{p_1,p_2}^{++}:=\set{L_2\otimes L_1:L_1\in\calL_{p_1}^{++},\; L_2\in\calL_{p_2}^{++}}$.
    \item $\calO_p:=\set{O\in\real^{p\times p}: OO^\top=O^\top O=I_p}$.
        \item $\calO_{p,q}:=\set{O_2\otimes O_1\in\real^{pq\times pq}: O_1\in \calO_p,O_2\in\calO_q}$.
    \item $K_{m,n}:$ a $mn\times mn$ commutation matrix such that $K_{m,n}\text{vec}(B^\top)=\text{vec}(B)$ for $B\in\real^{m \times n}$. 
    \item $GL_p$ : a general linear group of order $p$.
        \item $GL_{p_1,p_2}:=\set{B\otimes A:A\in GL_{p_1},B\in GL_{p_2}}.$
    \item For given a matrix $M,$ $C(M)$ and $N(M)$ denote the column and (right) null space of $M$, respectively.
    \item $\real_*^{p\times q}:=\set{X\in \real^{p\times q}:\text{rank}(X)=\min\set{p,q}}$.
    \item A map $\varphi_{p_1,p_2,r}:(\real^{p_1\times p_2})^r\rightarrow \real^{p_1p_2\times r}$ is defined by
    \begin{align*}
        \varphi_{p_1,p_2,r}(A):=\varphi_{p_1,p_2,r}(A_1,\ldots,A_r)=[\text{vec}(A_1),\ldots,\text{vec}(A_r)].
    \end{align*}
    Note that the map $\varphi_{p_1,p_2,r}$ is clearly a diffeomorphism. 
    \item $(\real^{p\times q} )_*^r:=\set{A=(A_1,\ldots,A_r)\in(\real^{p\times q})^r:\varphi_{p,q,r}(A)\in \real_*^{pq\times r}}$.
    \item For $u\in \real^{pq}$, $\textsf{mat}_{p\times q}(u)$ denotes the $p\times q $ matrix by reshaping $u$.
    \item $\textsf{Skew}_p:=\set{M\in\real^{p\times p}:M=-M^\top}$.
    \item For $M\in\real^{p\times p}$, $\sym(M):=(M+M^\top)/2$ and $\skewed(M):=(M-M^\top)/2$
    \item For $M\in\calS_p$, $\bbD(M):=\text{diag}(M)$ and $\floor{M}$ denotes strictly lower triangular part of $M$. Also, $(M)_{\frac{1}{2}}:=\floor{M}+\bbD(M)/2$ and $\bbL(M):=\floor{M}+\bbD(M)$.
    \item For $B_1,\ldots,B_R\in \real^{m\times n}$, a block-diagonal sum of $B_1,\ldots,B_R$ is given by $\bigoplus_{i=1}^R B_i:=\text{diag}(B_1,\ldots,B_R)$. 
    \item $\real_+:=\set{a\in\real:a>0}$.
    \item $\bfzero_m$ and $\bfzero_{m\times n}$ denote the $m-$dimensional zero vector and the $m\times n$ zero matrix, respectively. 
 \item For $M\in\real^{m\times n}$, $\sigma_i(M)$ denotes the $i$th largest signular value of $M$ and $||M||_2:=\sigma_1(M)$. Also, if $M\in\calS_p$, $\lambda_i(M)$ denotes the $i$th largest eigenvalue of $M$.
    \item For given $n\in\mathbb{N}$, $[n]:=\set{1,\ldots,n}$.
\end{itemize}
We shall refer to $\calS_{p_1,p_2}^{++}$ as the Kronecker covariance manifold, and to $\calC_{p_1,p_2}^{++}$ as the core covariance manifold. Note that the map $k$ associated with $\calC_{p_1,p_2}^{++}$ defined above is referred to as a Kronecker map, which will be formally defined in Section \ref{sec2.1}. We also define a block partition of a symmetric matrix, and introduce the partial trace operators. Suppose $M\in\calS_p$ is partitioned as 
\begin{align*}
M=\left[\begin{array}{cccc}
  M_{[1,1]} & M_{[1,2]}  & \cdots & M_{[1,p_2]} \\
  M_{[2,1]}   & M_{[2,2]} & \cdots & M_{[2,p_2]} \\
  \vdots   & \vdots & \ddots &\vdots  \\
  M_{[p_2,1]} & M_{[p_2,2]} & \cdots & M_{[p_2,p_2]} 
\end{array}\right],
\end{align*}
where each block $M_{[i,j]}\in\real^{p_1\times p_1}$ and $p=p_1p_2$. Let $(M_{[i,j]})$ be a block partition of $M$ as above. Also, the partial trace operators $\tr_1$ and $\tr_2$ are defined by 
\begin{align*}
    \tr_1&:M\in\calS_p\rightarrow \sum_{i=1}^{p_2}M_{[i,i]}\in\calS_{p_1},\\
    \tr_2&:M\in\calS_p\rightarrow N=(n_{ij})\in\calS_{p_2},\; \text{where } n_{ij}=\tr\parentheses{M_{[i,j]}}. 
\end{align*}
In the sequels of this article, $p:=p_1p_2$. Also, $\Omega^{1/2}$ denotes a square root of $\Omega\in\calS_p^{++}$, either a symmetric square root or Cholesky factor. We will specify the choice of the square root when necessary. Otherwise, $\Omega^{1/2}$ is either one of these square roots.    

\section{Preliminaries}\label{sec2}

\subsection{Kronecker-core decomposition}\label{sec2.1}
In this section, we review the Kronecker-core decomposition proposed by \cite{hoff2023}. Suppose $\Sigma\in\calS_p^{++}$ and define a function $d:\calS_{p_1,p_2}^{++}\rightarrow \real$ by 
\begin{align}\label{sec2.1.eq1}
    d(K|\Sigma):=d(K_2\otimes K_1|\Sigma)=\tr\parentheses{\Sigma K^{-1}}+p_1\log|K_2|+p_2\log|K_1|,
\end{align}
which is equivalent to the Kullback-Leibler (KL) divergence between $N_{p_1\times p_2}(0,K_2\otimes K_1)$ and $N_{p_1\times p_2}(0,\Sigma)$. The separable (Kronecker) component of $\Sigma$, $k(\Sigma)$, is then defined to be a unique minimizer of $d$ in $K\in\calS_{p_1,p_2}^{++}$. That is,
\begin{align*}
    k(\Sigma):=\argmin_{K=K_2\otimes K_1\in\calS_{p_1,p_2}^{++}}d(K|\Sigma).
\end{align*}
Thus, $k(\Sigma)$ is the Kronecker maximum likelihood estimate (MLE) of $\Sigma$, representing the most separable component of $\Sigma$. Note that $k(\Sigma)$ uniquely exists for any $\Sigma\in\calS_p^{++}$ \cite{srivastava2008,hoff2023} and we refer the map $k:\calS_{p}^{++}$ to as the Kronecker map. 

To define the core component, let $h$ be a bijective square root map defined on $\calS_{p_1,p_2}^{++}$, e.g., a symmetric square root (i.e., $h(\calS_{p_1,p_2}^{++})\equiv \calS_{p_1,p_2}^{++}$) or a Cholesky factor (i.e., $h(\calS_{p_1,p_2}^{++})\equiv \calL_{p_1,p_2}^{++}$). By a slight abuse of notation, we write $h\in \calS_{p_1,p_2}^{++}$ (resp. $\calL_{p_1,p_2}^{++})$ when $h$ is taken to be the symmetric square root (resp. Cholesky factor). With a fixed choice of $h$, the core component of $\Sigma$ is defined to be $c(\Sigma)\equiv h(k(\Sigma))^{-1}\Sigma h(k(\Sigma))^{-\top}$. By the definition of the Kronecker map, it holds that $k(G\Sigma G^{\top})=Gk(\Sigma)G^{\top}$ for any $G\in GL_{p_1,p_2}$ (\cite{hoff2023}, Proposition $2$). Namely, the map $k$ is equivariant with respect to the action of $GL_{p_1,p_2}$. Thus, $k(c(\Sigma))=I_p$; that is, the Kronecker MLE of any core is $I_p$. This leads to the definition of the core covariance matrix as a positive definite matrix whose Kronecker MLE equals $I_p$. That is, the set of such covariance matrices is equivalent to $\calC_{p_1,p_2}^{++}\equiv\set{C\in\calS_{p}^{++}:k(C)=I_p}$. By the uniqueness of the Kronecker MLE and the bijectivity of the map $h$, the core component is uniquely defined for any $\Sigma$. Consequently, the map $c:\Sigma\in\calS_p^{++}\rightarrow h(k(\Sigma))^{-1}\Sigma h(k(\Sigma))^{-\top} \in \calC_{p_1,p_2}^{++}$ is well-defined and referred to as the core map. By the definition of the maps $k$ and $c$, every $\Sigma$ admits a unique and identifiable Kronecker-core decomposition (KCD) as $h(k(\Sigma))c(\Sigma) h(k(\Sigma))^\top$ (see Proposition $5$ of \cite{hoff2023}). The definitions of the separable and core components are summarized below.  
\begin{definition}\label{sec2.1.def1}
The Kronecker map $k:\calS_p^{++}\rightarrow \calS_{p_1,p_2}^{++}$ sends $\Sigma$ to the unique minimizer $k(\Sigma)$ of $d(\cdot|\Sigma)$ defined in (\ref{sec2.1.eq1}). For a fixed choice of the square root map $h\in\calS_{p_1,p_2}^{++}$ or $h\in\calL_{p_1,p_2}^{++}$, the map $c:\Sigma\in\calS_p^{++}\rightarrow h(k(\Sigma))^{-1}\Sigma h(k(\Sigma))^{-\top} \in\calC_{p_1,p_2}^{++}$ defines the core map. Here, $k(\Sigma)$ and $c(\Sigma)$ are referred to as the separable and core components of $\Sigma$, respectively. Also, $h(k(\Sigma))c(\Sigma)h(k(\Sigma))^\top$ is a KCD of $\Sigma$.   
\end{definition}
 
By the construction, $c(\Sigma)=I_p$ if and only if $\Sigma=k(\Sigma) \in\calS_{p_1,p_2}^{++}$. Namely, the only separable core is $I_p$. As discussed in Section \ref{sec1}, the Kronecker MLE $K$ may also uniquely exist for $\Omega\in\calS_{p,r}^+$ if $r>p_1/p_2+p_2/p_1$ \cite{drton2021,derksen2021,soloveychik2016} by taking $\Sigma=\Omega$ in (\ref{sec2.1.eq1}). Provided that $K$ uniquely exists, its core component $C$ can be also uniquely defined by whitening $\Omega$ via $K^{1/2}$ as above. Since $K^{1/2}$ is non-singular, $\Omega$ and $C$ shares the same rank as $r$. Suppose $C=AA^\top$ for $A=[\text{vec}(A_1),\ldots,\text{vec}(A_r)]\in \real_*^{p\times r}$ with $A_i\in\real^{p_1\times p_2}$. By Proposition $3$ of \cite{hoff2023}, $\tilde{A}:=(A_1,\ldots,A_r)$ should satisfy that 
\begin{align}\label{sec2.1.eq2}
    \tr_1(C)\equiv \sum_{i=1}^r A_iA_i^\top=p_2 I_{p_1},\quad \tr_2(C)\equiv  \sum_{i=1}^r A_i^\top A_i=p_1 I_{p_2}.  
\end{align}
This motivates the set of the rank$-r$ core covariance matrices defined as 
\begin{align*}
\tilde{\calC}_{p_1,p_2,r}^+\equiv\set{C\in\calS_{p,r}^+:\tr_1(C)=p_2I_{p_1},\tr_2(C)=p_1I_{p_2}}. 
\end{align*}
 Note that if $r=p$, $\calC_{p_1,p_2}^{++}\equiv\tilde{\calC}_{p_1,p_2,r}^+ $.

We shall connect rank$-r$ cores to statistical applications, thereby motivating the study of rank$-r$ cores. Observe that the set $\calC_{p_1,p_2}^{++}$ is convex by (\ref{sec2.1.eq2}). Using this convexity, \cite{hoff2023} proposed a core shrinkage estimator ($\CSE$) that shrinks the sample core toward the unique separable core, $I_p$. However, if the population core exhibits a low-dimensional feature, the $\CSE$ can be subject to over-parameterization when $n<p$. Specifically, \cite{hoff2023} discussed a partial-isotropy structure as a possible structural assumption on $c(\Sigma)$, following the approach in factor analysis. Namely, $c(\Sigma)$ is represented as $BB^\top+\lambda I_p$ for some $B\in\real_*^{p\times r}$ and $\lambda>0$. By the construction of $\CSE$, its partial-isotropy rank is $n$ when $n<p$ (see Section $3.1$--$3.2$ of \cite{hoff2023}), which is typically larger than $r$. Thus, it may over-parameterize such a $c(\Sigma)$. Nevertheless, they did not pursue incorporating the partial-isotropy structure of $c(\Sigma)$ in estimation themselves due to a lack of understanding of such a core. By the linear system in (\ref{sec2.1.eq2}), there are constraints on $B$ and $\lambda$, compared to a usual partial-isotropy covariance. The following implies that a partial-isotropy rank$-r$ core is a non-trivial convex combination of a rank$-r$ core and a trivial core $I_p$, leading to the study of rank$-r$ cores. 

\begin{prop}\label{sec2.1.prop1}
For $C\in \calC_{p_1,p_2}^{++}$, suppose $C=BB^\top+\lambda I_p$ for some $B\in \real_*^{p\times r}$ and constant $\lambda>0$, where $p_1/p_2+p_2/p_1<r$. Then $\lambda\in(0,1)$ and $BB^\top=(1-\lambda)AA^\top$ for a rank$-r$ core $AA^\top$ with $A\in\real_{*}^{p\times r}$.
\end{prop}
\begin{proof}
By the linear system in (\ref{sec2.1.eq2}), $C$ should satisfy 
\begin{align*}
    \tr_1\parentheses{C}&=\tr_1\parentheses{BB^\top}+\lambda\tr_1(I_p)=p_2 I_{p_1} \Rightarrow\tr_1\parentheses{BB^\top}=p_2(1-\lambda)I_{p_1},\\
        \tr_2\parentheses{C}&=\tr_2\parentheses{BB^\top}+\lambda\tr_2(I_p)=p_1 I_{p_2} \Rightarrow\tr_2\parentheses{BB^\top}=p_1(1-\lambda)I_{p_2}.
\end{align*}
Since $BB^\top$ is positive semi-definite, so are its partial traces \cite{zhang2012,zhang2011}. Hence, we should have that $\lambda\leq 1$. Note that the linear system in (\ref{sec2.1.eq2}) implies that $\tr(C)=p$. Therefore, if $\lambda=1$, $\tr(BB^\top)=0$ so that $BB^\top=\bfzero_{p\times p}$, contradicting the assumption that $B\in \real_*^{p\times r}$. Thus, $\lambda<1$. Parameterizing $BB^\top$ by $(1-\lambda)AA^\top$ for some $A\in\real_*^{p\times r}$, we have that 
\begin{align*}
    \tr_1(AA^\top)=p_2I_{p_1}, \quad \tr_2(AA^\top)=p_1I_{p_2},
\end{align*}
implying that $AA^\top$ is a rank$-r$ core. 
\end{proof}
\begin{remark}
Regarding the condition on $r$ in Proposition \ref{sec2.1.prop1}, recall that it arises from the sample size threshold for which the Kronecker MLE exists. In fact, by Theorem $1.2$ of \cite{derksen2021}, the other scenarios on $(p_1,p_2,r)$ that admit the Kronecker MLE are either $p_1^2+p_2^2-rp_1p_2=0$, which is equivalent to $(p_1,p_2,r)=(p_1,p_1,2)$, or $p_1^2+p_2^2-rp_1p_2=d^2$ for $d=\text{gcd}(p_1,p_2)$. Assuming $p_1\geq p_2$ without loss of generality, the latter can be satisfied with $(p_1,p_2,r)=(p_2r,p_2,r)$ and $((k+1)m,km,2))$ for some $k,m\in\bbN$, for example (see \cite{sung2025} also). Compared to the regime where $p_1/p_2+p_2/p_1<r$, equivalently $p_1^2+p_2^2-rp_1p_2<0$, the other scenarios are highly restrictive, as generic $(p_1,p_2,r)$ do not satisfy them. On the other hand, the regime where $p_1/p_2+p_2/p_1<r$ applies to generic $(p_1,p_2,r)$ and allows freeness in the choice of $r$.     
\end{remark}

Because every $\Sigma\in\calS_{p_1,p_2}^{++}$ has a unique KCD, it is natural to question whether the same holds for every $\Omega\in \calS_{p,r}^+$ whenever $r>p_1/p_2+p_2/p_1$. The answer is negative, as illustrated by the example below.
\begin{example}\label{sec2.1.ex1}
Suppose $E=(E_{11},E_{12},E_{22})\in(\real^{2\times2})_*^3$ where each $E_{ij}$ has a $1$ in the $(i,j)-$th entry and $0$ elsewhere. With $F=\varphi_{2,2,3}(E)$, $FF^\top\in\calS_{4,3}^+$. However, $FF^\top$ does not admit a Kronecker MLE. The proof is deferred to Appendix \ref{append.A.1}.
\end{example}

The reason is that the threshold on $r$ for which the Kronecker MLE exists is understood in a generic (almost sure) sense \cite{derksen2021}. In a strict algebra sense, however, the Kronecker MLE may not exist for rank$-r$ $\Omega$ even if $r$ satisfies the threshold as seen above. Furthermore, unless $r=p$, $\tilde{\calC}_{p_1,p_2,r}^+$ may not be a smooth manifold. In Section \ref{sec4.1}, we show that the singularity preventing $\tilde{\calC}_{p_1,p_2,r}^+$ from being a manifold in view of Sard's theorem (\cite{lee2012}, Theorem $6.10$) corresponds to the set of canonically decomposable matrices. Here, $(A_1,\ldots,A_r)\in(\real^{p_1\times p_2})^r$ is canonically decomposable if there exists  $(P,Q)\in GL_{p_1}\times GL_{p_2}$ such that $PA_iQ^{-1}$ is of a non-trivial block-diagonal form for each $i\in[r]$. After removing this set from $\tilde{\calC}_{p_1,p_2,r}^+$, the remaining set $\calC_{p_1,p_2,r}^+$ is a smooth manifold, as shown in Section \ref{sec4.2}.  

\subsection{Riemannian manifolds}\label{sec2.2}
In this section, we briefly review some geometric properties of Riemannian manifolds. For details, we refer the reader to \cite{lee2012,lee2018,absil2008,boumal2023}. Suppose $(\calM,g)$ is a Riemannian manifold, where $\calM$ is a smooth manifold equipped with a Riemannian metric $g$. The Riemannian metric $g:T_x\calM\times T_x\calM\rightarrow \real$ defines an inner product on each tangent space $T_x\calM$, varying smoothly with $x\in\calM$. A smooth curve  $\gamma_x^v:[0,1]\rightarrow \calM$ emanating from $x\in \calM$ in the direction of $v\in T_x\calM$ is geodesic, i.e., a locally shortest curve with zero acceleration. Then the exponential map $\Exp_x:T_x\calM\rightarrow \calM$ is defined by $\Exp_x(v):=\gamma_x^v(1)$. Suppose $f$ is a smooth function on $\calM$. The Riemannian gradient $\grad f(x)$ of $f$ at $x\in\calM$ is the unique tangent vector in $T_x\calM$ satisfying that for any $v\in T_x\calM$, 
\begin{align*}
    g_x(\grad f(x),v)=D_vf(x)
\end{align*}
 where $D_vf(x)$ is a directional derivative of $f(x)$ along $v$. The Riemannian Hessian operator of $f$, denoted $\Hess f(x):T_x\calM\rightarrow T_x\calM$, is then defined to be a covariant derivative of the Riemannian gradient. By (5.35)--(5.36) of \cite{absil2008}, for a geodesic $\gamma_x^v$, 
 \begin{align*}
\frac{d^2}{dt^2}(f\circ \gamma_x^v)(t)\bigg|_{t=0}\equiv \Hess f(x)[v,v]=g_x(\Hess f(x)[v],v).
 \end{align*}
Since $\Hess f(x)[\cdot,\cdot]$ is a symmetric bilinear form on $T_x\calM$ (see (5.31) of \cite{absil2008}), the polarization identity implies that 
\begin{align}\label{sec2.2.eq1}
    g_x(\Hess f(x)[v],w)=\frac{ \Hess f(x)[v+w,v+w]- \Hess f(x)[v,v]- \Hess f(x)[w,w]}{2},
\end{align}
where the linear operator $\Hess f(x)[v]$ is identified as the unique tangent vector satisfying the above for any $w\in T_x\calM$. The smooth function $f$ on $\calM$ is {\it (strictly) geodesically convex} if the function $h=f\circ \gamma_x^v$ is (strictly) convex in usual sense for any geodesic $\gamma_x^v$ of non-zero speed. In Section \ref{sec2.2.1}--\ref{sec2.2.2}, we review the Riemannian geometry of $\calS_p^{++}$ and $\bbP(\calS_p^{++})$ (resp. $\calL_p^{++}$ and $\bbP(\calL_p^{++})$) under affine-invariant metric (resp. Cholesky metric), which are useful for Riemannian optimization in Section \ref{sec6}. 

\subsubsection{Riemannian geometry of $\calS_p^{++}$ and $\bbP(\calS_p^{++})$}\label{sec2.2.1}
We review the Riemannian geometry of $\calS_p^{++}$ and $\bbP(\calS_{p}^{++})$ under the affine-invariant metric $g^{\AI}$ \cite{skovgaard1984,pennec2006,moakher2005}. The tangent spaces of each manifold are given by
\begin{align*}
    &T_\Sigma \calS_p^{++}\equiv \calS_p, \quad T_\Sigma \bbP(\calS_p^{++})=\set{V\in\calS_p:\tr\parentheses{\Sigma^{-1}V}=0}.
\end{align*}
Thus, the dimensions of $\calS_p^{++}$ and $\bbP(\calS_p^{++})$ are $\binom{p+1}{2}$ and $\binom{p+1}{2}-1$, respectively. The affine-invariant metric $g^\AI$ on $\calS_p^{++}$ is defined as 
\begin{align*}
    g_{\Sigma}^{\AI}(U,V)=\tr\parentheses{\Sigma^{-1}U\Sigma^{-1}V}, \quad U,V\in T_\Sigma \calS_{p}^{++}.
\end{align*}
The geodesic, Riemannian gradient, and Riemannian Hessian operator of $(\calS_p^{++},g^{\AI})$ are given as follows;
\begin{itemize}[leftmargin=*]
    \item[]\textbf{(Geodesic)} Suppose $\Sigma\in\calS_p^{++}$ and $V\in T_\Sigma \calS_p^{++}$. Then the geodesic emanating from $\Sigma$ in the direction of $V$ is $\gamma_\Sigma^V:t\in[0,1]\rightarrow \Sigma^{1/2}\exp\parentheses{t\Sigma^{-1/2}V\Sigma^{-1/2}}\Sigma^{1/2}$ for a symmetric square root $\Sigma^{1/2}$ of $\Sigma$.
    \item[]\textbf{(Riemannian Gradient \& Hessian Operator)} Suppose $f$ is a smooth function over $\calS_p^{++}$. For $\Sigma\in\calS_p^{++}$ and $V\in T_\Sigma\calS_p^{++}$,
    \begin{align*}
        \grad f(\Sigma)&=\Sigma \nabla f(\Sigma)\Sigma,\quad \Hess f(\Sigma)[V]=\Sigma \nabla^2 f(\Sigma)[V]\Sigma+\sym\parentheses{V\nabla f(\Sigma)\Sigma}.
    \end{align*}
\end{itemize}
The smooth manifold $\bbP(\calS_{p}^{++})$ is  a totally geodesic submanifold of $\calS_p^{++}$ under $g^{\AI}$. Also, the orthogonal projection of $V\in T_\Sigma \calS_p^{++}$ onto $T_\Sigma \bbP(\calS_p^{++})$ (see ($32$) of \cite{simonis2025}) is given by 
\begin{align}\label{sec2.2.1.eq1}
    \calP_\Sigma(V):=V-\tr\parentheses{\Sigma^{-1}V}\Sigma/p.
\end{align}
Lastly, the Riemannian gradient and Hessian operator of a smooth function $f$ on  $\bbP(\calS_{p}^{++})$ are obtained as the orthogonal projections of those on $\calS_p^{++}$ as above, by smoothly extending $f$ to $\calS_p^{++}$.

\subsubsection{Riemannian geometry of $\calL_p^{++}$ and $\bbP(\calL_p^{++})$}\label{sec2.2.2}
We review the Riemannian geometry of $\calL_p^{++}$ and $\bbP(\calL_p^{++})$ under Choleksy metric $g^{\chol}$ \cite{lin2019}. The tangent spaces of each manifold are given by 
\begin{align*}
    &T_L \calL_p^{++}\equiv \calL_p, \quad T_L \bbP(\calL_p^{++})=\set{V\in\calL_p:\tr\parentheses{L^{-1}V}=0}.
\end{align*}
Note that the dimensions of $\calL_p^{++}$ and $\bbP(\calL_p^{++})$ are $\binom{p+1}{2}$ and $\binom{p+1}{2}-1$, respectively. Then the Cholesky metric $g^{\chol}$ on $\calL_p^{++}$ is defined as 
\begin{align*}
    g_L^{\chol}(U,V)=g^E(\floor{U},\floor{V})+g^E(\bbD(L)^{-2}\bbD(U),\bbD(V)),\quad U,V\in T_L\calL_p^{++}.
\end{align*}
where $g^E$ is the Euclidean metric. For $L\in \calL_{p}^{++}$ and the tangent vector  $V\in T_L\calL_p^{++}$, the geodesic is given by
\begin{align*}
 \gamma_L^V:t\in[0,1]\rightarrow \floor{L}+t\floor{V}+\bbD(L)\exp\parentheses{t\bbD(X)\bbD(L)^{-1}}\in \calL_p^{++}.  
\end{align*}
As an analogy to Section \ref{sec2.2.1}, $\bbP(\calL_p^{++})$ is a totally geodesic submanifold of $\calL_p^{++}$ under $g^{\chol}$. The formulas of Riemannian gradient and Hessian operator on $(\calL_p^{++},g^\chol)$ are provided below. 
\begin{prop}\label{sec2.2.2.prop1}
    Suppose $f:\calL_p^{++}\rightarrow \real$ is a smooth function. Given $L\in\calL_p^{++}$ and the tangent vector $V\in T_L \calL_p^{++}\equiv \calL_p$, the Riemannian gradient and Hessian operator of $f$ on $(\calL_p^{++},g^\chol)$ are given by
    \begin{align*}
        \grad f(L)&=\bbD(L)^2\bbD(\nabla f(L))+\floor{\nabla f(L)},\\
        \Hess f(L)[V]&=\bbD(L)^2\bbD(\nabla^2 f(L)[V])+\floor{\nabla^2 f(L)[V]}+\bbD(L)\bbD(\nabla f(L))\bbD(V).
    \end{align*}
\end{prop}
\begin{proof}
   See Appendix \ref{append.A.1} for the proof.
\end{proof}

Also, the orthogonal projection of $V\in T_L\calL_p^{++}$ onto $T_L\bbP(\calL_p^{++})$ can be derived as an analogy to (\ref{sec2.2.1.eq1}).

\begin{prop}\label{sec2.2.2.prop2}
Suppose $\calL_p^{++}$ is equipped with metric $g^\chol$. Let $L\in \calL_p^{++}$ and $V\in T_L\calL_p^{++}$. Then the operator $\calP_L:V\in T_L\calL_p^{++}\rightarrow V-\tr\parentheses{L^{-1}V}\bbD(L)/p\in T_L\bbP(\calL_p^{++})$ is an orthogonal projection.
\end{prop}
\begin{proof}
   See Appendix \ref{append.A.1} for the proof.
\end{proof}

It then directly follows that the Riemannian gradient and Hessian operator of a smooth function $f$ on $\bbP(\calL_p^{++})$ are obtained as the orthogonal projections of those on $\calL_p^{++}$ given in Proposition \ref{sec2.2.2.prop1}.

\subsection{Quotient manifold}\label{sec2.3}
In this section, we review the quotient geometry of a Riemannian manifold. We again refer to \cite{lee2012,lee2018} for the details. Suppose $(\calM,g)$ is a Riemannian manifold and $G$ is a Lie group acting smoothly, properly, and freely on $\calM$. The action $(g,x)\in G\times \calM\rightarrow g\cdot x\in\calM$ is smooth and proper if it is smooth and proper as a map. Note that the map $f:X\rightarrow Y$ between two topological spaces is proper if the preimage of every compact subset of $Y$ is also compact in $X$. Also, the action is free if there is no non-trivial action that fixes the elements of $\calM$, i.e., if $g\cdot x=x$, then $g$ is an identity $e$ for any $x\in\calM$. If the Lie group $G$ is also compact, e.g., $\calO_p$, then every smooth action of $G$ is proper (\cite{lee2012}, Corollary $21.6$). For any Lie group $G$ with a smooth, proper, and free action, there exists a unique smooth structure on $\calM^0=\calM/G$ such that the canonical projection $\pi:x\in \calM\rightarrow [x]\in\calM^0$ is a smooth submersion. Also, $\dim\calM^0=\dim\calM-\dim G$. With a submanifold $\calM_x:=\pi^{-1}([x])$ of $\calM$, the vertical space at $x$ is given as
\begin{align*}
    \calV_x\equiv T_x\calM_x=\text{ker\,}d\pi(x).
\end{align*}
The horizontal space $\calH_x$ is an orthogonal complement of $\calV_x$ in $T_x\calM$. For any $v\in T_{[x]}\calM^0$, a unique tangent vector $v_x^\#\in \calH_x$ such that $d\pi(x)[v_x^\#]=v$, referred to as a horizontal lift of $v$ at $x$. The quotient metric $g^0$ is defined as $g_x^0(v,w)=g_x(v_x^\#,w_x^\#)$, making $(\calM^0,g^0)$ a Riemannian manifold.

\subsection{Algebraic geometry}\label{sec2.4}
In this section, we briefly review the algebraic geometry, focusing on the ingredients necessary for proving that the set of canonically decomposable matrices is Zariski-closed and has a Lebesgue measure zero in Section \ref{sec4.1}. We shall refer to \cite{mumford1988,harris1992,hartshorne1977} for a more comprehensive review. For the subset $X\subset \real^n$, we say $X$ is {\it Zariski-closed} if $X$ is a zero locus of finitely many polynomials over the field $\real$. That is, for finitely many polynomials $p_1,\ldots,p_m$, 
\begin{align*}
    X=\set{(x_1,\ldots,x_n)\in\real^n:p_i(x_1,\ldots,x_n)=0,\; ^\forall i\in[m]}.
\end{align*}
Otherwise, $X$ is {\it Zariski-open}. Note that such a set $X$ is also referred to as an affine (resp. projective) algebraic set with the affine space $\real^n$ (resp. the projective space $\real\bbP^{n-1}$). For a projective algebraic set, the polynomials should be homogeneous, i.e., each term has the same degree. To define the {\it (topological) dimension} of any subset $X$ of $\real^n$, suppose $Y$ is a closed subset of $\real^n$. We say $Y$ is reducible if $Y$ is the union of two proper closed subsets $Y_1$ and $Y_2$. Otherwise, $Y$ is irreducible. Then the {\it (topological) dimension} of $X$ is defined to be the largest integer $d\in[n]$ such that there exists a chain $ Y_0\subsetneq Y_1\subsetneq \cdots \subsetneq Y_d\subset \bar{X}$, where each $Y_i$ is an irreducible closed subset of $\bar{X}$, the closure of $X$. Such $d$ always exists, and write $d:=\dim X$. We provide some useful facts about the topological dimension and Zariski-closed set to prove that the set of canonically decomposable matrices is proper Zariski-closed subset with measure zero, whose proofs are omitted (see Lemma $2.2$--$2.3$ and $2.7$ of \cite{domanov2015}).
\begin{lemma}\label{sec2.4.lemma1}
The followings are true:
\begin{itemize}
    \item If $X_1\subset X_2\subset\real^n$, $\dim X_1\leq \dim X_2\leq n$. Also, $\max_i\dim X_i\leq \dim (X_1\times X_2)$.
    \item For $X_1,\ldots,X_k\subset \real^n$, $\dim\parentheses{\cup_{i=1}^k X_i}=\max_i \dim X_i$.
  \item If $X$ is a Zariski-closed subset of $\real^n$, then $X$ is also closed under Euclidean topology. Also, a finite union of Zariski-closed sets is again Zariski-closed. 
  \item If $X$ is a proper subset of $\real^n$ with $\dim X<n$, $X$ has a Lebesgue measure zero. Thus, any proper Zariski closed set is a closed susbset (in Euclidean sense) with a measure zero.
\end{itemize}
\end{lemma}

Lastly, we define the affine variety, the projective variety, and the Grassmannian over the field $\real$. The affine (resp. projective) variety is an irreducible affine (resp. projective) algebraic subset. Note that a product of affine (resp. projective) varieties $V_1$ and $V_2$ is again an affine (resp. projective) variety \cite{sutherland2013}. The projective variety is known to be a complete variety; namely, the following is true for the projective variety \cite{sutherland2013,mumford1988,harris1992}. 
\begin{definition}\label{sec2.4.def1}
   The variety $X$ over a field $\real$ or $\mathbb{C}$ is complete if the projection morphism $\pi:X\times Y\rightarrow Y$ is closed for any variety $Y$. That is, if $U$ is a Zariski-closed subset of $X\times Y$, $\pi(U)$ is also Zariski-closed in $Y$. 
\end{definition}

The Grassmannian $\text{Gr}(d,n)$ is a collection of $d-$dimensional linear subspaces of $\real^n$, which is a manifold of dimension $d(n-d)$. While $\text{Gr}(d,n)$ can be realized as both affine and projective varieties \cite{devriendt2024}, we focus on its projective variety aspect. Note that every $d-$dimensional linear subspace $D$ of $\real^n$ can be represented as a $d\times n$ matrix $Z$ whose rows represent the basis of $D$. Via the Pl\"uker embedding \cite{michalek2021}, which realizes $\text{Gr}(d,n)$ as a projective variety, the matrix $Z$ can be identified with Pl\"uker coordinates and the system of the polynomials that these coordinates should satisfy, so-called the Pl\"uker equation \cite{michalek2021} (see Example $1.1$ of \cite{devriendt2024} for instance). Furthermore, if $\calP_Z$ is an orthogonal projection onto $R(Z)$, then each entry of $|ZZ^\top|\calP_Z$ is a quadratic polynomial in these Pl\"uker coordinates (\cite{devriendt2024}, Theorem $2.1$).   

\section{Smooth manifold $\calC_{p_1,p_2}^{++}$}\label{sec3}
In this section, we prove that $\calC_{p_1,p_2}^{++}$ is a compact, smooth, embedded submanifold of $\calS_p^{++}$. Throughout this and the next section, note that for $A=(A_1,\ldots,A_p)\in (\real^{p_1\times p_2})_*^r$, we write $A_R:=\sum_{i=1}^r A_iA_i^\top$ and $A_C:=\sum_{i=1}^r A_i^\top A_i$ for fixed $p_1/p_2+p_2/p_1<r \leq p$. We shall introduce the following sets and maps:
\begin{align}\label{sec3.eq1}
\begin{split}
        \calH_{p_1,p_2,r}&:=\set{A\in(\real^{p_1\times p_2})_*^r:\rank(A_R)=p_1,\rank(A_C)=p_2},\\
    \calD_{p_1,p_2,r}&:=F_{p_1,p_2,r}^{-1}(\set{(I_{p_1}/p_1,I_{p_2}/p_2,p)}),\; \calC_{p_1,p_2,r}:=\varphi_{p_1,p_2,r}(\calD_{p_1,p_2,r}),\\
    F_{p_1,p_2,r}&:A\in \calH_{p_1,p_2,r}\rightarrow (A_R/\tr(A_R),A_C/\tr(A_R),\tr(A_R))\in \calR_{p_1,p_2}, \\
    \calR_{p_1,p_2}&:=\bar{\calS}_{p_1}^{++}\times \bar{\calS}_{p_2}^{++}\times \real_+,\\
    s_{p_1,p_2,r}&:[B]\in \real_*^{p\times r}/\calO_r\rightarrow BB^\top \in\calS_{p,r}^{+},
\end{split}
\end{align}
Note that $\calR_{p_1,p_2}$ is a smooth manifold of dimension $\binom{p_1+1}{2}+\binom{p_2+1}{2}-1$. These notations will be used in Section \ref{sec4.2} also. For simplicity, since $r=p$ in this section, write $\calH_{p_1,p_2}:=\calH_{p_1,p_2,p}$, $\calD_{p_1,p_2}:=\calD_{p_1,p_2,p}$, $\calC_{p_1,p_2}:=\calC_{p_1,p_2,p}$, $\varphi_{p_1,p_2}:=\varphi_{p_1,p_2,p}$, $F_{p_1,p_2}:=F_{p_1,p_2,p}$ and $s_{p_1,p_2}:=s_{p_1,p_2,p}$. Observe that for any $A\in\calD_{p_1,p_2}$, if $\bar{A}=\varphi_{p_1,p_2}(A)$,
\begin{align*}
    \tr_1(\bar{A}\bar{A}^\top)=p_2I_{p_1},\quad \tr_2(\bar{A}\bar{A}^\top)=p_1I_{p_2},
\end{align*}
satisfying (\ref{sec2.1.eq2}). Therefore, $\calD_{p_1,p_2}$ is a key ingredient to construct $\calC_{p_1,p_2}^{++}$.  

We outline the proof strategy as follows. Although we state the strategy when $r=p$, note that this strategy can be straightforwardly extended to the rank-deficient case. Define the action of $\calO_p$ on $\real_*^{p\times p}$ by $(O,B)\in\calO_p\times \real_*^{p\times p} \rightarrow BO\in \real_*^{p\times p}$. Since the action is smooth, free, and proper as $\calO_p$ is a Lie compact group, $\real_*^{p\times p}/\calO_p$ is a quotient manifold. To show that $\calC_{p_1,p_2}^{++}$ is a smooth manifold, observe that if $\calD_{p_1,p_2}$ is a smooth submanifold embedded in the smooth manifold $\calH_{p_1,p_2}$, so is $\calC_{p_1,p_2}$ in $\real_*^{p\times p}$ as the map $\varphi_{p_1,p_2}$ is a diffeomorphism. Also, if $\calC_{p_1,p_2}$ is $\calO_p-$invariant with the action above, we can show that $\calC_{p_1,p_2}/\calO_{p_1,p_2}$ is embeddded in $\real_*^{p\times p}/\calO_p$. The result that $\calC_{p_1,p_2}^{++}$ is embedded in $\calS_{p_1,p_2}^{++}$ then  follows from the facts that the map $s_{p_1,p_2}$ is a diffeomorphism (\cite{massart2020}, Proposition $2.8$) and $\calC_{p_1,p_2}^{++}\equiv s_{p_1,p_2}(\calC_{p_1,p_2}/\calO_p)$. This strategy and the ancillary results below can be applied when $r<p$. The only difference is the way to show $\calD_{p_1,p_2,r}$ is a smooth manifold as shown in Section \ref{sec4.2}. Now taking $r=p$, we provide the ancillary results to prove the main result of this section. 

\begin{lemma}\label{sec3.lemma1}
The set $\calH_{p_1,p_2}$ is an open smooth submanifold of $(\real^{p_1\times p_2})_*^p$ with $\dim\calH_{p_1,p_2}=p^2$ and the tangent space $T_A\calH_{p_1,p_2}\equiv (\real^{p_1\times p_2})^p$. 
\end{lemma}
\begin{proof}
See Appendix \ref{append.A.2} for the proof.
\end{proof}

Lemma \ref{sec3.lemma1} ensures that $F_{p_1,p_2}$ is a smooth map between smooth manifolds. Next, we establish that $\calD_{p_1,p_2}$ is a closed and smooth submanifold embedded in $\real_*^{p\times p}$, which is a key ingredient to construct $\calC_{p_1,p_2}^{++}$. 

\begin{prop}\label{sec3.prop1}
    The level set $\calD_{p_1,p_2}:=F_{p_1,p_2}^{-1}\parentheses{\set{(I_{p_1}/p_1,I_{p_2}/p_2,p)}}$ is a closed, smooth, embedded submanifold of $\calH_{p_1,p_2}$ with dimension $p^2-\binom{p_1+1}{2}-\binom{p_2+1}{2}+1$.
\end{prop}
We provide a complete proof of Proposition \ref{sec3.prop1} in Appendix \ref{append.A.2}. Here we provide the main idea of the proof.

\begin{proof}[Sketch of Proof]
We use the constant-rank level set theorem (\cite{lee2012}, Theorem $5.12$) to prove the result. Take $B=(B_1,\ldots,B_p)\in T_A\calH_{p_1,p_2}$. Let $a=[\text{vec}(A_1)^\top,\ldots,\text{vec}(A_p)^\top]^\top$, and $b=[\text{vec}(B_1)^\top,\ldots,\text{vec}(B_p)^\top]^\top$. Since $\tr(A_R)=p$ for any $A\in\calD_{p_1,p_2}$, the differential of $F_{p_1,p_2}$ at $A$ is given by
\begin{align*}
    dF_{p_1,p_2}(A)[B]=\left(\frac{1}{p}\sum_{i=1}^p(A_iB_i^\top+B_iA_i^\top)-\frac{2\tr\parentheses{\sum_{i=1}^p A_iB_i^\top}}{p_1^2p_2}I_{p_1},\right.\\
    \left. \frac{1}{p}\sum_{i=1}^p (A_i^\top B_i+B_i^\top A_i) -\frac{2\tr\parentheses{\sum_{i=1}^p A_iB_i^\top}}{p_1p_2^2}I_{p_2},  2a^\top b \right)
\end{align*}
for any $A\in \calD_{p_1,p_2}$. Using vec-Kronecker identity, the value of $dF_{p_1,p_2}(A)[B]$ can be equivalently identified as 
\begin{align}\label{sec3.prop1.eq1}
\left[\begin{array}{c}
J_1\\
J_2\\
J_3 
\end{array}\right]b:=\underbrace{\left[\begin{array}{c}
\frac{1}{p}(I_{p_1^2}+K_{(p_1,p_1)})[A_1\otimes I_{p_1},\ldots,A_p\otimes I_{p_1}]-\frac{2}{p_1^2p_2}\text{vec}(I_{p_1})a^\top \\
\frac{1}{p}(I_{p_2^2}+K_{(p_2,p_2)})[I_{p_2}\otimes A_1^\top,\ldots,I_{p_2}\otimes A_p^\top]-\frac{2}{p_1p_2^2}\text{vec}(I_{p_2})a^\top\\
2a^\top 
\end{array}\right]}_{J:=J(A)} b.
\end{align} 
Hence, the dimension of the image of $dF_{p_1,p_2}(A)$ as a linear operator over $T_A\calH_{p_1,p_2}$ is equivalent to the rank of $J$. To compute the rank of $J$, note that 
\begin{align*}
        \text{rank}(J)=\dim C(J^\top)&=\dim C(J_1^\top)+\dim C(J_2^\top)+\dim C(J_3^\top)-\dim C(J_1^\top) \cap  C(J_2^\top)\\
    &-\dim C(J_2^\top) \cap  C(J_3^\top)-\dim C(J_1^\top) \cap  C(J_3^\top)\\
    &+\dim C(J_1^\top) \cap  C(J_2^\top)\cap C(J_3^\top).
\end{align*}
We claim in Appendix \ref{append.A.2} that 
\begin{align}\label{sec3.prop1.eq2}
\begin{split}
     &\dim C(J_1^\top)=\binom{p_1+1}{2}-1,\, \dim C(J_2^\top)=\binom{p_2+1}{2}-1,\, \dim C(J_3^\top)=1,\\
        &\dim C(J_1^\top) \cap  C(J_2^\top)=\dim C(J_2^\top) \cap  C(J_3^\top)=\dim C(J_1^\top) \cap  C(J_3^\top)\\
        =&\dim C(J_1^\top) \cap  C(J_2^\top)\cap C(J_3^\top)=0.
\end{split}
\end{align}
This implies that $\rank(J)=\dim \calR_{p_1,p_2}=\binom{p_1+1}{2}+\binom{p_2+1}{2}-1$. Since this holds for any $A\in \calD_{p_1,p_2}$, the constant-rank level set theorem implies that $F_{p_1,p_2}$ is a submersion on $\calD_{p_1,p_2}$ and $\calD_{p_1,p_2}$ is a smooth embedded submanifold of $\calH_{p_1,p_2}$ with a dimension 
\begin{align*}
    \dim \calH_{p_1,p_2}-\dim\calR_{p_1,p_2}=p^2-\binom{p_1+1}{2}-\binom{p_2+1}{2}+1.
\end{align*}

\end{proof}

By Proposition \ref{sec3.prop1}, the image of $\calD_{p_1,p_2}$ by the diffeomorphism $\varphi_{p_1,p_2}$, $\calC_{p_1,p_2}$, is closed and embedded in $\real_*^{p\times p}$. As discussed above, we show that the smooth manifold $\calC_{p_1,p_2}$ is $\calO_p-$invariant, and closed and embedded in $\real_*^{p\times p}/\calO_p$. 

\begin{lemma}\label{sec3.lemma2}
   For any $X\in \calC_{p_1,p_2}$ and $O\in\calO_p$, $XO\in \calC_{p_1,p_2}$. Hence, the action $(O,X)\in\calO_p\times \calC_{p_1,p_2}\rightarrow XO\in\calC_{p_1,p_2}$ is well-defined, smooth, and free.
\end{lemma}
\begin{proof}
See Appendix \ref{append.A.2} for the proof.
\end{proof}

\begin{lemma}\label{sec3.lemma3}
      Suppose $G$ is a compact Lie group acting smoothly and freely on a smooth manifold $\calM$. Assume that a smooth manifold $\calN$ is embedded in $\calM$ and $G-$invariant. Then $\calN/G$ is a smooth, embedded submanifold of $\calM/G$. 
\end{lemma}
\begin{proof}
See Appendix \ref{append.A.2} for the proof.
\end{proof}

With the ingredients above, we are ready to prove the main result of this section. 
\begin{thm}\label{sec3.thm1}
    The set $\calC_{p_1,p_2}^{++}$ is a compact, smooth, embedded submanifold of $\calS_p^{++}$ with a dimension $\binom{p+1}{2}-\binom{p_1+1}{2}-\binom{p_2+1}{2}+1$.
\end{thm}
\begin{proof}
      The compactness follows as (\ref{sec2.1.eq2}) implies that $\tr(C)=p$ for any $C\in\calC_{p_1,p_2}^{++}$. To show that $\calC_{p_1,p_2}^{++}$ is a smooth submanifold embedded in $\calS_{p}^{++}$, note that $\calD_{p_1,p_2}$ is embedded in $\calH_{p_1,p_2}$ by Proposition \ref{sec3.prop1}, and $\calH_{p_1,p_2}$ is also embedded in $(\real^{p_1\times p_2})_*^p$ as an open submanifold. Thus, $\calD_{p_1,p_2}$ is embedded in $(\real^{p_1\times p_2})_*^p$. Hence, $\calC_{p_1,p_2}\equiv \varphi_{p_1,p_2}(\calD_{p_1,p_2})$ is embedded in $\real_*^{p\times p}\equiv\varphi_{p_1,p_2}( (\real^{p_1\times p_2})_*^p)$ as the map $\varphi_{p_1,p_2}$ is a diffeomorphism. By Lemma \ref{sec3.lemma2}, $\calC_{p_1,p_2}$ is $\calO_p-$invariant. Thus, taking $\calM=\real_*^{p\times p}$, $\calN=\calC_{p_1,p_2}$, and $G=\calO_p$ in Lemma \ref{sec3.lemma3}, we have that $\calC_{p_1,p_2}/\calO_p$ is embedded in $\real_*^{p\times p}/\calO_p$. Also, the quotient manifold theorem implies that 
\begin{align*}
    \dim \calC_{p_1,p_2}/\calO_p=\dim\calC_{p_1,p_2}-\dim\calO_p=\binom{p+1}{2}-\binom{p_1+1}{2}-\binom{p_2+1}{2}+1.
\end{align*}
By Lemma \ref{sec3.lemma2} and (\ref{sec2.1.eq2}), we have that $AA^\top=BB^\top\in\calC_{p_1,p_2}^{++}$ for $A,B\in\real_*^{p\times p}$ if and only if $A,B\in\calC_{p_1,p_2}$ and $A=BO$ for some $O\in\calO_p$. Because the map $s_{p_1,p_2}$ defined in (\ref{sec3.eq1}) is a diffeomorphism by Proposition $2.8$ of \cite{massart2020}, $\calC_{p_1,p_2}^{++}\equiv s_{p_1,p_2}(\calC_{p_1,p_2}/\calO_p)$ is a smooth submanifold embedded in $\calS_{p}^{++}\equiv s_{p_1,p_2}(\real_*^{p\times p}/\calO_p)$.
\end{proof}

Since $\calC_{p_1,p_2}^{++}$ is a smooth manifold, we shall identify its tangent space. 

\begin{prop}\label{sec3.prop2}
   For $C\in\calC_{p_1,p_2}^{++}$, the tangent space of $\calC_{p_1,p_2}^{++}$ at $C$ is given by 
   \begin{align*}
      T_C\calC_{p_1,p_2}^{++}\equiv \set{W\in\calS_p:\tr_1(W)=\mathbf{0}_{p_1\times p_1},\tr_2(W)=\mathbf{0}_{p_2\times p_2}}.   
   \end{align*}
\end{prop}
\begin{proof}
See Appendix \ref{append.A.2} for the proof.
\end{proof}

\section{Smooth manifold $\calC_{p_1,p_2,r}^+$}\label{sec4}

\subsection{Canonically decomposable matrices}\label{sec4.1}
In this section, we review the notion of canonical decomposability of $A=(A_1,\ldots,A_r)\in (\real^{p_1\times p_2})^r$, and justify removing the set of such matrices from $\calH_{p_1,p_2,r}$ to construct $\calC_{p_1,p_2,r}^+$ for $\calH_{p_1,p_2,r}$ defined in (\ref{sec3.eq1}). We first give its definition below.    
\begin{definition}\label{sec4.1.def1}
    Suppose $A=(A_1,\ldots,A_r)\in (\real^{p_1\times p_2})^r$. We say $A$ is canonically decomposable if there exists a $(P,Q)\in GL_{p_1}\times GL_{p_2}$ such that, for each $i\in[r]$, $PA_iQ^{-1}$ is of a non-trivial block-diagonal form, i.e., $PA_iQ^{-1}=\oplus_{j=1}^2 A_{ij}$, where $A_{i1}\in \real^{a\times b}$ and $A_{i2}\in\real^{(p_1-a)\times (p_2-b)}$ for some $1\leq a\leq p_1-1$ and $1\leq b\leq p_2-1$. Otherwise, $A$ is canonically indecomposable.   
\end{definition}

As an example of canonical decomposability, for generic element $(A_1,A_2)$ in $(\real^{4\times 7})^2$, there exists a $(P,Q)\in GL_{4}\times GL_{7}$ such that $PA_iQ^{-1}=\oplus_{j=1}^3 B_{ij}$, where $B_{i1},B_{i2}\in \real^{1\times 2}$ and $B_{i3}\in \real^{2\times 3}$ (\cite{derksen2021}, Example $2.8$). Here the term generic should be understood as an almost sure sense. An example with specific values of $A_i$'s and $(P,Q)$ is provided in Example $4$ of \cite{maehara2011}.

The notion of canonical decomposability is mainly motivated by Kronecker quiver representation and its applications in the analysis of the sample size threshold for the existence of Kronecker MLE \cite{derksen2021} (see \cite{kac1980,kac1982} also). To say informally, the $n-$Kronecker quiver $Q$ is a directed acyclic graph of two vertices $x$ and $y$ with $n$ arrows. Then, a representation of $Q$ is to assign a finite-dimensional vector space to each vertex. If this representation cannot be written as a direct sum of a non-trivial subrepresentation, such a representation is referred to as a $\sigma-$stable representation. In the context of the Kronecker MLE problem, these vector spaces correspond to $\real^{p_2}$ and $\real^{p_1}$, and the arrows correspond to the $n$ data matrices. Using this Kronecker-quiver representation, along with the group-invariant theory, \cite{derksen2021} characterized the scenarios of $(p_1,p_2,r)$ for which Kronecker MLE exists (see their Theorem $1.2$). For generic $(p_1,p_2)$, it turns out that $r>p_1/p_2+p_2/p_1$. Under this threshold, the uniqueness of Kronecker MLE also follows if it exists. 

To interpret this threshold, the canonical decomposability of $r$ data matrices in Definition \ref{sec4.1.def1} corresponds to whether the $r$ data matrices induce a $\sigma-$stable Kronecker quiver representation as arrows (see Section $3$ and $5$ of \cite{derksen2021}). Then the threshold on $r$ for which Kronecker MLE exists comes from the minimum number of arrows for which the factors in $\real^{p_1}$ and $\real^{p_2}$ are well-connected. We shall formally formulate this below. We emphasize that this decomposability notion also appears in other works on the sample size threshold analysis for Kronecker MLE. For example, \cite{soloveychik2016} studied the threshold by analyzing the contribution of each block in the canonical decomposition of data matrices to the growth of the objective function $d$ in (\ref{sec2.1.eq1}) (see Section $6.6$ of \cite{soloveychik2016}). They referred canonically decomposable data matrices to as {\it bad} samples.

In these works, the sample size threshold should be understood in a generic (almost sure) sense. If $r \geq p$ and $r$ data matrices are linearly independent, then the Kronecker MLE always uniquely exists \cite{srivastava2008}, not just generically. However, in a strict algebra sense, even if $r>p_1/p_2+p_2/p_1$ and the linear independence holds for data matrices, the Kronecker MLE may not exist as observed in Example \ref{sec2.1.ex1}. Nevertheless, note that the data matrices in that example are canonically decomposable. Thus, a natural question one could raise is whether any canonically decomposable matrices never admit the Kronecker MLE. It turns out that the answer is no, as shown in Example \ref{sec4.1.ex1}. This example also suggests that the canonically decomposable matrices are the singularities that may prevent the set of rank-$r$ cores from being a smooth manifold, in light of Sard's theorem. 

\begin{example}\label{sec4.1.ex1}
    Take $p_1=p_2=r=3$. Consider the subset
    \begin{align*}
        \calU:=\set{(I_3,Y_1,Y_2):Y_1=Q\oplus [1],Y_2=Q^\top\oplus [1],Q\in \calO_2,Q\neq \pm I_2}.
    \end{align*}
It is obvious that every $(A_1,A_2,A_3)\in\calU$ satisfy (\ref{sec2.1.eq2}), thereby inducing the Kronecker MLE $I_p$, and $\calU\subset \calH_{3,3,3}$. Also, this set is clearly canonically decomposable. However, the map $F_{3,3,3}$ defined in (\ref{sec3.eq1}) is not a submersion on $\calU$. The proof is deferred to Appendix \ref{append.A.4}.
\end{example}

Denote the subset of canonically decomposable matrices in $(\real^{p_1\times p_2})^r$ by $\calV_{p_1,p_2,r}$. We show that this set is closed and has a Lebesgue measure zero. Note that its analogous results have been proven based on group-invariant theory and representation-theoretic approaches (see Proposition $3.19$ of \cite{hoskins2012}, and Lemma $2.16$ and Section $5$ of \cite{derksen2021}). However, we provide a proof using a more direct language of algebraic geometry to make the article self-contained and better motivate the canonical decomposability in studying the fixed-rank core covariance manifold.    

\begin{lemma}\label{sec4.1.lemma1}
    Define a map $m:[1,\alpha-1]\times[1,\beta-1]\rightarrow \real_+$ by 
\begin{align*}
    m(a,b):=a(\alpha-a)+b(\beta-b)+r\parentheses{ab+(\alpha-a)(\beta-b)}
\end{align*}
for some fixed $\alpha,\beta\geq 2$ and $r>\alpha/\beta+\beta/\alpha$. Then the maximum of $m$ is strictly smaller than $r\alpha\beta $.
\end{lemma}
\begin{proof}
    See Appendix \ref{append.A.3} for the proof. 
\end{proof}

\begin{prop}\label{sec4.1.prop1}
    Define a subset $\calV_{p_1,p_2,r}\subset (\real^{p_1\times p_2})^r$ consisting of canonically decomposable $A\in (\real^{p_1\times p_2})^r$. Then the set $\calV_{p_1,p_2,r}$ is Zarisi-closed and thus closed in Euclidean sense. Furthermore, if $p_1/p_2+p_2/p_1<r\leq p_1p_2$, the dimension of $\calV_{p_1,p_2,r}$ is strictly smaller than $p_1p_2r$. Thus, $\calV_{p_1,p_2,r}$ is closed in Euclidean sense and has a Lebesgue measure zero.
\end{prop}
\begin{proof}
   See Appendix \ref{append.A.3} for the proof. 
\end{proof}

Now we mathematically formulate how the canonically indecomposability induces the connectivity between the row and column variables illustrated above. Note that this is crucial in concluding that the set of rank$-r$ cores is indeed a smooth manifold in Section \ref{sec4.2}. To this end, we give the definition of an undirected bipartite graph and provide its mathematical formulation via the canonically indecomposability.

\begin{definition}\label{sec4.1.def2}
    Suppose $G=(V,E)$ is an undirected graph with a vertex set $V$ and an edge set $E$. The graph $G$ is connected if there is a path between any two vertices in $G$, otherwise disconnected. Also, the graph $G$ is bipartite if $V$ can be partitioned into two disjoint and nonempty sets $V_1$ and $V_2$ such that every edge of $G$ connects a vertex in $V_1$ to one in $V_2$. Hence, a vertex in $V_1$ can be reached from the other vertex in $V_2$ only after alternating between $V_1$ and $V_2$, provided that there is a path.  
\end{definition}

A standard fact on the disconnected graph is that any such graph can be decomposed into connected components, which are maximally connected subgraphs. Then the result of the connectivity of the undirected bipartite graph induced by canonically indecomposable matrices is immediate.   

\begin{prop}\label{sec4.1.prop2}
Suppose $A=(A_1,\ldots,A_r)\in (\real^{p_1\times p_2})^r$ is canonically indecomposable. Take $(P,Q)\in GL_{p_1}\times GL_{p_2}$. Define a bipartite undirected graph $G_{A,P,Q}:=(\set{s_j:j\in[p_1]}\sqcup \set{q_k:k\in[p_2]},E)$, where $s_j$ is connected to $q_k$ if and only if there exists $i\in[r]$ such that $(PA_iQ^{-1})_{jk}\neq 0$. Then $G_{A,P,Q}$ is connected. 
\end{prop}
\begin{proof}
Suppose otherwise. Then there exists indecomopsable $A$ and $(P,Q)\in GL_{p_1}\times GL_{p_2}$ such that the graph $G_{A,P,Q}$ is disconnected. Hence, there exist partitions $U_1$ and $U_2$ of $\set{s_i}$ and accordingly $W_1$ and $W_2$ of $\set{q_j}$ such that a vertex in $U_1$ (resp. $U_2$) is never connected to $W_2$ (resp. $W_1$). After arranging the row and columns of $PA_iQ^{-1}$, we can obtain $(P',Q')\in GL_{p_1}\times GL_{p_2}$ such that $P'A_i(Q')^{-1}$ is of non-trivial block-diagonal form where the zero entries correspond to the absence of edges between $U_1$ (resp. $U_2$) and $W_2$ (resp. $W_1$), contradicting the indecomposability of $A$. 
\end{proof}

\subsection{Proof of the smooth manifold $\calC_{p_1,p_2,r}^+$}\label{sec4.2}
Using the ingredients developed in Section \ref{sec4.1}, together with analogies to ancillary results in Section \ref{sec3}, we prove that $\calC_{p_1,p_2,r}^+$ is a compact smooth submanifold embedded in $\calS_{p,r}^+$. To this end, recall the notations in (\ref{sec3.eq1}). Following the discussion and results in Section \ref{sec4.1}, we shall rewrite $\calH_{p_1,p_2,r}:=\calH_{p_1,p_2,r}\setminus \calV_{p_1,p_2,r}$, and the rest of the notations in (\ref{sec3.eq1}) are built upon this $\calH_{p_1,p_2,r}$. By Proposition \ref{sec4.1.prop1} and a version of Lemma \ref{sec3.lemma1}, $\calH_{p_1,p_2,r}$ is open in $(\real^{p_1\times p_2})_*^r$ and thus has the tangent space $(\real^{p_1\times p_2})^r$. Also, as an analogy to (\ref{sec3.prop1.eq1}), we define the following matrix-valued linear operator $J$ on $\calH_{p_1,p_2,r}$ by
\begin{align}\label{sec4.2.eq1}
    J(A):=\left[\begin{array}{c}
    J_1(A)\\
    J_2(A)\\
    J_3(A)
    \end{array}\right]=\left[\begin{array}{c}
\frac{1}{p}(I_{p_1^2}+K_{(p_1,p_1)})[A_r\otimes I_{p_1},\ldots,A_r\otimes I_{p_1}]-\frac{2}{p_1^2p_2}\text{vec}(I_{p_1})a^\top \\
\frac{1}{p}(I_{p_2^2}+K_{(p_2,p_2)})[I_{p_2}\otimes A_r^\top,\ldots,I_{p_2}\otimes A_r^\top]-\frac{2}{p_1p_2^2}\text{vec}(I_{p_2})a^\top\\
2a^\top 
\end{array}\right],
\end{align}
where $a=[\text{vec}(A_1)^\top,\ldots, \text{vec}(A_r)^\top ]^\top$. The proof strategy to establish the main result of this section is similar to that in Section \ref{sec3}. The ancillary results to prove Theorem \ref{sec3.thm1} can be established similarly when $r<p$. A slight difference lies in establishing the analogy of Proposition \ref{sec3.prop1}; namely, that  $\calD_{p_1,p_2,r}$ is a smooth, closed, and embedded submanifold of $(\real^{p_1\times p_2})_*^r$. Following the analogy to Proposition \ref{sec3.prop1}, we show that $\rank(J(A))=\dim \calR_{p_1,p_2}$ for any $A\in\calD_{p_1,p_2,r}$ so that the constant-rank level set theorem applies. As in the proof of Proposition \ref{sec3.prop1}, it suffices to verify (\ref{sec3.prop1.eq2}), with $J_i:=J_i(A)$ for $i=1,2,3$. Now the difference arises in the way showing that $\dim C(J_1^\top)\cap C(J_2^\top)=0$, which relies on the connectivity between row and column variables by canonically indecomopsable matrices in Proposition \ref{sec4.1.prop2}. 

Using the proof strategy outlined above, we state the result that $\calC_{p_1,p_2,r}^{+}$ is compact and embedded manifold in $\calS_{p,r}^+$ and provide a sketch of proof. The complete proof is deferred to Appendix \ref{append.A.4}.

\begin{thm}\label{sec4.2.thm1}
Recall the sets and maps in (\ref{sec3.eq1}). With $\calH_{p_1,p_2,r}$ defined above, the followings are true:
\begin{itemize}
    \item A smooth submanifold $\calD_{p_1,p_2,r}$ is closed and embedded in $(\real^{p_1\times p_2})_*^r$ with a dimension $p_1p_2r-\binom{p_1+1}{2}-\binom{p_2+1}{2}+1$.
    \item A smooth submanifold $\calC_{p_1,p_2,r}$ is closed and embedded in $\real_*^{p\times r}$ with a dimension $p_1p_2r-\binom{p_1+1}{2}-\binom{p_2+1}{2}+1$. 
    \item A smooth submanifold $\calC_{p_1,p_2,r}^{+}$ is compact and embedded in $\calS_{p,r}^+$ wtih a dimension $p_1p_2r-\binom{r}{2}-\binom{p_1+1}{2}-\binom{p_2+1}{2}+1$. 
\end{itemize}
\end{thm}
\begin{proof}[Sketch of Proof]
 Provided that the first item is true, the last two items directly follow from the argument for the proof of Theorem \ref{sec3.thm1}. The results of Lemma \ref{sec3.lemma1}, \ref{sec3.lemma2}, and Proposition \ref{sec3.prop1} can be developed similarly. For Lemma \ref{sec3.lemma2} with $\calC_{p_1,p_2,r}$ and $\calO_r$ instead of $\calC_{p_1,p_2}$ and $\calO_p$, respectively, we additionally show that $\tilde{X}=\varphi_{p_1,p_2,r}^{-1}(\varphi_{p_1,p_2,r}(X)O)$ is canonically indecomposable for any $X=(X_1,\ldots,X_r)\in \calD_{p_1,p_2,r}$ and $O\in\calO_r$ so that the action $(O,B)\in \calO_r\times \calC_{p_1,p_2,r}\rightarrow BO\in\calC_{p_1,p_2,r}$ is well-defined.  
 
To prove the first item, recall the operator $J$ in (\ref{sec4.2.eq1}). Following the proof of Proposition \ref{sec3.prop1}, the first item can be concluded if $C(J_1(A)^\top)\cap C(J_2(A)^\top)=\set{\bfzero_{pr}}$ for any $A\in\calD_{p_1,p_2,r}$. This can be done similarly to the proof of Proposition \ref{sec3.prop1}, along with the result of Proposition \ref{sec4.1.prop2}. 
\end{proof}

The tangent spaces of $\calC_{p_1,p_2,r}$ and $\calC_{p_1,p_2,r}^+$ follow from the proof of Theorem \ref{sec4.2.thm1}.

\begin{prop}\label{sec4.2.prop1}
    Let $A\in\calD_{p_1,p_2,r}$ and suppose $\tilde{A}=\varphi_{p_1,p_2,r}(A)$. Then 
\begin{align*}
    T_{\tilde{A}}\calC_{p_1,p_2,r}&\equiv \set{B\in\real^{p\times r}:\text{vec}(B)\in N(J(A))},\quad 
    T_{\tilde{A}\tilde{A}^\top}\calC_{p_1,p_2,r}^+\equiv \set{\tilde{A}B^\top+B\tilde{A}^\top:B\in  T_{\tilde{A}}\calC_{p_1,p_2,r}}.
\end{align*}
\end{prop}
\begin{proof}
See Appendix \ref{append.A.4} for the proof.
\end{proof}

\section{Differential geometry of $\calC_{p_1,p_2}^{++}$, $\calC_{p_1,p_2,r}$, and $\calC_{p_1,p_2,r}/\calO_r$}\label{sec5}

\subsection{Diffeomorphic relationship between $\calS_p^{++}$ and $\calS_{p_1,p_2}^{++}\times \calC_{p_1,p_2}^{++}$}\label{sec5.1}
In this section, we prove that $\calS_{p}^{++}$ is diffeomorphic to the product manifold $\calS_{p_1,p_2}^{++}\times \calC_{p_1,p_2}^{++}$ via the map $f:\Sigma\in \calS_{p}^{++}\rightarrow (k(\Sigma),c(\Sigma))\in\calS_{p_1,p_2}^{++}\times \calC_{p_1,p_2}^{++} $ with its inverse $g:(K,C)\in\calS_{p_1,p_2}^{++}\times \calC_{p_1,p_2}^{++}\rightarrow h(K)Ch(K)^\top\in\calS_p^{++}$. Consequently, we provide a new insight into the smooth structure of $\calS_p^{++}$ in terms of the separability. This generalizes the result of Proposition $5$ from \cite{hoff2023} on the homeomorphic relationship between $\calS_{p}^{++}$ and $\calS_{p_1,p_2}^{++}\times \calC_{p_1,p_2}^{++}$. We also calculate the differentials of $f$ and $g$ to examine how the tangent vectors transform via the maps $f$ and $g$.

To prove the diffeomorphic relationship, note that for either choice of the square root $h\in\calS_{p_1,p_2}^{++}$ or $h\in\calL_{p_1,p_2}^{++}$, the map $h$ is smooth. Hence, it is clear that the map $g$ is also smooth. Thus, if the maps $k$ and $c$ are smooth, then we are done as $f$ is also smooth then. To compute the differentials of $h$, $k$ and $c$, and thus $f$ and $g$, note that 
\begin{align*}
    T_{\Sigma_2\otimes \Sigma_1}\calS_{p_1,p_2}^{++}&\equiv \set{U_2\otimes \Sigma_1+\Sigma_2\otimes U_1:U_i\in\calS_{p_i}},\\
     T_{L_2\otimes L_1}\calL_{p_1,p_2}^{++}&\equiv \set{V_2\otimes L_1+L_2\otimes V_1:V_i\in\calL_{p_i}}.
\end{align*}
We provide an ancillary result below.

\begin{lemma}\label{sec5.1.lemma1}
Suppose $K=\Sigma_2\otimes \Sigma_1\in\calS_{p_1,p_2}^{++}$. Let $\Gamma_i\Lambda_i\Gamma_i^\top$ be the eigendecomposition of $\Sigma_i$, where $\Gamma_i\in\calO_{p_i}$ and $\Lambda_i$ is a diagonal matrix with the eigenvalues on its diagonal. Let $L_i=\calL(\Sigma_i)$, and take $U_i\in\calS_{p_i}$ to form $U=U_2\otimes \Sigma_1+\Sigma_2\otimes U_1$. Then the differential of the square root map $h$ is given as follows: if $h\in\calL_{p_1,p_2}^{++}$,
\begin{align*}
   dh(K)[U]=  (L_2\otimes L_1)\parentheses{I_{p_2}\otimes L_1^{-1}U_1L_1^{-\top}+L_2^{-1}U_2L_2^{-\top}\otimes I_{p_1} }_{\frac{1}{2}},
\end{align*}
and if $h\in \calS_{p_1,p_2}^{++}$,
\begin{align*}
    dh(K)[U]=     (\Gamma_2\otimes \Gamma_1)\left[\Lambda^{-}\circ \parentheses{\Lambda_2\otimes \Gamma_1^\top U_1\Gamma_1+\Gamma_2^\top U_2\Gamma_2\otimes \Lambda_1}\right](\Gamma_2\otimes \Gamma_1)^\top.
\end{align*}
Here $\Lambda^{-}$ is an elementrywise inverse of $\Lambda=\mathbf{1}_{p_2}\lambda_2^\top \otimes \mathbf{1}_{p_1}\lambda_1^\top+\lambda_2\mathbf{1}_{p_2}^\top\otimes\lambda_1\mathbf{1}_{p_1}^\top $ for $\lambda_1=\text{vec}(\Lambda_1^{1/2})$ and $\lambda_2=\text{vec}(\Lambda_2^{1/2})$, and $\circ$ denotes the Hadmard product. 
\end{lemma}
\begin{proof}
    See Appendix \ref{append.A.5} for the proof. 
\end{proof}

The differential of the map $g$ directly follows from the above lemma.
\begin{prop}\label{sec5.1.prop1}
Given $\Sigma\in \calS_{p}^{++}$, let $K=k(\Sigma)=\Sigma_2\otimes \Sigma_1$ and $C=c(\Sigma)$. Suppose $U\in T_{\Sigma_2\otimes \Sigma_1}\calS_{p_1,p_2}^{++}$ and $W\in T_C\calC_{p_1,p_2}^{++}$. For the map $g:(K,C)\in \calS_{p_1,p_2}^{++}\times \calC_{p_1,p_2}^{++}\rightarrow h(K)Ch(K)^\top\in\calS_p^{++}$, the differential is given by
\begin{align*}
    dg(K,C)[U,W]=h(K)Wh(K)^\top+(dh(K)[U])Ch(K)^\top+h(K)C (dh(K)[U])^\top,
\end{align*}
where  $dh(K)[U]$ is given in Lemma \ref{sec5.1.lemma1}. 
\end{prop}
\begin{proof}
    See Appendix \ref{append.A.5} for the proof.
\end{proof}

It remains to show that the maps $k$ and $c$ are smooth, proving that the map $f$ is smooth so that the diffeomorphic relationship holds, and compute their differentials. In some sense, the ambiguity due to a constant factor in identifying the factors of the elements in $\calS_{p_1,p_2}^{++}$ makes the proof complicated, i.e, $\Sigma_2\otimes \Sigma_1=(c\Sigma_2)\otimes (\Sigma_1/c)$ for any $c>0$. To avoid this ambiguity, we introduce the orthogonal parameterization of $\calS_{p_1,p_2}^{++}$ under $g^{\AI}$ \cite{mccormack2025,simonis2025,cox1987}. Specifically, suppose $\calE:=\bbP(\calS_{p_1}^{++})\times \calS_{p_2}^{++}$. Define a diffeomorphism 
\begin{align*}
    \psi_{p_1,p_2}:\Sigma_2\otimes \Sigma_1\in\calS_{p_1,p_2}^{++}\rightarrow (\Sigma_1,\Sigma_2)\in \calE
\end{align*}
where $|\Sigma_1|=1$. As studied by \cite{mccormack2025,simonis2025}, if $\calS_{p_1,p_2}^{++}$ is equipped with $g^{\AI}$, the induced metric $\tilde{g}_{\Sigma_2\otimes\Sigma_1}^{\AI}=\tilde{g}_1^\AI\oplus \tilde{g}_2^{\AI}$ on $\calE$ by $\psi_{p_1,p_2}$ via pullback geometry is given by $\tilde{g}_i^{\AI}=g_{\Sigma_i}^{\AI}/p_{-i}$ for $p_{-1}=p_2$ and $p_{-2}=p_1$. Then for the Kronecker map $k$, let $\eta_{p_1,p_2}:=\psi_{p_1,p_2}\circ k$. Since $k:=\psi_{p_1,p_2}^{-1}\circ \eta_{p_1,p_2}$, if $\eta_{p_1,p_2}$ is smooth, so is $k$. Also, the chain rule implies that 
\begin{align}\label{sec5.1.eq1}
    dk(\Sigma)[V]=d\psi_{p_1,p_2}^{-1}(k(\Sigma))[d\eta_{p_1,p_2}(\Sigma)[V]]
\end{align}
for $V\in T_\Sigma \calS_{p}^{++}$. Note that $\eta_{p_1,p_2}$ maps $\Sigma\in\calS_{p_1,p_2}^{++}$ to a unique minimizer of function $d$ in (\ref{sec2.1.eq1}) over $\calE$ by identifying $K_2\otimes K_1$ in (\ref{sec2.1.eq1}) via the map $\psi_{p_1,p_2}$. Also, 
\begin{align}\label{sec5.1.eq2}
    d\psi_{p_1,p_2}^{-1}(\Sigma_1,\Sigma_2):(U_1,U_2)\in T_{(\Sigma_1,\Sigma_2)}\calE\rightarrow U_2\otimes \Sigma_1+\Sigma_2\otimes U_1\in T_{\Sigma_2\otimes \Sigma_1} \calS_{p_1,p_2}^{++}.
\end{align}
After identifying $d\eta_{p_1,p_2}(\Sigma)[V]$ via the manifold  implicit function theorem (\cite{loomis2014}, Section $3.11$), we use (\ref{sec5.1.eq1})--(\ref{sec5.1.eq2}) to obtain $dk(\Sigma)[V]$ and so $dc(\Sigma)[V]$ using Lemma \ref{sec5.1.lemma1}. 

\begin{prop}\label{sec5.1.prop2}
The Kronecker map $k:\calS_p^{++}\rightarrow \calS_{p_1,p_2}^{++}$ is smooth. Consequently, the map $f:\Sigma\in \calS_p^{++}\rightarrow (k(\Sigma),c(\Sigma))\in \calS_{p_1,p_2}^{++}\times \calC_{p_1,p_2}^{++}$ is so for either $h\in\calS_{p_1,p_2}^{++}$ or $h\in\calC_{p_1,p_2}^{++}$. Therefore, $ \calS_p^{++}$ is diffeomorphic to $\calS_{p_1,p_2}^{++}\times \calC_{p_1,p_2}^{++}$ as the map $g$ in Proposition \ref{sec5.1.prop1} is also smooth. Moreover, let $k(\Sigma)=\Sigma_2\otimes\Sigma_1$ with $|\Sigma_1|=1$, $c(\Sigma)=C$, and $V\in T_\Sigma \calS_p^{++}\equiv \calS_p$.  Also, define the bilinear operator $\calR_C:T_{(\Sigma_1,\Sigma_2)}\calE\rightarrow T_{(\Sigma_1,\Sigma_2)}\calE$ by 
\begin{align*}
    \calR_C(U_1,U_2)&=\left(U_1+\Sigma_1^{1/2}M_1\Sigma_1^{1/2}/p_2-\tr\parentheses{\Sigma_2^{-1/2}U_2\Sigma_2^{-1/2}}\Sigma_1/p_2,\; U_2+\Sigma_2^{1/2}M_2\Sigma_2^{1/2}/p_1 \right)
\end{align*}
where $M_1\in \calS_{p_1}$ and $M_2\in\calS_{p_2}$ are given by 
\begin{align*}
    M_1&=\sum_{i,j=1}^{p_2}(\Sigma_2^{-1/2}U_2\Sigma_2^{-1/2})_{ij}C_{[j,i]},\quad (M_2)_{ij}=\tr\parentheses{C_{[i,j]}\Sigma_1^{-1/2}U_1\Sigma_1^{-1/2}}.
\end{align*}
Here $\Sigma_i^{1/2}\in\calS_{p_i}^{++}$. Then the operator $\calR_C$ is a bijection. Furthermore,  the differential of $k$ at $\Sigma$ is given by
\begin{align*}
    dk(\Sigma)[V]=U_2\otimes\Sigma_1+\Sigma_2\otimes U_1,
\end{align*}
where $(U_1,U_2)$ is a unique solution to the equation that
\begin{align*}
    \calR_C(U_1,U_2)&=\parentheses{\Sigma_1^{1/2}\left[\tr_1(\tilde{V})-\tr(\tilde{V})/p_1 I_{p_1}\right]\Sigma_1^{1/2}/p_2,\Sigma_2^{1/2}\tr_2(\tilde{V})\Sigma_2^{1/2}/p_1)}\\
    &=(M_1,M_2)
\end{align*}
for $\tilde{V}:=K^{-1/2}VK^{-1/2}$ with symmetric $K^{1/2}\equiv \Sigma_2^{1/2}\otimes\Sigma_1^{1/2}$. Also,
\begin{align*}
    dc(\Sigma)[V]&=h(k(\Sigma))^{-1}Vh(k(\Sigma))^{-\top}-h(k(\Sigma))^{-1}(dh(k(\Sigma))[dk(\Sigma)[V]])C\\
    &-C(dh(k(\Sigma))[dk(\Sigma)[V]])^\top h(k(\Sigma))^{-\top},
\end{align*}
From the differentials computed above, we consequently have that
\begin{align*}
    df(\Sigma)[V]=(dk(\Sigma)[V],dc(\Sigma)[V]).
\end{align*}
In particular, if $C=I_p$, i.e., $\Sigma\equiv k(\Sigma)=\Sigma_2\otimes\Sigma_1$, $\calR_C$ reduces to an identity operator so that $(U_1,U_2)=(M_1,M_2)$, in which
\begin{align*}
    dk(\Sigma)[V]&=\parentheses{\Sigma_2^{1/2}\tr_2(\tilde{V})\Sigma_2^{1/2}}\otimes\Sigma_1/p_1+\Sigma_2\otimes\parentheses{\Sigma_1^{1/2}\tr_1(\tilde{V})\Sigma_1^{1/2}}/p_2-\frac{\tr\parentheses{\tilde{V}}}{p}\Sigma,\\
    dc(\Sigma)[V]&=h(k(\Sigma))^{-1}Vh(k(\Sigma))^{-\top}-\tilde{R}-\tilde{R}^\top.
\end{align*}
Here if $h\in\calL_{p_1,p_2}^{++}$, 
\begin{align*}
    \tilde{R}=\parentheses{I_{p_2}\otimes L_1^{-1}U_1L_1^{-\top}+L_2^{-1}U_2L_2^{-\top}\otimes I_{p_1}}_{\frac{1}{2}}.
\end{align*}
Otherwise, if $h\in\calS_{p_1,p_2}^{++}$,
\begin{align*}
    \tilde{R}=
        (\Gamma_2\otimes\Gamma_1)(\Lambda_2^{-1/2}\otimes\Lambda_1^{-1/2})\left[\Lambda^{-}\circ \parentheses{\Lambda_2\otimes \Gamma_1^\top U_1\Gamma_1+\Gamma_2^\top U_2\Gamma_2\otimes \Lambda_1}\right](\Gamma_2\otimes\Gamma_1)^\top.
\end{align*}
where the quantities associated with $\tilde{R}$ are those defined in Lemma \ref{sec5.1.lemma1}.
\end{prop}
\begin{proof}
    See Appendix \ref{append.A.5} for the proof.
\end{proof}

\subsection{Riemannian gradient and Hessian operator on $\calC_{p_1,p_2}^{++}$}\label{sec5.2}
Endow $\calC_{p_1,p_2}^{++}$ with the Euclidean metric $g^E$. By Proposition \ref{sec3.prop2}, the form of the tangent space $T_C\calC_{p_1,p_2}^{++}$ does not depend on $C\in\calC_{p_1,p_2}^{++}$. The same holds for the form of the metric, i.e., $g_C^E(U,V)=\tr\parentheses{U^\top V}$ for $U,V\in T_C\calC_{p_1,p_2}^{++}$. Thus, letting $\mathcal{W}:=T_C\calC_{p_1,p_2}^{++}$, which does not depend on $C$, $\mathcal{W}$ is a linear subspace of $\calS_p$. Also, with any fixed basis of $\mathcal{W}$ as a coordinate on each tangent space, we have that the Christoffel symbols (see $(5.10)$ of \cite{lee2018}) vanish on that coordinate. Thus, $(\calC_{p_1,p_2}^{++},g^E)$ has a zero-sectional curvature, and so $(\calC_{p_1,p_2}^{++},g^E)$ is flat (\cite{lee2018}, Theorem $7.10$). Now we derive the Riemannian gradient and Hessian operator on $(\calC_{p_1,p_2}^{++},g^E)$. For a scalar-valued smooth map $f$ on $\calC_{p_1,p_2}^{++}$, denote the Euclidean derivative and Hessian operator of $f$ by $\nabla f(C)$ and $\nabla^2 f(C)[V]$ for $C\in\calC_{p_1,p_2}^{++}$ and $V\in T_C\calC_{p_1,p_2}^{++}$.

\begin{lemma}\label{sec5.2.lemma1}
For $C\in\calC_{p_1,p_2}^{++}$, let $W\in T_C\calC_{p_1,p_2}^{++}$. The operator $\calG:\calS_p\rightarrow \calS_p$ given by 
\begin{align}\label{sec5.2.lemma1.eq1}
    \calG(V):=V-\frac{1}{p_2}(I_{p_2}\otimes \tr_1(V))-\frac{1}{p_1}(\tr_2(V)\otimes I_{p_1})+\frac{\tr(V)}{p} I_p
\end{align}
is an orthogonal projection of $V\in T_C\calS_p^{++}\equiv \calS_p$ onto $T_C\calC_{p_1,p_2}^{++}$.
\end{lemma}
\begin{proof}
        See Appendix \ref{append.A.6} for the proof.
\end{proof}

\begin{prop}\label{sec5.2.prop1}
Suppose $f$ is a scalar-valued smooth map on $(\calC_{p_1,p_2}^{++},g^E)$. Let $C\in \calC_{p_1,p_2}^{++}$ and $V\in T_C\calC_{p_1,p_2}^{++}$. For the operator $\calG$ defined in (\ref{sec5.2.lemma1.eq1}), 
\end{prop}
\begin{align*}
    \grad f(C)=\calG(\nabla f(C)),\quad \Hess f(C)[V]=\calG(\nabla^2 f(C)[V]).
\end{align*}
\begin{proof}
      See Appendix \ref{append.A.6} for the proof.  
\end{proof}

\subsection{Riemannian gradient and Hessian operator on $\calC_{p_1,p_2,r}$ and  $\calC_{p_1,p_2,r}/\calO_r$}\label{sec5.3}
We derive the Riemannian gradient and Hessian operator on $(\calC_{p_1,p_2,r},g^E)$ and then deduce those on $(\calC_{p_1,p_2,r}/\calO_r,g^{E,0})$ via quotient geometry. Here $g^{E,0}$ is the quotient metric induced by $g^E$. Throughout this section, we denote the Moore-Penrose pseudoinverse of a matrix $M$ by $M^\dagger$ (see p.50--51 of \cite{rao1971}). We establish the Riemannian gradient and Hessian operator on $(\calC_{p_1,p_2,r},g^E)$ as follows. 

\begin{lemma}\label{sec5.3.lemma1}
Recall the linear operator $J$ in (\ref{sec4.2.eq1}). Let $A\in\calC_{p_1,p_2,r}$ and endow $\calC_{p_1,p_2,r}$ with metric $g^E$. Suppose $\tilde{A}=\varphi_{p_1,p_2,r}^{-1}(A)$. Then for any $V\in T_A\real_*^{p\times r}$, if $W$ is the orthogonal projection of $V$ onto $T_A\calC_{p_1,p_2,r}$,
\begin{align*}
    \text{vec}(W)=(I-J(\tilde{A})^\dagger J(\tilde{A}))\text{vec}(V).
\end{align*}
\end{lemma}
\begin{proof}
          See Appendix \ref{append.A.7} for the proof.  
\end{proof}

\begin{prop}\label{sec5.3.prop1}
Recall the linear operator $J$ in (\ref{sec4.2.eq1}). Let $f$ be a scalar-valued smooth function on $(\calC_{p_1,p_2,r},g^E)$. Let $A\in\calC_{p_1,p_2,r}$ (resp. $V\in T_A\calC_{p_1,p_2,r}$) and suppose $\tilde{A}:=\varphi_{p_1,p_2,r}^{-1}(A)$ (resp. $\tilde{V}:=\varphi_{p_1,p_2,r}^{-1}(V)$). Then,
\begin{align*}
    \text{vec}(\grad f(A))&=(I-J(\tilde{A})^\dagger J(\tilde{A}))\text{vec}(\nabla f(A)),\\
    \text{vec}(\Hess f(A)[V])&=(I-J(\tilde{A})^\dagger J(\tilde{A}))\text{vec}(\nabla^2 f(A)[V])\\
    &-(I-J(\tilde{A})^\dagger J(\tilde{A}))J(\tilde{V})^\top (J(\tilde{A})^\dagger)^\top J(\tilde{A})^\dagger J(\tilde{A})\text{vec}(\nabla f(A)).
\end{align*}
\end{prop}
\begin{proof}
     See Appendix \ref{append.A.7} for the proof.  
\end{proof}

The Riemannian gradient and Hessian operator on $(\calC_{p_1,p_2,r}/\calO_r,g^{E,0})$ can be derived using the results of \cite{absil2008,massart2020} (see Section $5.1$ and $5.6$ of \cite{chen2025} for example). For $A\in \calC_{p_1,p_2,r}$, the vertical and horizontal spaces at $A$ are given by
\begin{align*}
    \calV_A&=\set{A\Theta:\Theta\in\textsf{Skew}_r},\quad \calH_A\equiv \calV_A^\perp =\set{B\in T_A\calC_{p_1,p_2,r}:A^\top B=B^\top A}.
\end{align*}
Note that $T_A\calC_{p_1,p_2,r}=\calV_A\oplus \calH_A$. Then we introduce the operators on $T_A\calC_{p_1,p_2,r}$ by \cite{massart2020} as 
\begin{align}\label{sec5.2.eq1}
    P_A^v(W):=A\mathbf{T}_{A^\top A}^{-1}(2\textsf{skew}(A^\top W)),\quad P_A^h(W):=W- P_A^v(W).
\end{align}
for $W\in T_A\calC_{p_1,p_2,r}$. Here the operator $\mathbf{T}_{E}^{-1}(\cdot)$ is the inverse of the map $\mathbf{T}_E:Y\in \real^{r\times r}\rightarrow YE+EY\in \real^{r\times r}$. If $E\in\calS_r^{++}$, $\mathbf{T}_E$ is invertible and the value of its inverse $\mathbf{T}_{E}^{-1}(V)$ for given $V$ is a unique solution to the Sylvester equation $YE+EY=V$ (\cite{massart2020}, Lemma A.10). Also, by Section $5.3.4$ of \cite{absil2008}, the Riemannian connection $\nabla$ on $\calC_{p_1,p_2,r}/\calO_r$ satisfies that 
\begin{align*}
    (\nabla_\eta \xi)_A^\#=P_A^h ((\nabla_{\eta^\#}^\# \xi^\#)_A)
\end{align*}
for any $A\in\calC_{p_1,p_2,r}$, vector fields $\eta,\xi$ on $\calC_{p_1,p_2,r}/\calO_r$ and the operator $P_A^h$ defined in (\ref{sec5.2.eq1}). Note that $\eta^\#$ and $\xi^\#$ are horizontal lifts of $\eta$ and $\xi$, respectively. As a direct consequence of Section $3.6.2$ and $5$ of \cite{absil2008} and Proposition A.14 of \cite{massart2020}, we have the following results.   

\begin{prop}\label{sec5.3.prop2}
    Suppose $f$ is a smooth map on $\calC_{p_1,p_2,r}/\calO_r$, and let $f^\#=f\circ \pi$ for the canonical projection $\pi:X\in \calC_{p_1,p_2,r}\rightarrow [X]\in \calC_{p_1,p_2,r}/\calO_r$. Then for any $A\in \calC_{p_1,p_2,r}$, the Riemannian gradient of $f$ satisfies  
    \begin{align*}
        (\grad f([A]))_A^\#=\grad f^\# (A),
    \end{align*}
Also, the Riemannian Hessian operator of $f$ satisfies
\begin{align*}
    (\Hess f([A])[\xi_{[A]}])_A^\#=P_A^h((\nabla_{\xi^\#}^{\#}\grad f^\#)_A)
\end{align*}
for any vector field $\xi$ on $\calC_{p_1,p_2,r}/\calO_r$.
\end{prop}

\section{Partial isotropy core shrinkage estimator}\label{sec6}
Using the geometry of $\calC_{p_1,p_2,r}$, we shall propose a shrinkage estimator that shrinks the low-dimensional core toward the trivial core, $I_p$. Suppose $Y_1,\ldots,Y_n\overset{i.i.d.}{\sim}N_{p_1\times p_2}(0,\Sigma)$ and let $K^{1/2}CK^{1/2,\top}$ be a KCD of $\Sigma$. In the estimation of $\Sigma$, note that the dimension of the space where $K=k(\Sigma)$ is living is $O(p_1^2+p_2^2)=o(p^2)$, whereas that of the space where $C=c(\Sigma)$ is living is $O(p^2)$ by Theorem \ref{sec3.thm1}. Thus, the difficulty of the estimation arises mainly from estimating $C$, particularly in a high-dimensional regime where $p>n$. 

As a remedy, we consider a partial-isotropy structure on $C$ to introduce a low-dimensional structure to $C$ as discussed in Section \ref{sec1} and \ref{sec2.1}. By Proposition \ref{sec2.1.prop1}, $C=(1-\lambda)AA^\top+\lambda I_p$ for some $\lambda\in(0,1)$ and $A\in\calC_{p_1,p_2,r}$, with $r>p_1/p_2+p_2/p_1$. Thus, $\Sigma=K^{1/2}((1-\lambda)AA^\top+\lambda I_p)K^{1/2,\top}$, leading to the partial-isotropy core covariance model in (\ref{sec1.eq2}). Namely,
\begin{align}\label{sec6.eq1}
    Y_1,\ldots,Y_n\overset{i.i.d.}{\sim}N_{p_1\times p_2}(0, K^{1/2}((1-\lambda)AA^\top+\lambda I_p)K^{1/2,\top}).
\end{align}
Note that $\lambda$ denotes the shrinkage amount of the low-dimensional covariance $K^{1/2}AA^\top K^{1/2,\top}$ toward the separable part $K:=K^{1/2}K^{1/2,\top}$. Hence, $\lambda$ quantifies an effective departure from the separability assumption on $\Sigma$, i.e., $\Sigma=K$, where the correlation structure of $\Sigma$ may be too simplified as $p$ grows. 

For the identifiability of the covariance model above, recall that there is an ambiguity in identifying the factors of $K^{1/2}:=K_2\otimes K_1$ due to a constant factor. To avoid this ambiguity, we shall reparameterize $K^{1/2}$ by $K^{1/2}:=\nu (\bar{K}_2\otimes\bar{K}_1)$, where $\bar{K}_i\in\bbP(\calS_{p_i}^{++})$ or $\bbP(\calL_{p_i}^{++})$, and $\nu>0$. Under this parameterization, the parameters constituting the partial-isotropy core covariance are identifiable with $A$ up to right-rotation. 
\begin{prop}\label{sec6.prop1}
Given $p_1/p_2+p_2/p_1<r< p$, define the parameter space 
\begin{align*}
    \Theta:=\bbP(\calM_1)\times \bbP(\calM_2)\times \real_+\times \calC_{p_1,p_2,r}\times (0,1),
\end{align*}
where $(\calM_1,\calM_2)$ is either $(\calL_{p_1}^{++},\calL_{p_2}^{++})$ or $(\calS_{p_1}^{++},\calS_{p_2}^{++})$. Suppose a smooth map $\Omega:\Theta\rightarrow\calS_p^{++}$ is defined by 
\begin{align*}
    \Omega(\tau)=K^{1/2}((1-\lambda)AA^\top+\lambda I_p)K^{1/2,\top},
\end{align*}
where $K^{1/2}=\nu(\bar{K}_2\otimes \bar{K}_1)$ and $\tau=(\bar{K}_1,\bar{K}_2,\nu,A,\lambda)$. For $\tau_i=(\bar{K}_1^i,\bar{K}_2^i,\nu^i,A^i,\lambda^i)\in \Theta$ with $i=1,2$,
$\Omega(\tau_1)=\Omega(\tau_2)$ if and only if for some $O\in\calO_r$,
\begin{align*}
    &\bar{K}_1^1=\bar{K}_1^2,\quad \bar{K}_2^1=\bar{K}_2^2,\quad \nu^1=\nu^2,\quad A^1=A^2O,\quad \lambda^1=\lambda^2.
\end{align*}
\end{prop}
\begin{proof}
    See Appendix \ref{append.A.7} for the proof.
\end{proof}

 Now we propose a partial-isotropy core shrinkage estimator ($\PICSE$) by $\hat{K}^{1/2}((1-\hat{\lambda})\hat{A}\hat{A}^\top+\hat{\lambda})\hat{K}^{1/2,\top}$, where $\hat{\theta}$ is a MLE of the parameter $\theta$. We shall consider both square roots $K^{1/2}\in\calS_{p_1,p_2}^{++}$ or $\calL_{p_1,p_2}^{++}$ in the estimation. Given the data $Y_1,\ldots,Y_n$ according to the model in (\ref{sec6.eq1}), the negative log-likelihood is given by 
\begin{align}\label{sec6.eq2}
\begin{split}
\ell(\bar{K}_1,\bar{K}_2,\nu,A,\lambda)&:=\tr\parentheses{\bar{K}^{-1/2}S\bar{K}^{-1/2,\top}((1-\lambda)AA^\top+\lambda I_p)^{-1}}/\nu^2\\
    &+\log|(1-\lambda)AA^\top+\lambda I_p|+2p\log \nu.
\end{split}
\end{align}
where $S=1/n\sum_{i=1}^n y_iy_i^\top$ is a sample covariance matrix of $Y_1,\ldots, Y_n$ with $y_i=\text{vec}(Y_i)$, and $\bar{K}^{1/2}=\bar{K}_{2}\otimes \bar{K}_1$. We shall minimize $\ell$ in (\ref{sec6.eq2}). For the minimizer of $\ell$, denoted $\hat{\tau}=(\hat{\bar{K}}_1,\hat{\bar{K}}_2,\hat{\nu},\hat{A},\hat{\lambda})$, \texttt{PICSE} is defined to be 
\begin{align}\label{sec6.eq3}
    \hat{\Sigma}_{\texttt{PICSE}}:=\hat{\nu}^2(\hat{\bar{K}}_2\otimes \hat{\bar{K}}_1)((1-\hat{\lambda})\hat{A}\hat{A}^\top+\hat{\lambda}I_p)(\hat{\bar{K}}_2\otimes \hat{\bar{K}}_1)^\top. 
\end{align}
Since there is no closed form for $\hat{\tau}$, we propose an alternating minimization approach to compute $\hat{\tau}$, and thus $\hat{\Sigma}_{\PICSE}$. Specifically, at $t$-th iteration, we sequentially update each parameter fixing all other parameters as  
\begin{align}\label{sec6.eq4}
\begin{split}
    \bar{K}_1^{(t)}&:=\argmin_{\bar{K}_1\in\bbP(\calM_1)}\ell (\bar{K}_1,\bar{K}_2^{(t-1)},\nu^{(t-1)},A^{(t-1)},\lambda^{(t-1)}),\\  
\bar{K}_2^{(t)}&:=\argmin_{\bar{K}_2\in\bbP(\calM_2)}\ell (\bar{K}_1^{(t)},\bar{K}_2,\nu^{(t-1)},A^{(t-1)},\lambda^{(t-1)}),\\    
\nu^{(t)}&:=\argmin_{\nu\in \real_{+}}\ell (\bar{K}_1^{(t)},\bar{K}_2^{(t)},\nu,A^{(t-1)},\lambda^{(t-1)}),\\  
A^{(t)}&:=\argmin_{A\in\calC_{p_1,p_2,r}}\ell (\bar{K}_1^{(t)},\bar{K}_2^{(t)},\nu^{(t)},A,\lambda^{(t-1)}),\\
\lambda^{(t)}&:=\argmin_{\lambda\in(0,1)}\ell (\bar{K}_1^{(t)},\bar{K}_2^{(t)},\nu^{(t)},A^{(t)},\lambda).
\end{split}
\end{align}
given the initialization $(\bar{K}_1^{(0)},\bar{K}_2^{(0)},\nu^{(0)},A^{(0)},\lambda^{(0)})$. Here $\calM_i=\calS_{p_i}^{++}$ (resp. $\calL_{p_i}^{++}$) if $K^{1/2}\in\calS_{p_1,p_2}^{++}$ (resp. $\calL_{p_1,p_2}^{++}$). We iterate (\ref{sec6.eq4}) until the convergence, and obtain the estimate $\hat{\Sigma}_{\PICSE}$ by plugging the output for each parameter into (\ref{sec6.eq3}). 

We discuss the update rule for each parameter in (\ref{sec6.eq4}). Note that in the sequels, the core component of some positive semi-definite matrix is defined by whitening it through the square root of its separable component as the same type of $K^{1/2}$ in (\ref{sec6.eq1}). Except for $\nu$ and $\lambda$, we adopt second-order Riemannian optimization to update the parameter. Namely, suppose $\theta\in\set{\bar{K}_1,\bar{K}_2,A}$ and let $(\calM,g)$ be Riemannian manifold on which $\theta$ is living. If $\theta$ is either $\bar{K}_1$ or $\bar{K}_2$, we take $g=g^{\text{AI}}$ (resp. $g=g^{\text{Chol}}$) if $\calM=\bbP(\calS_{p_i}^{++})$ (resp. $\calM=\bbP(\calL_{p_i}^{++})$). For $\theta=A$, $g=g^{E}$, i.e., Euclidean metric. Suppose $V\in T_\theta\calM$. Fixing all the parameters other than $\theta$, we obtain the optimal direction $V$ to update $\theta^{(t-1)}$ by solving the equation in $V$ that
\begin{align}\label{sec6.eq5}
    \Hess\ell(\theta^{(t-1)})[V]=-\grad\ell(\theta^{(t-1)}),
\end{align}
which leads to
\begin{align*}
    \bar{V}=-\Hess\ell(\theta^{(t-1)})^\dagger [\grad\ell(\theta^{(t-1)})].
\end{align*}
When obtaining $\bar{V}$ according to the above, we also need Euclidean derivative and Hessian operator of $\ell$ in $\theta$, as the Riemannian gradient and Hessian operator depend on them (see Section \ref{sec2.2.1}--\ref{sec2.2.2} and Section \ref{sec5.3}). We provide their formulas in Appendix \ref{append.B}. If $\theta$ is either $\bar{K}_1$ or $\bar{K}_2$, we obtain 
\begin{align}\label{sec6.eq6}
    \bar{K}_i^{(t)}=\Exp_{\bar{K}_i^{(t-1)}}(\bar{V})
\end{align}
for the solution $\bar{V}$ of (\ref{sec6.eq5}). On the other hand, to obtain $A^{(t)}$, suppose $\bar{V}$ is the solution of (\ref{sec6.eq5}) with $\theta=A$. For $D^{(t-1)}:=A^{(t-1)}A^{(t-1),\top}+A^{(t-1)}\bar{V}^{(t-1),\top}+\bar{V}^{(t-1)}A^{(t-1),\top}$, let $\bar{D}^{(t-1)}$ be its core component. Suppose $\Gamma^{(t-1)}\in\real^{p\times r}$ is the matrix of top$-r$ eigenvectors of $\bar{D}^{(t-1)}$ and $\Lambda^{(t-1)}=\diag\parentheses{\sqrt{\lambda_1(\bar{D}^{(t-1)})},\ldots,\sqrt{\lambda_r(\bar{D}^{(t-1)})}}$. Then we have an update $A^{(t)}$ as 
\begin{align}\label{sec6.eq7}
 A^{(t)}=\Gamma^{(t-1)}\Lambda^{(t-1)}.
\end{align}
To obtain $\nu^{(t)}$ and $\lambda^{(t)}$, note that the closed form of $\nu^{(t)}$ is available as 
\begin{align}\label{sec6.eq8}
    \nu^{(t)}=\sqrt{\frac{\tr\parentheses{(\bar{K}^{(t)})^{-1}S(\bar{K}^{(t)})^{-\top}((1-\lambda^{(t-1)})A^{(t-1)}(A^{(t-1)})^\top+\lambda^{(t-1)} I_p)^{-1}}}{p}},
\end{align}
where $\bar{K}^{(t)}=\bar{K}_2^{(t)}\otimes\bar{K}_1^{(t)}$. We numerically obtain $\lambda^{(t)}$ using \texttt{R} function \texttt{optimize}. 

Now to discuss the initialization, suppose $\tilde{K}$ and $\tilde{C}$ are the separable and core components of $S$. Let $\tilde{C}_r$ be the core component of the best rank$-r$ approximation of $\tilde{C}$. Then the initialization is given as 
\begin{align}\label{sec6.eq9}
        \bar{K}_i^{(0)}&=\tilde{K}_i/|\tilde{K}_i|^{1/p_i},\, \nu^{(0)}=\prod_{i=1}^2|\tilde{K}_i|^{1/p_i},\,\lambda^{(0)}=\frac{p-\sum_{i=1}^r \lambda_i(\tilde{C}_r)}{p-r},\, A^{(0)}=U_r\Lambda_r,
\end{align}
where $U_r$ is a top$-r$ eigenvectors of $\tilde{C}_r$ and $\Lambda_r=\diag\parentheses{\sqrt{\lambda_1(\tilde{C}_r)},\ldots,\sqrt{\lambda_r(\tilde{C}_r)}}$. We summarize the optimization procedure discussed above in Algorithm \ref{sec6.algorithm1}. 
\begin{algorithm}
\caption{An algorithm for alternating minimization of $\ell$ in $(\bar{K}_1,\bar{K}_2,\nu,A,\lambda)$.}\label{sec6.algorithm1}
 $\epsilon>0:$ tolerance parameter, $T\in\bbN:$ maximum number of iterations, $Y_1,\ldots,Y_n\in\real^{p_1\times p_2}$: $n$ data matrices, $r\in\bbN:$ partial-isotropy rank. 
\begin{algorithmic}
\State Compute the sample covariance matrix $S$ of $Y_1,\ldots,Y_n$.
\State Compute the initialization $(\bar{K}_1^{(0)},\bar{K}_2^{(0)},\nu^{(0)},A^{(0)},\lambda^{(0)})$ according to (\ref{sec6.eq9}).
\State $\ell^{(1)}=\ell(\bar{K}_1^{(0)},\bar{K}_2^{(0)},\nu^{(0)},A^{(0)},\lambda^{(0)})$, $\ell^{(0)}=\ell^{(1)}/2$.
\State $t=1$.
\While{$|\ell^{(t-1)}-\ell^{(t)}|/|\ell^{(t)}|>\epsilon $ and $t< T$}
\State $\ell^{(t-1)}=\ell^{(t)}$.
\State $t=t+1$. 
\State Obtain $\bar{K}_1^{(t)}$ according to (\ref{sec6.eq5})--(\ref{sec6.eq6}).
\State Obtain $\bar{K}_2^{(t)}$ according to (\ref{sec6.eq5})--(\ref{sec6.eq6}).
\State Obtain $\nu^{(t)}$ according to (\ref{sec6.eq8}). 
\State Obtain $A^{(t)}$ according to (\ref{sec6.eq5}) and (\ref{sec6.eq7}).
\State Solve the fifth equation of (\ref{sec6.eq4}) using \texttt{R} function \texttt{optimize} to obtain $\lambda^{(t)}$.
\State $\ell^{(t)}=\ell(\bar{K}_1^{(t)},\bar{K}_2^{(t)},\nu^{(t)},A^{(t)},\lambda^{(t)})$.
\EndWhile
\end{algorithmic}
\end{algorithm}

\section{Illustration of PICSE}\label{sec7}

We illustrate the effectiveness of $\PICSE$ based on synthetic data.\footnote{Replication code is available at \url{https://github.com/Seungbongjung/riemmCore}.} We randomly generate the random matrices $Y_1,\ldots,Y_n$ according to $N_{p_1\times p_2}(0,\Sigma)$, where $\Sigma$ is given as follows; for $K^{1/2}=K_2\otimes K_1\in\calS_{p_1,p_2}^{++}$, $A\in\calC_{p_1,p_2,r}$, $\lambda\in(0,1)$, and a diagonal $D\in\calC_{p_1,p_2}^{++}$,
\begin{itemize}
    \item[]\textbf{(M1)} $\Sigma=K^{1/2}((1-\lambda)AA^\top+\lambda I_p)K^{1/2,\top}$,
    \item[]\textbf{(M2)} $\Sigma=K^{1/2}((1-\lambda)AA^\top+\lambda D)K^{1/2,\top}$.
\end{itemize}
The model \textbf{(M1)} is the partial isotropy core covariance model in (\ref{sec6.eq1}), and thus $c(\Sigma)=(1-\lambda)AA^\top+\lambda I_p$. On the other hand, the model \textbf{(M2)} is a variant of \textbf{(M1)}, motivated by a general factor covariance model \cite{lopes2008,carvalho2008}. We also consider this model to examine the robustness of $\PICSE$ under a broader class of covariance models for matrix-variate data, containing the partial isotropy core covariance model. Under this model, $c(\Sigma)=(1-\lambda)AA^\top+\lambda D$. By the linear system in (\ref{sec2.1.eq2}) that defines the core, if a diagonal $D$ is a core, $(1-\lambda)AA^\top+\lambda D$ is again a core. One can easily generate such a $D$ by randomly generating a positive definite diagonal $\tilde{D}$ and then taking its core $D=c(\tilde{D})$. This is because the Kronecker MLE of any positive definite diagonal matrix $\tilde{D}$ is again diagonal, and so is its core (\cite{hoff2023}, Corollary $1$).      

To investigate how $\PICSE$ behaves by varying degrees of how $\Sigma$ is separable, we take $\lambda=0.2,0.4,0.6,0.8$. Note that the smaller $\lambda$ is, the less separable $\Sigma$ is. Also, we consider $r=3,5$, $(p_1,p_2)=(16,12),(18,8)$, and $n=p/8,p/4,p/2,p,2p$. The other parameters $(K_1,K_2,A,D)$ are randomly generated. We assume the known $r$. In practice, the value of $r$ can be determined by estimating the number of the spiked eigenvalues of the true core using the sample core in view of Kronecker-invariance \cite{sung2025}, assuming the constant non-spiked eigenvalues, e.g., \cite{passemier2012,passemier2014}.  

To describe the competitors of $\PICSE$, suppose $S$ is a sample covariance matrix based on random matrices $Y_1,\ldots,Y_n$. Let $\tilde{K}$ and $\tilde{C}$ be the separable and the core components of $S$, respectively. Then we consider the Kronecker MLE ($\KMLE$) \cite{soloveychik2016,srivastava2008,dutilleul999}, which is exactly $\tilde{K}$, and the core shrinkage estimator ($\CSE$) proposed by \cite{hoff2023}. The $\CSE$ is defined by $\tilde{K}^{1/2}((1-\hat{w})\tilde{C}+\hat{w}I_p)\tilde{K}^{1/2}$, where $\tilde{K}^{1/2}$ is a symmetric square root of $\tilde{K}$ and the shrinkage amount $\hat{w}$ is estimated via empirical Bayes (see Section $3.1$ of \cite{hoff2023}). Additionally, we consider the ad hoc estimator based on the initialization in (\ref{sec6.eq9}). Namely, we obtain the estimators by plugging the initialization in (\ref{sec6.eq9}) into (\ref{sec6.eq3}), denoted \texttt{Base-AI} ($K^{1/2}\in\calS_{p_1,p_2}^{++}$). Note that the initialization method yields the same estimate of $\Sigma$ if $K^{1/2}\in\calL_{p_1,p_2}^{++}$. For $\PICSE$, we consider two versions, denoted by \texttt{PI-AI} ($K^{1/2}\in\calS_{p_1,p_2}^{++}$) and \texttt{PI-Chol} $(K^{1/2}\in\calL_{p_1,p_2}^{++})$. The ad hoc estimator is considered to examine whether leveraging the geometry of $\calC_{p_1,p_2,r}$ to find the optimal direction in updating $A$ leads to a better estimator.   

We present the results with respect to estimating $\Sigma$ for both models (\textbf{M1})--(\textbf{M2}). Figures \ref{sec7.figure1}--\ref{sec7.figure4} show the box plots of the relative errors across $100$ iterations with different $(p_1,p_2,r)$, $\lambda$, and $n$. In general, one can verify that $\KMLE$ performs poorly compared to other methods and shows a small standard deviation of the relative norms. This is because the dimension of $\calS_{p_1,p_2}^{++}$ is much lower than that of the space where the partial-isotropy core covariance lies. Thus, $\KMLE$ is already close to the pseudo-true parameter, namely, the separable component of $\Sigma$, $K$. Hence, $\KMLE$ may be estimating $K$ well but yield a poor estimate of $\Sigma$ as its core component is fixed as $I_p$. This implies that the Kronecker MLE is not a good estimate if the true covariance is not separable. 

On the other hand, the other methods tend to show the improved performance as $n$ grows for each choice of $\lambda$ and $(p_1,p_2,r)$ under the models \textbf{(M1)}--\textbf{(M2)}. Note that both \texttt{PI-AI} and \texttt{PI-Chol} perform better than $\CSE$ and \texttt{Base-AI}, particularly for the small values of $n$ and $\lambda$ under both \textbf{(M1)}--\textbf{(M2)}. This illustrates the robustness of $\PICSE$. Also, the performance gap between \texttt{Base-AI} and $\PICSE$ is more obvious with small $n$ and $\lambda$. This is because $\PICSE$ leverages the curvature of the negative log-likelihood in (\ref{sec6.eq2}) to find the optimal $A$. Also, while there might be no significant improvement for $\PICSE$ in estimating $K$ compared to other methods as the dimension of $\calS_{p_1,p_2}^{++}$ is $o(p^2)$, there is in estimating $C$, particularly with small $n$ and $\lambda$. This implies that $\PICSE$ is effective in a high-dimensional regime when the true core exhibits a low-dimensional feature and is far from the separability.         

Lastly, note that the only scenario where $\CSE$ performs better than $\PICSE$ and \texttt{Base-AI} is for small $n$ but large $\lambda$. That is, when the sample size is small, whereas the true covariance is close to separability, $\CSE$ can perform better than these two estimators. The reason is that $\CSE$ estimates the non-spiked eigenvalue $\lambda$ via empirical Bayes and tends to shrink more toward the separability compared to $\PICSE$. Hence, it may be less prone to overfitting for small sample sizes when $\Sigma$ is close to separability.

\begin{figure}
    \centering
    \includegraphics[width=14cm,height=10cm]{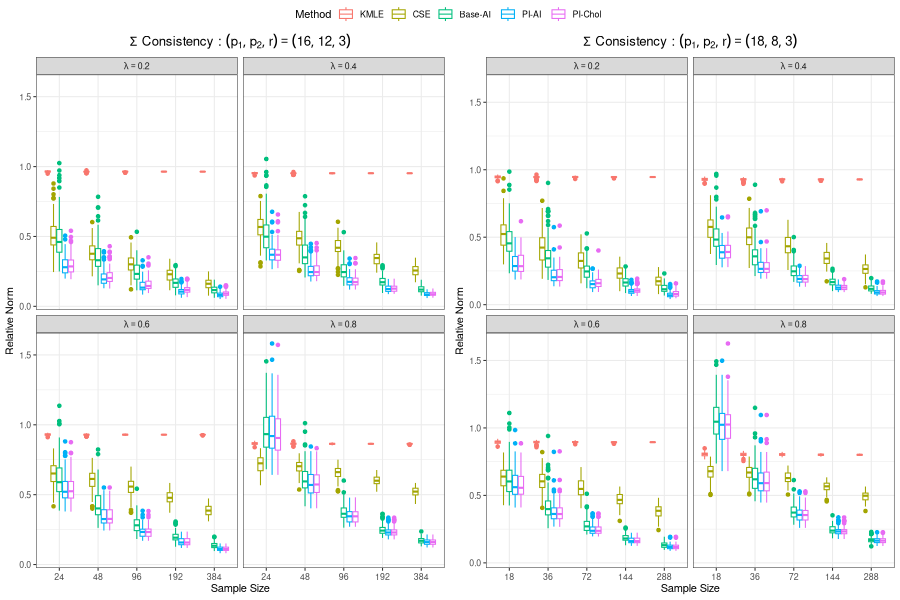}
    \caption{\small The box plots of the relative norms $||\hat{\Sigma}-\Sigma||_2/||\Sigma||_2$ by \texttt{KMLE}, \texttt{CSE}, \texttt{Base-AI}, \texttt{PI-AI}, and \texttt{PI-Chol}, and the sample size $n=p/8,p/4,p/2,p,2p$ across $100$ iterations for $(p_1,p_2,r)=(16,12,3),(18,8,3)$ and $\lambda=0.2,0.4,0.6,0.8$ under the model \textbf{(M1)}. \texttt{Base-AI} and \texttt{Base-Chol} yield the same $\hat{\Sigma}$, and thus the result is reported only for \texttt{Base-AI} as a representative.}
    \label{sec7.figure1}
\end{figure}

\begin{figure}
    \centering
    \includegraphics[width=14cm,height=10cm]{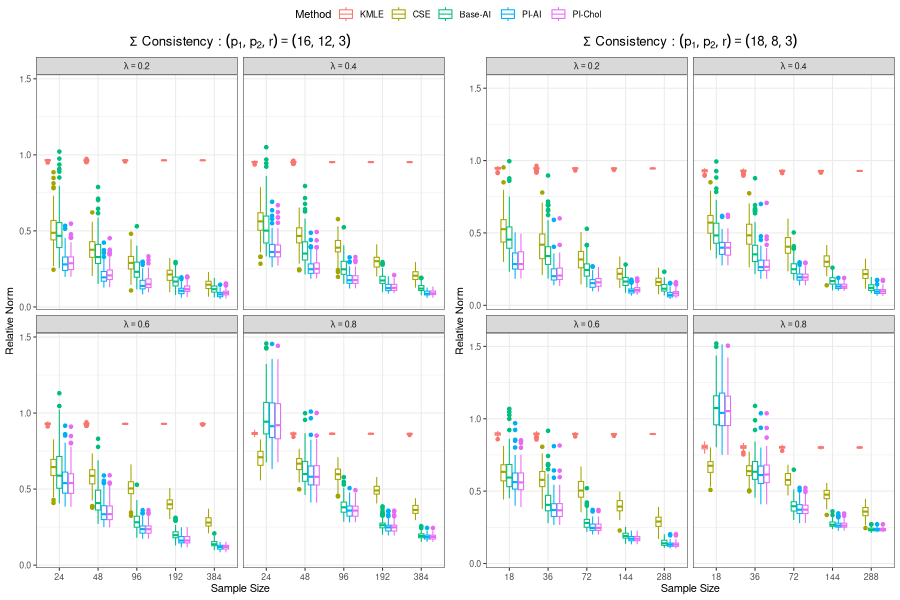}
    \caption{\small The box plots of the relative norms $||\hat{\Sigma}-\Sigma||_2/||\Sigma||_2$ by \texttt{KMLE}, \texttt{CSE}, \texttt{Base-AI}, \texttt{PI-AI}, and \texttt{PI-Chol}, and the sample size $n=p/8,p/4,p/2,p,2p$ across $100$ iterations for $(p_1,p_2,r)=(16,12,3),(18,8,3)$ and $\lambda=0.2,0.4,0.6,0.8$ under the model \textbf{(M2)}. \texttt{Base-AI} and \texttt{Base-Chol} yield the same $\hat{\Sigma}$, and thus the result is reported only for \texttt{Base-AI} as a representative.}
    \label{sec7.figure2}
\end{figure}

\begin{figure}
    \centering
    \includegraphics[width=14cm,height=10cm]{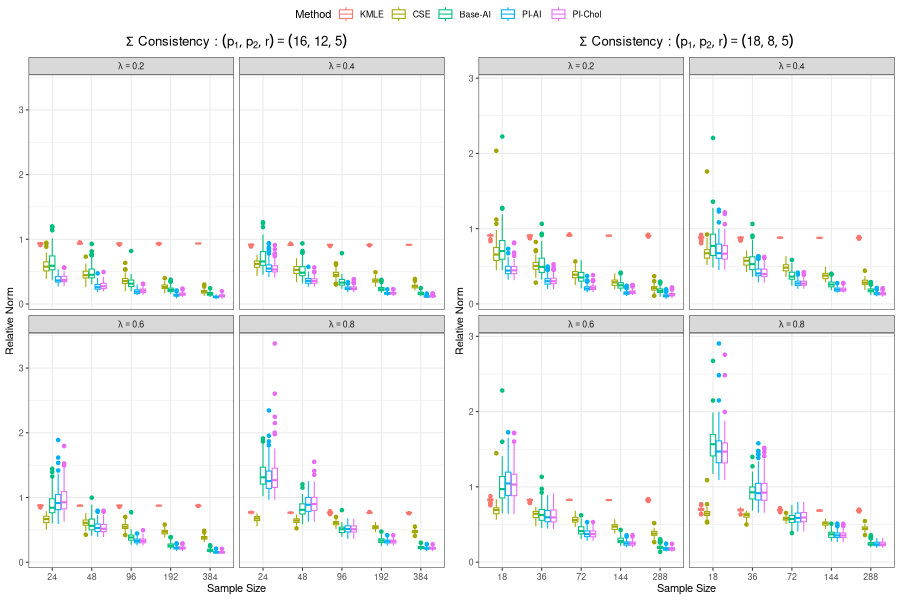}
    \caption{\small The box plots of the relative norms $||\hat{\Sigma}-\Sigma||_2/||\Sigma||_2$ by \texttt{KMLE}, \texttt{CSE}, \texttt{Base-AI}, \texttt{PI-AI}, and \texttt{PI-Chol}, and the sample size $n=p/8,p/4,p/2,p,2p$ across $100$ iterations for $(p_1,p_2,r)=(16,12,5),(18,8,5)$ and $\lambda=0.2,0.4,0.6,0.8$ under the model \textbf{(M1)}. \texttt{Base-AI} and \texttt{Base-Chol} yield the same $\hat{\Sigma}$, and thus the result is reported only for \texttt{Base-AI} as a representative.}
    \label{sec7.figure3}
\end{figure}

\begin{figure}
    \centering
    \includegraphics[width=14cm,height=10cm]{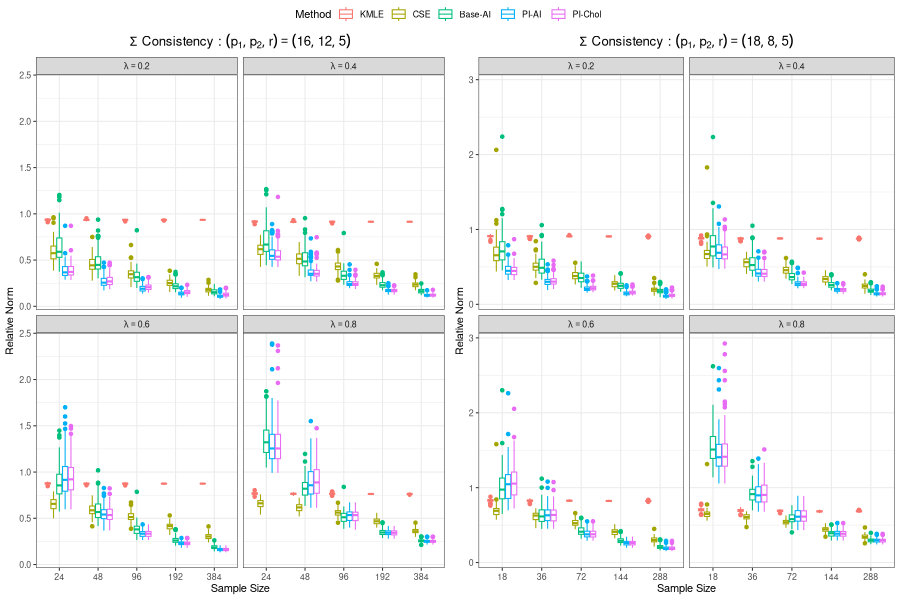}
    \caption{\small The box plots of the relative norms $||\hat{\Sigma}-\Sigma||_2/||\Sigma||_2$ by \texttt{KMLE}, \texttt{CSE}, \texttt{Base-AI}, \texttt{PI-AI}, and \texttt{PI-Chol}, and the sample size $n=p/8,p/4,p/2,p,2p$ across $100$ iterations for $(p_1,p_2,r)=(16,12,5),(18,8,5)$ and $\lambda=0.2,0.4,0.6,0.8$ under the model \textbf{(M2)}. \texttt{Base-AI} and \texttt{Base-Chol} yield the same $\hat{\Sigma}$, and thus the result is reported only for \texttt{Base-AI} as a representative.}
    \label{sec7.figure4}
\end{figure}

\section{Concluding remarks}\label{sec8}
We have studied the geometry of the fixed-rank core covariance manifold $\calC_{p_1,p_2,r}^{+}$ with $p_1/p_2+p_2/p_1<r\leq p$. When $r<p$, we established that $\calC_{p_1,p_2,r}^{+}$ is a smooth manifold after removing the set of canonically decomposable matrices. For the full-rank case, we further established a diffeomorphic relationship between $\calS_{p}^{++}$ and $\calS_{p_1,p_2}^{++}\times \calC_{p_1,p_2}^{++}$, providing a new insight into the smooth structure of $\calS_{p}^{++}$ in terms of the separability. We also derived differential quantities on $\calC_{p_1,p_2,r}^+$, including tangent vectors, the differentials of the diffeomorphism when $r=p$, and Riemannian gradient and Hessian operator on  $\calC_{p_1,p_2,r}$, and $\calC_{p_1,p_2}^{++}$ under the Euclidean metric, with respect to which $\calC_{p_1,p_2}^{++}$ is flat. The corresponding Riemannian gradient and Hessian operator are also obtained for $\calC_{p_1,p_2,r}/\calO_r$ via quotient geometry. 

An interesting future direction is to identify a Riemannian metric on $\calC_{p_1,p_2,r}^+$ that induces nice geometric properties, such as completeness, closed-form geodesics, or nonpositive sectional curvature. One approach is to construct a group-invariant metric (see \cite{thanwerdas2023b}). By the linear system that defines the core as in (\ref{sec2.1.eq2}), the Kronecker orthogonal group $\calO_{p_1,p_2}$ acts smoothly on $\calC_{p_1,p_2,r}^+$ via the action $(O_2\otimes O_1,C)\in \calO_{p_1,p_2}\times \calC_{p_1,p_2,r}^+\rightarrow (O_2\otimes O_1)C(O_2\otimes O_1)^\top\in \calC_{p_1,p_2,r}^+$. However, this action is not transitive, and the invariant metric is hence not unique under this action. Thus, the metric will vary across orbits, leading to infinitely many invariant metrics. We leave this as a future direction to further explore the geometry of $\calC_{p_1,p_2,r}^+$.

We introduced the partial isotropy core shrinkage estimator ($\PICSE$), assuming that the population core has a partial-isotropy structure. Since the partial-isotropy (factor) covariance model is often used for vector data and the covariance of a random matrix is defined by that of its vectorization (see (\ref{sec1.eq1})), one might ask whether $\PICSE$ can be used to estimate a factor-type covariance matrix for general $p-$dimensional vector data with correctly specified $p_1$ and $p_2$. Although technically possible, we do not recommend using $\PICSE$ for vector data as it will lose interpretation. As discussed in Section \ref{sec6}, the partial-isotropy core covariance model aims to make an effective departure from the separability assumption commonly used in modeling matrix-variate data. Thus, using $\PICSE$ is meaningful only when the separability assumption is valid. Because the assumption enables a separate inference of the correlation structures of the row and column variables \cite{dawid1981,wang2022}, the data must have two different modes for the assumption, which is not the case for general vector data.  

\section*{Declaration of Competing Interest}

The authors declare that they have no known competing financial interests or personal relationships that could have appeared to influence the work reported in this paper.

\section*{Data availability}

No data was used for the research described in the article.

\appendix 

\section{Deferred proofs}\label{append.A}
In this section, we provide the omitted proofs from the main text.

\subsection{Proofs of the results from Section \ref{sec2}}\label{append.A.1}
\begin{proof}[Proof of Example \ref{sec2.1.ex1}]
    The proof is based on Proposition $3$ of \cite{hoff2023}. Suppose $K=\Sigma_2\otimes\Sigma_1$ is a Kronecker MLE of $FF^\top$. Then since $K^{-1/2}FF^\top K^{-1/2,\top}$ is a core, Proposition $3$ of \cite{hoff2023} implies that 
    \begin{align}\label{sec2.1.ex1.eq1}
        \sum_{(i,j)} E_{ij}\Omega_2 E_{ij}^\top&=2\Sigma_1,\; 
         \sum_{(i,j)} E_{ij}^\top \Omega_1 E_{ij}=2\Sigma_2,
    \end{align}
    for $\Omega_i=(\omega_{i,ab}):=\Sigma_i^{-1}$. From the first equation of (\ref{sec2.1.ex1.eq1}),
    \begin{align*}
        \Sigma_1=\text{diag}(\omega_{2,11}+\omega_{2,22},\omega_{2,22})/2.
    \end{align*}
Since $\Sigma_1$ is diagonal, so is $\Omega_1$. Thus, the second equation of (\ref{sec2.1.ex1.eq1}) should imply that $\Sigma_2$ is also diagonal. Hence, writing $\Sigma_2=\text{diag}(\sigma_{2,11},\sigma_{2,22})$, we have that 
\begin{align*}
    &\Sigma_{1}=\text{diag}(1/\sigma_{2,11}+1/\sigma_{2,22},1/\sigma_{2,22})/2\\
    \Rightarrow &\Omega_1=2\text{diag}(\sigma_{2,11}\sigma_{2,22}/(\sigma_{2,11}+\sigma_{2,22}),\sigma_{2,22}).
\end{align*}
Again, the second equation of (\ref{sec2.1.ex1.eq1}) implies that 
\begin{align*}
    \sigma_{2,11}=\sigma_{2,11}\sigma_{2,22}/(\sigma_{2,11}+\sigma_{2,22}).
\end{align*}
Since $\Sigma_{2}\in \calS_2^{++}$, we have that $\sigma_{2,22}/(\sigma_{2,11}+\sigma_{2,22})=1$, which is not true unless $\sigma_{2,11}=0$. Hence, this contradicts the existence of the Kronecker MLE $K$.
\end{proof}

\begin{proof}[Proof of Proposition \ref{sec2.2.2.prop1}]
     Let $v(t)=f(\gamma(t))$ for $t\in(0,1)$, where $\gamma(t)=\Exp_L(tV)$. Then $v$ is also smooth. Observe that  
    \begin{align}\label{sec2.2.2.prop1.eq1}
    \begin{split}
         \gamma'(t)&=\floor{V}+\bbD(V)\exp\parentheses{t\bbD(V)\bbD(L)^{-1}},\\
      \gamma''(t)&=\bbD(V)^2\bbD(L)^{-1}\exp\parentheses{t\bbD(V)\bbD(L)^{-1}}.
    \end{split}
    \end{align}
    By chain rule, 
    \begin{align}\label{sec2.2.2.prop1.eq2}
    \begin{split}
              v'(t)&=\tr\parentheses{\nabla f(\gamma(t))^\top \gamma'(t)},\\
              v''(t)&=\tr\parentheses{\nabla^2 f(\gamma(t))[\gamma'(t)]^\top \gamma'(t)}+\tr\parentheses{\nabla f(\gamma(t))^\top \gamma''(t)}.
    \end{split}
    \end{align}
    Then we have that 
    \begin{align*}
    v'(0)&=g_L(\grad f(L),V)\equiv \tr\parentheses{\floor{\grad f(L)}^\top \floor{V}}+\tr\parentheses{\bbD(L)^{-2}\bbD(\grad f(L))\bbD(V)}\\
    &=\tr\parentheses{\nabla f(L)^\top V}=\tr\parentheses{\floor{\nabla f(L)}^\top \floor{V}}+\tr\parentheses{\bbD(\nabla f(L))\bbD(V)}.
    \end{align*}
    Since this holds for any $V\in T_L\calL_{p}^{++}\equiv \calL_p$, we have that $\floor{\grad f(L)}=\floor{\nabla f(L)}$ and $\bbD(L)^{-2}\bbD(\grad f(L))=\bbD(\nabla f(L))$, resulting in 
    \begin{align*}
        \grad f(L)=\floor{\grad f(L)}+\bbD(\grad f(L))=\floor{\nabla f(L)}+\bbD(L)^2 \bbD(\nabla f(L)).
    \end{align*}
    Similarly, by (\ref{sec2.2.2.prop1.eq1})--(\ref{sec2.2.2.prop1.eq2}), 
    \begin{align*}
        v''(0)&=g_L(\Hess f(L)[V],V)=\tr\parentheses{\nabla^2 f(L)[V]^\top V}+\tr\parentheses{\nabla f(L)^\top \bbD(V)^2\bbD(L)^{-1}}\\
        &=\tr\parentheses{\nabla^2 f(L)[V]^\top V}+\tr\parentheses{\bbD(\nabla f(L)) \bbD(V)^2\bbD(L)^{-1}}.
    \end{align*}
    By polarization and the symmetry of Riemannian Hessian operator, for any $V,W\in T_L\calL_p^{++}$,
    \begin{align*}
        g_L(\Hess f(L)[V],W)=\tr\parentheses{\nabla^2 f(L)[V]^\top W}+\tr\parentheses{\bbD(\nabla f(L)) \bbD(V)\bbD(W)\bbD(L)^{-1}}.
    \end{align*}
    As an analogy to $\grad f(L)$, one can then identify $\Hess f(L)[V]$ by 
    \begin{align*}
        \Hess f(L)[V]=\bbD(L)^2\bbD(\nabla^2 f(L)[V])+\floor{\nabla^2 f(L)[V]}+\bbD(L)\bbD(\nabla f(L))\bbD(V).
    \end{align*}
\end{proof}

\begin{proof}[Proof of Proposition \ref{sec2.2.2.prop2}]
Observe that 
\begin{align*}
    \tr\parentheses{L^{-1}\calP_L(V)}&=\tr(L^{-1}V)-\tr(L^{-1}V)\cdot \tr(L^{-1}\bbD(L))/p\\
    &=\tr(L^{-1}V)-\tr(L^{-1}V)\cdot \tr(\bbD(L^{-1})\bbD(L))/p\\
    &=\tr(L^{-1}V)-\tr(L^{-1}V)\cdot \tr(\bbD(L)^{-1}\bbD(L))/p\\
    &=\tr(L^{-1}V)-\tr(L^{-1}V)=0,
\end{align*}
where the third equality holds as $L\in\calL_{p}^{++}$. Thus, $\calP_L$ indeed maps $V\in T_L\calL_p^{++}$ to $T_L\bbP(\calL_p^{++})$. Hence,  it suffices to verify that $g_L^\chol(V,\bbD(L))=0$ for any $V\in T_L\bbP(\calL_p^{++})$ to claim that $\calP_L$ is an orthogonal projection. This follows because
\begin{align*}
    g_L^\chol(V,\bbD(L))=g^E(\bbD(L)^{-2}\bbD(V),\bbD(L))=\tr(\bbD(L)^{-1}\bbD(V))=\tr(L^{-1}V)=0,
\end{align*}
where the third equality holds because $L\in\calL_{p}^{++}$ and $V\in\calL_p$, and the last equality follows as $V\in T_L\bbP(\calL_p^{++})$.

\end{proof}

\subsection{Proofs of the results from Section \ref{sec3}}\label{append.A.2}

\begin{proof}[Proof of Lemma \ref{sec3.lemma1}]
Let $\phi_{R}:A\in (\real^{p_1\times p_2})_*^p\rightarrow [A_1,\ldots,A_p]\in \real^{p_1\times pp_2}$ and $\phi_C:A\in (\real^{p_1\times p_2})_*^p\rightarrow [A_1^\top,\ldots,A_p^\top]\in \real^{p_2\times pp_1}$. For any $A\in (\real^{p_1\times p_2})_*^p$, $A_R\equiv \phi_R(A)\phi_R(A)^\top$ and $A_C\equiv \phi_C(A)\phi_C(A)^\top$ are of full-rank if and only if $\phi_R(A)$ and $\phi_C(A)$ are so, respectively. Thus, $\calH_{p_1,p_2}=\phi_R^{-1}(\real_*^{p_1\times pp_2})\cap \phi_C^{-1}(\real_*^{p_2\times pp_1})$. Note that both the maps $\phi_R$ and $\phi_C$ are smooth, and $\real_*^{p_1\times pp_2}$ and $\real_*^{p_2\times pp_1}$ are open in their respective ambient space. Thus, both $\phi_R^{-1}(\real_*^{p_1\times pp_2})$ and $ \phi_C^{-1}(\real_*^{p_2\times pp_1})$ are open in  $(\real^{p_1\times p_2})_*^p$. Hence, the transversality theorem (\cite{lee2012}, Theorem 6.35) implies that $\calH_{p_1,p_2}$ is an open submanifold of $(\real^{p_1\times p_2})_*^p$. Because $(\real^{p_1\times p_2})_*^p$ is diffeomorphic to $\real_*^{p\times p}$ via the map $\varphi_{p_1,p_2}$, $(\real^{p_1\times p_2})_*^p$ is open in $(\real^{p_1\times p_2})^p$. Therefore, $T_A\calH_{p_1,p_2}=T_A(\real^{p_1\times p_2})_*^p\equiv (\real^{p_1\times p_2})^p$ for any $A\in\calH_{p_1,p_2}$.
\end{proof}

\begin{proof}[Proof of Proposition \ref{sec3.prop1}]
Following the main idea outlined in the sketch of proof, it suffices to verify (\ref{sec3.prop1.eq2}). Take $B=(B_1,\ldots,B_p)\in T_A\calH_{p_1,p_2}$. Note that $a^\top a=p$, $A_R=p_2 I_{p_1}$, and $A_C=p_1 I_{p_2}$. By Theorem $3.1$ of \cite{magnus1979}, $K_{(q,q)}$ is a symmetric matrix whose eigenvalues are either $1$ or $-1$, with respective multiplicities $q(q+1)/2$ and $q(q-1)/2$. Moreover, the eigenspace of $K_{(q,q)}$ corresponding to the eigenvalue $1$ (resp. $-1$) is exactly the vectorization of $\calS_{q}$ (resp. $\textsf{Skew}_{q}$). By vec-Kronecker identity,
\begin{align*}
[A_1\otimes I_{p_1},\ldots,A_p\otimes I_{p_1}]a&=\sum_{i=1}^p \text{vec}(A_iA_i^\top)=p_2\text{vec}(I_{p_1}),\\
\end{align*}

Combining these facts, one can verify that  
\begin{align*}
    J_1J_1^\top&=\frac{2}{p_1^2p_2}(I_{p_1^2}+K_{(p_1,p_1)})-\frac{4}{p_1^3p_2}\text{vec}(I_{p_1})\text{vec}(I_{p_1})^\top,\\
     J_2J_2^\top&=\frac{2}{p_1p_2^2}(I_{p_2^2}+K_{(p_2,p_2)})-\frac{4}{p_1p_2^3}\text{vec}(I_{p_2})\text{vec}(I_{p_2})^\top.
\end{align*}
By the aforementioned properties of $K_{(q,q)}$, $I_{q^2}+K_{(q,q)}$ takes eigenvalues either $2$ or $0$ with respective multiplicities $q(q+1)/2$ and $q(q-1)/2$. The eigenspaces of this matrix corresponding to $2$ and $0$ are the same as those of $K_{(q,q)}$ corresponding to $1$ and $-1$, respectively. Therefore,
\begin{align*}
    \dim C(J_1^\top)=\text{rank}(J_1^\top)=\text{rank}(J_1J_1^\top)=\binom{p_1+1}{2}-1
\end{align*}
and similarly, $\dim C(J_2^\top)=\binom{p_2+1}{2}-1$. It remains to show that $\dim C(J_1^\top)\cap C(J_2^\top)=0$, i.e., $C(J_1^\top)\cap C(J_2^\top)=\set{\mathbf{0}_{p^2}}$. Because $(I_{q^2}+K_{(q,q)})u=2\text{vec}(\textsf{sym}(U))$ for $u=\text{vec}(U)$ and $U\in \real^{q\times q}$, if we assume $v\in\real^{p_1^2}$ is a vectorization of some $V\in\calS_{p_1}$ without loss of generality, 
\begin{align*}
    J_1^\top v&=\frac{2}{p}\left[\begin{array}{c}
    (A_1^\top \otimes I_{p_1}) v \\
    \vdots \\
    (A_p^\top \otimes I_{p_1})v
    \end{array}\right]-\frac{2\tr\parentheses{V}}{p_1^2p_2} a=\frac{2}{p}\left[\begin{array}{c}
     \text{vec}(VA_1) \\
    \vdots \\
    \text{vec}(VA_p)
    \end{array}\right]-\frac{2\tr\parentheses{V}}{p_1^2p_2}\left[\begin{array}{c}
    \text{vec}(A_1)\\
    \vdots \\
    \text{vec}(A_p)
    \end{array}\right].
\end{align*}
Likewise, if $w=\text{vec}(W)$ for some $W\in \calS_{p_2}$, 
\begin{align*}
    J_2^\top w=\frac{2}{p}\left[\begin{array}{c}
     \text{vec}(A_1W) \\
    \vdots \\
    \text{vec}(A_pW)
    \end{array}\right]-\frac{2\tr\parentheses{W}}{p_1p_2^2}\left[\begin{array}{c}
    \text{vec}(A_1)\\
    \vdots \\
    \text{vec}(A_p)
    \end{array}\right].
\end{align*}
Therefore, for any element in $C(J_1^\top)\cap C(J_2^\top)$, there exist $V\in \calS_{p_1}$ and $W\in \calS_{p_2}$ such that for all $i\in [p]$,
\begin{align}\label{sec3.prop1.eq3}
    (V-\tr\parentheses{V}/p_1 I_{p_1})A_i=A_i(W-\tr\parentheses{W}/p_2 I_{p_2}).
\end{align}
Suppose $\Gamma_V \Lambda_V\Gamma_V^\top$ and $\Gamma_W\Lambda_W\Gamma_W^\top$ are eigendecompositions of $V$ and $W$, respecitvely, where $\Gamma_V$ and $\Gamma_W$ are orthogonal, and $\Lambda_V$ and $\Lambda_W$ are diagonal. With $\tilde{A}_i=\Gamma_V^\top A_i\Gamma_W$ for each $i\in[r]$, the equation (\ref{sec3.prop1.eq3}) can be reformulated as 
\begin{align}\label{sec3.prop1.eq4}
   \tilde{\Lambda}_V\tilde{A}_i=\tilde{A}_i\tilde{\Lambda}_W
\end{align}
where $\tilde{\Lambda}_V=\Lambda_V-\tr\parentheses{\Lambda_V}/p_1 I_{p_1}$ and  $\tilde{\Lambda}_W=\Lambda_W-\tr\parentheses{\Lambda_W}/p_2 I_{p_2}$. The equation (\ref{sec3.prop1.eq4}) holds for each $i$ if and only if $((\tilde{\Lambda}_V)_{jj}-(\tilde{\Lambda}_W)_{kk})(\tilde{A}_i)_{jk}=0$ for any $(j,k)\in[p_1]\times [p_2]$ and $i\in [p]$. Note that $\tilde{A}_1,\ldots,\tilde{A}_p$ are linearly independent by the definition of $\calH_{p_1,p_2}$. Thus, for any fixed $(j,k)\in[p_1]\times [p_2]$, there must exist at least one $i\in[p]$ for which $(\tilde{A}_i)_{jk}$ is nonzero, otherwise the linear independence is violated. Hence, $(\tilde{\Lambda}_V)_{jj}=(\tilde{\Lambda}_W)_{kk}$ for any $j,k$. As such, the diagonal entries of $\tilde{\Lambda}_V$ and $\tilde{\Lambda}_W$ are all equal to some constant $c$. However, as the traces of these matrices are the same as $0$, both $\tilde{\Lambda}_V$ and $\tilde{\Lambda}_W$ are zero diagonal matrices so that $V=\Lambda_V=c_1I_{p_1}$ and $W=\Lambda_W=c_2 I_{p_2}$ for some constants $c_1,c_2$, leading to zero vectors $J_1^\top v$ and $J_2^\top w$. Hence, $C(J_1^\top)\cap C(J_2^\top)=\set{0}$, concluding the proof.  
\end{proof}

\begin{proof}[Proof of Lemma \ref{sec3.lemma2}]
    The smoothness of the action is obvious. Also, since $\rank(X)=p$ for any $X\in\calC_{p_1,p_2}$ so that $X$ is non-singular, $XO_1=XO_2$ implies that $O_1=O_2$ and thus the action is free. To show that the action is well-defined, let $X=[x_1,\ldots,x_p]\in\calC_{p_1,p_2}$, $O=[y_1,\ldots,y_p]\in \calO_p$, and $Y=XO=[y_1,\ldots,y_p]$. Suppose $X_i:=\textsf{mat}_{p_1\times p_2}(x_i)$ and $Y_i:=\textsf{mat}_{p_1\times p_2}(y_i)$. With $\bar{Y}=(Y_1,\ldots,Y_p)$, we shall verify that $Y\in \calD_{p_1,p_2}$. Note that $\bar{Y}\in(\real^{p\times p})_*^p$ if and only if $Y_i$'s are linearly independent, which holds as $Y$ is of rank$-p$ and the action of $\calO_p$ does not alter the rank of $X$. Since $Y_i=\sum_{j=1}^p X_jo_{ ji}$,
\begin{align*}
  \tilde{Y}_R\tilde{Y}_R^\top =\sum_{i=1}^p \sum_{j,j'=1}^p X_jX_{j'}^\top o_{ji}o_{j'i}=\sum_{j,j'=1}^p X_jX_{j'}^\top \sum_{i=1}^p o_{ji}o_{j'i}=\sum_{j=1}^p X_jX_j^\top=p_2 I_{p_1},
\end{align*}
where the last equality holds as $(X_1,\ldots,X_p)\in\calD_{p_1,p_2}$. Similarly,  $\tilde{Y}_C\tilde{Y}_C^\top =p_1 I_{p_2}$.
\end{proof}

\begin{proof}[Proof of Lemma \ref{sec3.lemma3}]
      By the quotient manifold theorem (\cite{lee2012}, Theorem $21.10$), both $\calM/G$ and $\calN/G$ are smooth manifolds. Let $i:\calN/G\xhookrightarrow{}\calM/G$ be an inclusion map. We claim that $i$ is a smooth embedding. By Theorem $4.4$ of \cite{martin2017}, $i$ is an injective immersion. To show that $i$ is a topological embedding, let $\pi:X\in\calM\rightarrow[X]\in\calM/G$ be the canonical projection, which is a smooth submersion. Hence, $\pi$ is open and $\calN$ is $G-$invariant so that $\calN$ is saturated for the map $\pi$. Here the subset $C\subset U$ is saturated for the map $f:U\rightarrow V$ between two topological spaces if $f^{-1}(f(C))=C$. Since $\calN$ is a topological subspace of $\calM$ as it is embedded, so is $\pi(\calN)\equiv \calN/G$ of the quotient space $\pi(\calM)\equiv \calM/G$ (\cite{munkres2000}, Theorem $22.1$). Thus, $i$ is a topological embedding.
\end{proof}

\begin{proof}[Proof of Proposition \ref{sec3.prop2}]
        Let $\mathcal{W}:= \set{W\in\calS_p:\tr_1(V)=\mathbf{0}_{p_1\times p_1},\tr_2(V)=\mathbf{0}_{p_2\times p_2}}$. Taking any $W\in T_C\calC_{p_1,p_2}^{++}$, suppose $\gamma:(-\epsilon,\epsilon)\rightarrow \calC_{p_1,p_2}^{++}$ is a smooth curve emanating from $C$ in the direction of $W$ for sufficiently small $\epsilon>0$ so that the curve is moving around $\calC_{p_1,p_2}^{++}$. For any such $t$, (\ref{sec2.1.eq2}) implies that 
\begin{align*}
\tr_1(\gamma(t))=p_2I_{p_1},\quad \tr_2(\gamma(t))=p_1I_{p_2}.
\end{align*}
Evaluating the derivative of the terms in both hand sides at $t=0$ for each equation above, we have that 
\begin{align}\label{sec3.prop2.eq1}
  \tr_1(W)=\mathbf{0}_{p_1\times p_1},\quad \tr_2(W)=\mathbf{0}_{p_2\times p_2}.
\end{align}
Thus, $W\in \mathcal{W}$ and so $T_C\calC_{p_1,p_2}^{++}\subset \mathcal{W}$. Note that $\dim T_C\calC_{p_1,p_2}^{++}=\binom{p+1}{2}-\binom{p_1+1}{2}-\binom{p_2+1}{2}+1$. Since both $T_C\calC_{p_1,p_2}^{++}$ and $\mathcal{W}$ are linear subspaces of $\calS_{p}$, it suffices to show that $\dim \mathcal{W}$ is the same as $\dim T_C\calC_{p_1,p_2}^{++}$. Let $(W_{[i,j]})$ be a block-partition of $W$. Then the equation (\ref{sec3.prop2.eq1}) is satisfied if and only if $\tr\parentheses{W_{[i,j]}}=0$ for all $i,j$, and $W_{[p_2,p_2]}=-\sum_{i=1}^{p_2-1}W_{[i,i]}$. The subspace of $\real^{p_1\times p_1}$ whose trace is $0$ is of dimension $p_1^2-1$, and there are exactly $\binom{p_2}{2}$ upper-diagonal blocks. Also, each of the diagonal blocks belongs to the subspace of $\calS_{p_1}$ whose trace is $0$ and dimension is $p_1-1+\binom{p_1}{2}$. Since $W_{[p_2,p_2]}$ is determined by the rest of the diagonal blocks, the dimension of $\mathcal{W}$ is given by 
\begin{align*}
    (p_1^2-1)\binom{p_2}{2}+\parentheses{p_1-1+\binom{p_1}{2}}(p_2-1)=\binom{p+1}{2}-\binom{p_1+1}{2}-\binom{p_2+1}{2}+1.
\end{align*}
\end{proof}

\subsection{Proofs of the results from Section \ref{sec4.1}}\label{append.A.3}

\begin{proof}[Proof of Lemma \ref{sec4.1.lemma1}]
The claim is obvious if either $\alpha=2$ or $\beta=2$, and so assume $\alpha,\beta>2$. Also, assume $\alpha\geq \beta$ without loss of generality. Since $m$ is smooth on its compact domain, $m$ attains its maximum on the domain. This happens at either its stationary point within the interior of the domain or the boundary of the domain. Noting that $m(a,b)=m(\alpha-a,\beta-b)$ and $[1,\alpha-1]\times [1,\beta-1]$ is symmetric around $(\alpha/2,\beta/2)$, it suffices to examine the maximum of $m$ at the boundary of its domain by the maximum of $f(b):=m(1,b)$ over $[1,\beta-1]$. It is straightforward to see that $(\alpha/2,\beta/2)$ is a unique stationary point of $m$ with $m(\alpha/2,\beta/2)=\alpha^2/4+\beta^2/4+r\alpha\beta/2<r\alpha\beta$ as $r>\alpha/\beta+\beta/\alpha$. To examine the maximum of $f$, note that 
\begin{align*}
    f(b)&=\alpha-1+b(\beta-b)+r(b+(\alpha-1)(\beta-b))\\
    &=-b^2+b(\beta+r(2-\alpha))+(\alpha-1)(r\beta+1).
\end{align*}
Observe that $(\beta+r(2-\alpha))/2\leq \beta-1$ as $\alpha>2$. Also, $(\beta+r(2-\alpha))/2\geq 1$ if and only if $r<(\beta-2)/(\alpha-2)$. However, because $(\beta-2)/(\alpha-2)\leq \beta/\alpha$ as $\alpha\geq \beta$, this cannot hold as $r>\alpha/\beta+\beta/\alpha>\beta/\alpha$. Thus, $f$ attains its maximum at $b=1$, where $f(1)=\alpha+\beta-2+r(\alpha\beta-\alpha-\beta+2)$. Since 
\begin{align*}
    r\alpha\beta-f(1)&=(r-1)(\alpha+\beta-2)>0
\end{align*}
as $r>\alpha/\beta+\beta/\alpha\geq 2$ and $\alpha,\beta>2$, we conclude that the maximum of $m$ is strictly smaller than $r\alpha\beta$.
\end{proof}

\begin{proof}[Proof of Proposition \ref{sec4.1.prop1}]
    Define a subset $\calV_{p_1,p_2,r}^{(a,b)}$ of $\calV_{p_1,p_2,r}$ for $a\in[p_1-1]$ and $b\in[p_2-1]$, consisting of canoincally decomposable $(A_1\ldots,A_r)$ for which there exists $(P,Q)\in GL_{p_1}\times GL_{p_2}$ such that $PA_iQ^{-1}=A_{i1}\oplus A_{i2}$ for some $A_{i1}\in \real^{a\times b}$ and $A_{i2}\in\real^{(p_1-a)\times (p_2-b)}$. Then 
\begin{align*}
    \calV_{p_1,p_2,r}=\cup_{a=1}^{p_1-1}\cup_{b=1}^{p_2-1}\calV_{p_1,p_2,r}^{(a,b)}.
\end{align*}
We claim that for each $(a,b)$, the set $\calV_{p_1,p_2,r}^{(a,b)}$ is a proper Zariski-closed in $(\real^{p_1\times p_2})^r$. By Lemma \ref{sec2.4.lemma1}, $\calV_{p_1,p_2,r}$ is also Zariski-closed and hence closed in Euclidean sense. Also, the fourth item of Lemma \ref{sec2.4.lemma1} implies that $\calV_{p_1,p_2,r}$ also has a measure zero, concluding the claim. It will be shown that the dimension of $\calV_{p_1,p_2,r}^{(a,b)}$ is upper bounded by $m(a,b)$ for the map $m$ defined in Lemma \ref{sec4.1.lemma1} with $\alpha=p_1$ and $\beta=p_2$. The first item of Lemma \ref{sec2.4.lemma1} implies that the dimension of $\calV_{p_1,p_2,r}$ is also upper bounded by $\max_{(a,b)\in[p_1-1]\times [p_2-1]}m(a,b)$. Because $r>p_1/p_2+p_2/p_1$ by the assumption, Lemma \ref{sec4.1.lemma1} implies that the dimension of  $\calV_{p_1,p_2,r}^{(a,b)}$ is strictly smaller than that of the ambient space $(\real^{p_1\times p_2})^r$, $pr$, and thus $\calV_{p_1,p_2,r}^{(a,b)}$ is indeed proper. 

To show that each $\calV_{p_1,p_2,r}^{(a,b)}$ is Zariski-closed in $(\real^{p_1\times p_2})^r$, let $A=(A_1,\ldots,A_r)\in\calV_{p_1,p_2,r}^{(a,b)}$. Then there exists $(P,Q)\in GL_{p_1}\times GL_{p_2}$ such that $PA_iQ^{-1}=A_{i1}\oplus A_{i2}$ for $A_{i1}\in \real^{a\times b}$ and $A_{i2}\in\real^{(p_1-a)\times(p_2-b)}$. Viewing each matrix $A_i$ as a linear operator that maps $\real^{p_2}$ to $\real^{p_1}$, this implies that there exists a $b-$dimensional (resp. $a-$dimensional) subspace $U_1\subset\real^{p_2}$ (resp. $W_1\subset\real^{p_1}$) such that $\real^{p_2}=U_1\oplus U_2$ and $\real^{p_1}=W_1\oplus W_2$ for which $A_i(U_j)\subseteq W_j$, $i\in[r]$ and $j=1,2$. For given a linear subspace $V$, let $\calP_V$ be the orthogonal projection onto $V$. Then it is obvious that $I-\calP_{W_1}$ (resp. $I-\calP_{U_1}$) is an orthogonal projection onto $W_2$ ($U_2)$. Thus,  
\begin{align}\label{prop3.1.eq1}
    \mathcal{P}_{U_1}A_i(I-\mathcal{P}_{W_1})=0,\quad (I-\mathcal{P}_{U_1})A_i\mathcal{P}_{W_1}=0.
\end{align}

Recall that the orthogonal projection into a $d-$dimensional linear subspace of $\real^n$ corresponds to a unique element in a projective variety $\text{Gr}(d,n)$ as every linear subspace has its unique orthogonal projection. Namely, suppose $Z$ denotes a $d\times n$ matrix whose rows denote the basis of a $d-$dimensional subspace. Then the orthogonal projection onto this subspace is given by $Z^\top\textsf{adj}(ZZ^\top)Z/|ZZ^\top|$, where $\textsf{adj}(M)$ is an adjoint of a square matrix $M$. By Theorem $2.1$ of \cite{devriendt2024}, each entry of $Z^\top\textsf{adj}(ZZ^\top)Z$ and $|ZZ^\top|$ can be expressed as a quadratic polynomial in Pl\"uker coordinates. Note that $|ZZ^\top|\neq 0$. Hence, multiplying the quadratic polynomial corresponding to $|ZZ^\top|$ in Pl\"uker coordinates to both hand sides of the equations in (\ref{prop3.1.eq1}), we see that (\ref{prop3.1.eq1}) induces the system of finitely many equations of polynomials in Pl\"uker coordinates and affine coordinates (the entries of $A_1,\ldots,A_r$). Hence, if $\tilde{\calV}_{p_1,p_2,r}^{(a,b)}$ is a subset of $\mathcal{Z}:=\real\bbP^{\binom{p_1}{a}-1}\times\real \bbP^{\binom{p_2}{b}-1} \times \real^{pr}$ for which (\ref{prop3.1.eq1}) is satisfied, then $\tilde{\calV}_{p_1,p_2,r}^{(a,b)}$ is Zariski-closed as it is exactly the zero set of finitely many polynomials over $\mathcal{Z}$. Furthermore, since there are $ab+(p_1-a)(p_2-b)$ free coordinates for each $A_i$ with fixed $\calP_{U_1},\calP_{W_1}$, 
\begin{align*}
    \dim \tilde{\calV}_{p_1,p_2,r}^{(a,b)}=\dim \text{Gr}(a,p_1)+\dim \text{Gr}(b,p_2)+r(ab+(p_1-a)(p_2-b))=m(a,b)
\end{align*}
for the map $m$ defined in Lemma \ref{sec4.1.lemma1} with $\alpha=p_1$ and $\beta=p_2$.

Now suppose $\pi:\real\bbP^{\binom{p_1}{a}-1}\times\real \bbP^{\binom{p_2}{b}-1} \times \real^{pr}\rightarrow \real^{pr}$ is a projection morphism. As discussed in Section \ref{sec2.4} (see Definition \ref{sec2.4.def1}), $\calV_{p_1,p_2,r}^{(a,b)}\equiv \pi(\tilde{ \calV}_{p_1,p_2,r}^{(a,b)})$ is Zariski-closed in $\real^{pr}\cong (\real^{p_1\times p_2})^r$. Also, by the first item of Lemma \ref{sec2.4.lemma1}, 
\begin{align*}
   \dim \calV_{p_1,p_2,r}^{(a,b)}\leq \dim \tilde{ \calV}_{p_1,p_2,r}^{(a,b)}=m(a,b),
\end{align*}
proving the claim.  
\end{proof}

\subsection{Proofs of the results from Section \ref{sec4.2}}\label{append.A.4}

\begin{proof}[Proof of Theorem \ref{sec4.2.thm1}]
   Following the sketch of the proof in Section \ref{sec4.2}, we first prove that $\tilde{X}=\varphi_{p_1,p_2,r}^{-1}(\varphi_{p_1,p_2,r}(X)O)$ is canonically indecomposable for any $X=(X_1,\ldots,X_r)\in \calD_{p_1,p_2,r}$ and $O\in\calO_r$ so that the action $(O,B)\in \calO_r\times \calC_{p_1,p_2,r}\rightarrow BO\in\calC_{p_1,p_2,r}$ is well-defined. Suppose otherwise. Then there exists $(P,Q)\in GL_{p_1}\times GL_{p_2}$ such that $(Q^{-\top}\otimes P)\varphi_{p_1,p_2,r}(\tilde{X})=[\text{vec}(Y_1),\ldots,\text{vec}(Y_r)]=:Y$, where each $Y_i$ takes a non-trivial block diagonal form. Note that $\varphi_{p_1,p_2,r}(\tilde{X})=(Q^{\top} \otimes P^{-1})YO^\top$. Since any linear combination of $Y_i$ also has a non-trivial block diagonal form, this implies that $PX_iQ^{-1}$ is of non-trivial block-diagonal form, contradicing the indecomposability of $X$. Hence, $\tilde{X}$ is indecomposable.  

 Next, we claim that $C(J_1(A)^\top)\cap C(J_2(A)^\top)=\set{\bfzero_{pr}}$ for any $A\in\calD_{p_1,p_2,r}$.Adopting the notations in the proof of Proposition \ref{sec3.prop1}, this is equivalent to show that $\tilde{\Lambda}_V=\mathbf{0}_{p_1\times p_1}$ and $\tilde{\Lambda}_W=\mathbf{0}_{p_2\times p_2}$ defined in (\ref{sec3.prop1.eq4}). Recall that (\ref{sec3.prop1.eq3}) implies that $((\tilde{\Lambda}_V)_{jj}-(\tilde{\Lambda}_W)_{kk})(\tilde{A}_i)_{jk}=0$ for any $(j,k)\in[p_1]\times [p_2]$. If there exists $i\in[r]$ for which $(\tilde{A}_i)_{jk}$ is nonzero, we have that $(\tilde{\Lambda}_V)_{jj}=(\tilde{\Lambda}_W)_{kk}$ for any given $(j,k)$. Hence, if $G=(\set{s_i:i\in[p_1]}\sqcup \set{q_j:j\in [p_2]},E)$ is the undirected bipartite graph induced from $(\tilde{A}_1,\ldots,\tilde{A}_r)$ as in Proposition \ref{sec4.1.prop2} (take $P=I_{p_1},Q=I_{p_2}$), we have that whenever a vertex $s_j$ is connected to $q_k$, $(\tilde{\Lambda}_V)_{jj}=(\tilde{\Lambda}_W)_{kk}$.  If we identify $(\tilde{\Lambda}_V)_{jj}$ and $(\tilde{\Lambda}_W)_{kk}$ as $s_j$ and $q_k$, respectively,  this implies  that for any $(j,k)$ such that $(\tilde{A}_i)_{jk}$ is nonzero for at least one $i$, $(\tilde{\Lambda}_V)_{jj}$ and $(\tilde{\Lambda}_W)_{kk}$ are the same as some constant. Since the graph $G$ is connected by Proposition \ref{sec4.1.prop2}, this implies that $(\tilde{\Lambda}_V)_{jj}=(\tilde{\Lambda}_V)_{kk}$ for any $j,k$. Hence, $\tilde{\Lambda}_V=\mathbf{0}_{p_1\times p_1}$ and $\tilde{\Lambda}_W=\mathbf{0}_{p_2\times p_2}$ and thus the first item holds. 
\end{proof}

In view of the proof of Theorem \ref{sec4.2.thm1}, it is clear why the map $F_{p_1,p_2,r}$ fails to be a submersion on the set $\mathcal{U}$ in Example \ref{sec4.1.ex1}, as $C(J_1(A)^\top)\cap C(J_2(A)^\top)\neq \set{\bfzero_{pr}}$ for any $A\in\mathcal{U}$. Thus, the set $\calV_{p_1,p_2,r}$ is the singularity in sense of Sard's theorem.  

\begin{proof}[Proof of Example \ref{sec4.1.ex1}]
For fixed $A\in \mathcal{U}$, simply write $J_1:=J_1(A)$ and $J_2:=J_2(A)$. We  claim that there is a nonzero element in $C(J_1^\top)\cap C(J_2^\top)$. Take $V=W=c_1 I_2\oplus [c_2]$ for different constants $c_1$ and $c_2$. Then one can verify that $J_1^\top v=J_2^\top w$ for $v=\text{vec}(V)$ and $w=\text{vec}(W)$, and is nonzero. Hence, if $F_{p_1,p_2,r}$ is defined on $\calH_{3,3,3}$ as in (\ref{sec3.eq1}), it fails to be a submersion on the subset $\calD_{3,3,3}$ in (\ref{sec3.eq1}).    
\end{proof}

\begin{proof}[Proof of Proposition \ref{sec4.2.prop1}]
    Deduce from the proof of Theorem \ref{sec4.2.thm1} that $T_A\calD_{p_1,p_2,r}$ is the vectorization of $\varphi_{p_1,p_2,r}^{-1}(N(J(A)))$. Since $\varphi_{p_1,p_2,r}$ is a diffeomorphism and $\calD_{p_1,p_2,r}$ is embedded in $(\real^{p_1\times p_2})_*^r$, the form of $ T_{\tilde{A}}\calC_{p_1,p_2,r}$ follows from the differential of $\varphi_{p_1,p_2,r}$ on $\calD_{p_1,p_2,r}$. Also, the proof of Lemma \ref{sec3.lemma3} implies that the canonical projection $\pi:X\in \real_*^{p\times r}\rightarrow [X]\in\real_*^{p\times r}/\calO_r$ is a smooth submersion when it is restricted to $\calO_r-$invariant embedded submanifold of $\real_*^{p\times r}$. Hence, recalling the diffeomorphism $s_{p_1,p_2,r}$ in (\ref{sec3.eq1}), the map $\Phi_{p_1,p_2,r}\equiv s_{p_1,p_2,r}\circ \pi:X\in \real_*^{p\times r}\rightarrow XX^\top \in \calS_{p,r}^+$ is a smooth submersion when restricted to $\calC_{p_1,p_2,r}$. Thus, $T_{\tilde{A}\tilde{A}^\top}\calC_{p_1,p_2,r}^+\equiv d\Phi_{p_1,p_2,r}(\tilde{A})[T_{\tilde{A}}\calC_{p_1,p_2,r}]$.
\end{proof}

\subsection{Proofs of the results from Section \ref{sec5.1}}\label{append.A.5}

\begin{proof}[Proof of Lemma \ref{sec5.1.lemma1}]
     Define a curve $\tilde{\gamma}_K(t)=K+tU$ on $(-\epsilon,\epsilon)$ for some sufficiently small $\epsilon>0$ so that the curve lies on $\calS_{p_1,p_2}^{++}$. Since $h(\tilde{\gamma}_K(t))h(\tilde{\gamma}_K(t))^\top=\tilde{\gamma}_K(t)$, letting $R=dh(K)[U]=\frac{d}{dt}h(\tilde{\gamma}_K(t))\bigg|_{t=0}$, 
\begin{align}\label{sec5.1.lemma1.eq1}
\begin{split}
    &h(\tilde{\gamma}_K(0))\parentheses{\frac{d}{dt}h(\tilde{\gamma}_K(t))\bigg|_{t=0}}^\top+\parentheses{\frac{d}{dt}h(\tilde{\gamma}_K(t))\bigg|_{t=0}}h(\tilde{\gamma}_K(0))^\top=\frac{d}{dt}\tilde{\gamma}_K(t)\bigg|_{t=0},\\
  \Rightarrow &  h(K)R^\top+Rh(K)^\top=U.
\end{split}
\end{align}
Suppose $h\in\calL_{p_1,p_2}^{++}$ so that $R\in T_{L_2\otimes L_1}\calL_{p_1,p_2}^{++}$. Then we have that 
\begin{align*}
    [(L_2^{-1}\otimes L_1^{-1})R]^\top+(L_2^{-1}\otimes L_1^{-1})R=I_{p_2}\otimes L_1^{-1}U_1L_1^{-\top}+L_2^{-1}U_2L_2^{-\top}\otimes I_{p_1}. 
\end{align*}
Because $(L_2^{-1}\otimes L_1^{-1})R$ is lower triangular while $I_{p_2}\otimes L_1^{-1}U_1L_1^{-\top}+L_2^{-1}U_2L_2^{-\top}\otimes I_{p_1}$ is symmetric, following the proof of Proposition $4$ from \cite{lin2019} yields that 
\begin{align*}
&(L_2^{-1}\otimes L_1^{-1})R=\parentheses{I_{p_2}\otimes L_1^{-1}U_1L_1^{-\top}+L_2^{-1}U_2L_2^{-\top}\otimes I_{p_1}}_{\frac{1}{2}}\\
  \Rightarrow & R=(L_2\otimes L_1)\parentheses{I_{p_2}\otimes L_1^{-1}U_1L_1^{-\top}+L_2^{-1}U_2L_2^{-\top}\otimes I_{p_1}}_{\frac{1}{2}}.
\end{align*}
Now assume $h\in \calS_{p_1,p_2}^{++}$ so that $R\in T_{S_2\otimes S_1}\calS_{p_1,p_2}^{++}$ for $S_2\otimes S_1=h(\Sigma_2\otimes \Sigma_1)$. Replacing $L_2$ and $L_1$ with $S_2$ and $S_1$ in above, we have the Sylvester's equation as  
\begin{align*}
   (S_2\otimes S_1)R+R(S_2\otimes S_1)=\Sigma_2\otimes U_1+U_2\otimes \Sigma_1.
\end{align*}
We follow the standard approach in solving the equation above over symmetric matrices with coefficients being positive definite. After some algebra, we have that 
\begin{align*}
    (\Lambda_2^{1/2}\otimes \Lambda_1^{1/2})(\tilde{R}_2\otimes \tilde{R}_1)+(\tilde{R}_2\otimes \tilde{R}_1)(\Lambda_2^{1/2}\otimes \Lambda_1^{1/2})=\Lambda_2\otimes \Gamma_1^\top U_1\Gamma_1+\Gamma_2^\top U_2\Gamma_2\otimes \Lambda_1,
\end{align*}
where $\tilde{R}=(\Gamma_2\otimes \Gamma_1)^\top R(\Gamma_2\otimes \Gamma_1)$. Through entry-wise comparison of the matrices in the above equation, we see that 
\begin{align*}
    R=(\Gamma_2\otimes \Gamma_1)\left[\Lambda^{-}\circ \parentheses{\Lambda_2\otimes \Gamma_1^\top U_1\Gamma_1+\Gamma_2^\top U_2\Gamma_2\otimes \Lambda_1}\right](\Gamma_2\otimes \Gamma_1)^\top.
\end{align*}
\end{proof}

\begin{proof}[Proof of Proposition \ref{sec5.1.prop1}]
 Define two curves $\tilde{\gamma}_K(t)=K+tU$ and $\tilde{\gamma}_C(t)=C+tW$ on $(-\epsilon,\epsilon)$ for some sufficiently small $\epsilon>0$ so that each curve is living on the desired manifold. Suppose $\tilde{g}(t)=g(\tilde{\gamma}_K(t),\tilde{\gamma}_C(t))$ on $(-\epsilon,\epsilon)$. Then 
   \begin{align*}
 \tilde{g}(t)=h(\tilde{\gamma}_K(t))\tilde{\gamma}_C(t)h(\tilde{\gamma}_K(t))^\top.
   \end{align*}
   Thus,
   \begin{align*}
       \frac{d\tilde{g}}{dt}=h(\tilde{\gamma}_K(t))\frac{d\tilde{\gamma}_C}{dt}h(\tilde{\gamma}_K(t))^\top+\frac{d}{dt}h(\tilde{\gamma}_K(t)) \tilde{\gamma}_C(t)h(\tilde{\gamma}_K(t))^\top+h(\tilde{\gamma}_K(t))\tilde{\gamma}_C(t)\frac{d}{dt}h(\tilde{\gamma}_K(t))^\top
   \end{align*}
and so
\begin{align*}
    \frac{d\tilde{g}}{dt}\bigg|_{t=0}&=dg(K,C)[U,W] \\
    &=h(K) W h(K)^\top+dh(K)[U] Ch(K)^\top+h(K)C(dh(K)[U])^\top.
\end{align*}
Here $dh(K)[U]$ is given in Lemma \ref{sec5.1.lemma1}.
\end{proof}

To prove Proposition \ref{sec5.1.prop2}, the following lemma is useful. 
\begin{lemma}\label{append.A.5.lemma1}
Suppose $A\in\calS_{p_1}$, $B\in\calS_{p_2}$, and $C\in\real^{p\times p}$. Then the followings are true:
\begin{align*}
    \tr\parentheses{C(B\otimes I_{p_1})}&=\tr\parentheses{([\tr_2(\sym(C))/p_1]\otimes I_{p_1})(B\otimes I_{p_1})}=\tr\parentheses{\tr_2(\sym(C))B},\\
      \tr\parentheses{C(I_{p_2}\otimes A)}&=\tr\parentheses{\left[(I_{p_2}\otimes \tr_1(\sym(C))/p_2)\right](I_{p_2}\otimes A)}=\tr\parentheses{\tr_1(\sym(C))A}.
\end{align*}
\end{lemma}
\begin{proof}
This is a direct consequence of the fact that $\tr(CA)=\tr(\sym(C)A)$ for any square matrix $C$ and a symmetric $A$, and Proposition $1.3$ from \cite{simonis2025}.     
\end{proof}

\begin{proof}[Proof of Proposition \ref{sec5.1.prop2}]
Note that $\Theta^{1/2}$ denotes the symmetric square root of $\Theta\in\calS_p^{++}$ throughout the proof. Provided that the Kronecker map $k$ is smooth, the differential of the core map $c$ follows from that of the Kronecker map $k$. To see this, take $t\in(-\epsilon,\epsilon)$ for sufficiently small $\epsilon>0$ so that a curve $\mu(t)=\Sigma+tV\in \calS_p^{++}$ for any such $t$. If the map $k$ is smooth, because $h(k(\mu(t)))h(k(\mu(t)))^\top=k(\mu(t))$, the analogy to (\ref{sec5.1.lemma1.eq1}) and the chain rule yield that  
\begin{align*}
    h(k(\Sigma))U^\top+Uh(k(\Sigma))^\top=dk(\Sigma)[V],
\end{align*}
where $U=dh(k(\Sigma))[dk(\Sigma)[V]]$. Because $k(\Sigma)\in\calS_{p_1,p_2}^{++}$ and $dk(\Sigma)[V]\in T_{k(\Sigma)}\calS_{p_1,p_2}^{++}$, $U$ can be computed using Lemma \ref{sec5.1.lemma1} after computing $dk(\Sigma)[V]$. Then the differential of the core map $c$ is given as  
\begin{align*}
    dc(\Sigma)[V]&\equiv\frac{d}{dt}\bigg|_{t=0}h(k(\mu(t)))^{-1}\mu(t)h(k(\mu(t)))^{-\top}\\
    &=h(k(\Sigma))^{-1}Vh(k(\Sigma))^{-\top}-h(k(\Sigma))^{-1}U h(k(\Sigma))^{-1}\Sigma h(k(\Sigma))^{-\top}\\
    &-h(k(\Sigma))^{-1}\Sigma h(k(\Sigma))^{-\top}U^\top  h(k(\Sigma))^{-\top}\\
    &=h(k(\Sigma))^{-1}Vh(k(\Sigma))^{-\top}-h(k(\Sigma))^{-1}UC-CU^\top  h(k(\Sigma))^{-\top},
\end{align*}
concluding the claim.

Now we prove the smoothness of $k$ as outlined in the preceding discussion of this proposition. To this end, we prove the strict geodesically convexity of the map $\eta_{p_1,p_2}$ over $\calE$. Take $\Omega=(\Omega_1,\Omega_2)\in\calE$ and a tangent vector $W=(W_1,W_2)\in T_\Omega\calE$. Then the geodesic $\gamma:[0,1]\rightarrow \calE$ emanating from $\Omega$ in the direction of $W$ is given by
\begin{align*}
   \gamma(t)&:=(\gamma_1(t),\gamma_2(t))\\
&=\parentheses{\Omega_1^{1/2}\exp\parentheses{t\Omega_1^{-1/2}W_1\Omega_1^{-1/2}}\Omega_1^{1/2},\Omega_2^{1/2}\exp\parentheses{t\Omega_2^{-1/2}W_2\Omega_2^{-1/2}}\Omega_2^{1/2}}
\end{align*}
For given $\Sigma\in\calS_p^{++}$, let $\theta(\cdot|\Sigma)=\ell(\gamma(\cdot)|\Sigma)$, where 
\begin{align}\label{sec5.1.prop2.eq1}
    \ell(\cdot|\Sigma):(K_1,K_2)\in \calE\mapsto \tr\parentheses{(K_2\otimes K_1)^{-1}\Sigma}+p_1\log |K_2|\in\real. 
\end{align}
Note that the map $\theta(\cdot|\Sigma)$ is smooth over $[0,1]$. By direct computations, for $\tilde{\gamma}(t)=\gamma_2(t)\otimes\gamma_1(t)$,   
\begin{align}\label{sec5.1.prop2.eq2}
\begin{split}
\theta(t|\Sigma)&=\ell(\gamma_1(t),\gamma_2(t)|\Sigma)=\tr\parentheses{\Sigma[\tilde{\gamma}(t)]^{-1}}+tp_1\tr\parentheses{\Omega_2^{-1/2}W_2\Omega_2^{-1/2}}+p_1\log|\Omega_2|,\\
    \theta'(t|\Sigma)&=-\tr\parentheses{\Sigma[\tilde{\gamma}(t)]^{-1}\tilde{\gamma}'(t)[\tilde{\gamma}(t)]^{-1}}+p_1\tr\parentheses{\Omega_2^{-1/2}W_2\Omega_2^{-1/2}},\\
    \theta''(t|\Sigma)&=2\tr\parentheses{\Sigma[\tilde{\gamma}(t)]^{-1}\tilde{\gamma}'(t)[\tilde{\gamma}(t)]^{-1}\tilde{\gamma}'(t)[\tilde{\gamma}(t)]^{-1}}-\tr\parentheses{\Sigma[\tilde{\gamma}(t)]^{-1}\tilde{\gamma}''(t)[\tilde{\gamma}(t)]^{-1}},
\end{split}
\end{align}
where $t\in(0,1)$. Using the facts that $\tau'(t)\tau(t)=\tau(t)\tau'(t)$ for
\begin{align*}
    \tau(t)=(\Omega_2\otimes\Omega_1)^{-1/2}\tilde{\gamma}(t)(\Omega_2\otimes\Omega_1)^{-1/2},
\end{align*}
observe that $\tilde{\gamma}''(t)=\tilde{\gamma}'(t)[\tilde{\gamma}(t)]^{-1}\tilde{\gamma}'(t)$. Hence,
\begin{align*}
    \theta''(t|\Sigma)&=\tr\parentheses{\Sigma[\tilde{\gamma}(t)]^{-1}\tilde{\gamma}'(t)[\tilde{\gamma}(t)]^{-1}\tilde{\gamma}'(t)[\tilde{\gamma}(t)]^{-1}}\\
    &=\tr\parentheses{[\tilde{\gamma}(t)]^{-1/2}\tilde{\gamma}'(t)[\tilde{\gamma}(t)]^{-1}\Sigma[\tilde{\gamma}(t)]^{-1}\tilde{\gamma}'(t)[\tilde{\gamma}(t)]^{-1/2}}\geq 0
\end{align*}
for any $t\in(0,1)$. Since both $\tilde{\gamma}(t)$ and $\Sigma$ are strictly positive definite, the equality in the inequality above holds only when $\tilde{\gamma}'(t)=0$. Noting that 
\begin{align*}
    \tilde{\gamma}'(t)&=(\Omega_2^{1/2}\otimes \Omega_1^{1/2})([\Omega_2^{-1/2}W_2\Omega_2^{-1/2}]\otimes I_{p_1}\\
    &+I_{p_2}\otimes [\Omega_1^{-1/2}W_1\Omega_1^{-1/2}])(\Omega_2^{-1/2}\otimes \Omega_1^{-1/2})\tilde{\gamma}(t),
\end{align*}
this can happen only when $[\Omega_2^{-1/2}W_2\Omega_2^{-1/2}]\otimes I_{p_1}+I_{p_2}\otimes [\Omega_1^{-1/2}W_1\Omega_1^{-1/2}]=0$, which holds only when 
\begin{align*}
    \Omega_2^{-1/2}W_2\Omega_2^{-1/2}=\alpha I_{p_2},\quad \Omega_1^{-1/2}W_1\Omega_1^{-1/2}=-\alpha I_{p_1}
\end{align*}
for some constant $\alpha$. Since $\tr\parentheses{\Omega_1^{-1}W_1}=0$, the above implies that $\alpha=0$ so that $\gamma(t)=(\Omega_1,\Omega_2)$, a trivial geodesic. Thus, whenever $\gamma$ is a non-trivial geodesic, we have that $\theta''(t|\Sigma)>0$ for all $t\in(0,1)$, implying that $\theta(\cdot|\Sigma)$ is strictly convex. Thus, the smooth map $\ell(\cdot|\Sigma)$ defined in (\ref{sec5.1.prop2.eq1}) is strictly geodesically convex on $\calE$. Per the preceding discussion of Proposition \ref{sec5.1.prop2},
\begin{align*}
    \eta_{p_1,p_2}(\Sigma)=(\Sigma_1,\Sigma_2)=\argmin_{(\Omega_1,\Omega_2)\in \calE} \ell(\Omega_1,\Omega_2|\Sigma).
\end{align*}
By the uniqueness of the minimizer $\ell(\cdot|\Sigma)$ over $\xi$, $(\Sigma_1,\Sigma_2)$ is the unique solution to the equation 
\begin{align*}
    \grad \ell(\Omega_1,\Omega_2|\Sigma)=0
\end{align*}
in $\Omega=(\Omega_1,\Omega_2)\in\calE$. Letting $m(\Omega,\Sigma):=\grad \ell(\Omega_1,\Omega_2|\Sigma)$, it is clear that $m$ is a smooth map on $\calE\times \calS_p^{++}$. Moreover, since the operator $\Hess \ell(\Omega_1,\Omega|\Sigma)$ is invertible as $\ell$ is strictly geodesically convex,  the manifold  implicit function theorem (\cite{loomis2014}, Section $3.11$) implies that $\eta_{p_1,p_2}$ is smooth in $\Sigma$ and its differential $d\eta_{p_1,p_2}(\Sigma)[V]$ can be computed as
\begin{align}\label{sec5.1.prop2.eq3}
    d\eta_{p_1,p_2}(\Sigma)[V]=-\Hess^{-1} \ell(\Sigma_1,\Sigma_2|\Sigma)\left[\frac{d}{dt}\bigg|_{t=0}m((\Sigma_1,\Sigma_2),\Sigma+tV)\right].
\end{align}
By (\ref{sec5.1.eq1}), the Kronecker map $k=\psi_{p_1,p_2}^{-1}\circ \eta_{p_1,p_2}$ is also smooth and its differential can be derived according to (\ref{sec5.1.eq2}) and (\ref{sec5.1.prop2.eq3}).

It remains to derive the differential of $\eta_{p_1,p_2}$. Take $\Omega=(\Sigma_1,\Sigma_2)$ and $W=(W_1,W_2)\in T_{(\Sigma_1,\Sigma_2)}\calE$. With $\tilde{W}=\Sigma_2\otimes W_1+W_2\otimes \Sigma_1$, 
\begin{align*}
    \theta''(0|\Sigma)&=\Hess\ell(\Sigma_1,\Sigma_2)[W,W]\\
    &=\tr\parentheses{ (\Sigma_2^{-1/2}\otimes \Sigma_1^{-1/2})\tilde{W}(\Sigma_2^{-1/2}\otimes \Sigma_1^{-1/2}
)C(\Sigma_2^{-1/2}\otimes \Sigma_1^{-1/2})\tilde{W}(\Sigma_2^{-1/2}\otimes \Sigma_1^{-1/2})}.
\end{align*}
Here if $h\in \calL_{p_1,p_2}^{++}$, simply replace $C$ with $(O_2\otimes O_1)C(O_2\otimes O_1)^\top\in\calC_{p_1,p_2}^{++}$ for $O_2\otimes O_1=(\Sigma_2^{-1/2}\otimes \Sigma_1^{-1/2})(L_2\otimes L_1)\in\calO_{p_1,p_2}$ for $L_2\otimes L_1=\calL(\Sigma)$. Choose any tangent vector $Y=(Y_1,Y_2)\in T_{(\Sigma_1,\Sigma_2)}\calE$ and let $    \tilde{Y}=\Sigma_2\otimes Y_1+Y_2\otimes \Sigma_1$. By polarization, if $M=(\Sigma_2^{-1/2}\otimes \Sigma_1^{-1/2})\tilde{W}(\Sigma_2^{-1/2}\otimes \Sigma_1^{-1/2}
)C$, 
\begin{align*}
    \Hess\ell(\Sigma_1,\Sigma_2)[W,Y]&=\tr\parentheses{ M(\Sigma_2^{-1/2}\otimes \Sigma_1^{-1/2})\tilde{Y}(\Sigma_2^{-1/2}\otimes \Sigma_1^{-1/2})}\\
&=\tr\parentheses{M(I_{p_2}\otimes\Sigma_1^{-1/2}Y_1\Sigma_1^{-1/2})}+\tr\parentheses{M(\Sigma_2^{-1/2}Y_2\Sigma_2^{-1/2}\otimes I_{p_1})}\\
&=\tr\parentheses{\tr_1(\sym(M))\Sigma_1^{-1/2}Y_1\Sigma_1^{-1/2}}\\
&+\tr\parentheses{\tr_2(\sym(M))\Sigma_2^{-1/2}Y_2\Sigma_2^{-1/2}}.
\end{align*}
where the last equality follows from Lemma \ref{append.A.5.lemma1}. By (\ref{sec2.2.eq1}),  $\Hess\ell(\Sigma_1,\Sigma_2)[W]=(X_1,X_2)=:X\in T_{(\Sigma_1,\Sigma_2)}\calE$ is a unique tangent vector such that 
\begin{align*}
    \Hess\ell(\Sigma_1,\Sigma_2)[W,Y]&=\tilde{g}^\AI (X,Y)=\tilde{g}_1^\AI (X_1,Y_1)+\tilde{g}_2^\AI (X_2,Y_2).
\end{align*}
To identify $X$, observe that 
\begin{align*}
    \sym(M)&=\underbrace{\left[(I_{p_2}\otimes \Sigma_1^{-1/2}W_1\Sigma_1^{-1/2})C+C(I_{p_2}\otimes \Sigma_1^{-1/2}W_1\Sigma_1^{-1/2})\right]/2}_{M_1(W_1)}\\
    &+\underbrace{\left[(\Sigma_2^{-1/2}W_2\Sigma_2^{-1/2}\otimes I_{p_1})C+C(\Sigma_2^{-1/2}W_2\Sigma_2^{-1/2}\otimes I_{p_1})\right]/2}_{M_2(W_2)}.
\end{align*}
Recalling that $\tr_1(C)=p_2I_{p_1}$ and $\tr_2(C)=p_1 I_{p_2}$, applying Lemma \ref{append.A.5.lemma1} yields that 
\begin{align*}
 \tr_1\parentheses{M_1(W_1)}&=p_2 \Sigma_1^{-1/2}W_1\Sigma_1^{-1/2},\; \tr_2\parentheses{M_2(W_2)}=p_1 \Sigma_2^{-1/2}W_2\Sigma_2^{-1/2},\\
\end{align*}
 and also 
 \begin{align*}
 \tr_1\parentheses{M_2(W_2)}&=\sum_{i,j=1}^{p_2}(\Sigma_2^{-1/2}W_2\Sigma_2^{-1/2})_{ij} C_{[j,i]},\,
 [\tr_2(M_1(W_1))]_{ij}\\
 &=\tr\parentheses{C_{[i,j]}\Sigma_1^{-1/2}W_1\Sigma_1^{-1/2}}
\end{align*}
for $i,j\in [p_2]$. Hence,
\begin{align*}
    \Hess\ell(\Sigma_1,\Sigma_2)[W,Y]&=p_2\tr\parentheses{\Sigma_1^{-1}W_1\Sigma_1^{-1}Y_1}+p_1\tr\parentheses{\Sigma_2^{-1}W_2\Sigma_2^{-1}Y_2}+p_2\tr\parentheses{\Sigma_1^{-1}Y_1\Sigma_1^{-1}\Sigma_1^{1/2}\tr_1(M_2(W_2))\Sigma_1^{1/2}/p_2}\\
    &+p_1\tr\parentheses{\Sigma_2^{-1}Y_2\Sigma_2^{-1}\Sigma_2^{1/2}\tr_2(M_1(W_1))\Sigma_2^{1/2}/p_1}\\
    &=\tilde{g}_1^{\AI}(X_1,Y_1)+\tilde{g}_2^{\AI}(X_2,Y_2).
\end{align*}
By (32) of \cite{simonis2025}, 
\begin{align*}
  \calP_{\Sigma_1}(\Sigma_1^{1/2}\tr_1(M_2(W_2))\Sigma_1^{1/2})\equiv   \Sigma_1^{1/2}\tr_1(M_2(W_2))\Sigma_1^{1/2}-\tr\parentheses{\Sigma_2^{-1/2}W_2\Sigma_2^{-1/2}}\Sigma_1
\end{align*}
for the operator $\calP$ defined in (\ref{sec2.2.1.eq1}). Therefore,
\begin{align}\label{sec5.1.prop1.eq4}
\begin{split}
        X_1&=W_1+\Sigma_1^{1/2}\tr_1(M_2(W_2))\Sigma_1^{1/2}/p_2-\tr\parentheses{\Sigma_2^{-1/2}W_2\Sigma_2^{-1/2}}\Sigma_1/p_2,\\
    X_2&=W_2+\Sigma_2^{1/2}\tr_2(M_1(W_1))\Sigma_2^{1/2}/p_1,
\end{split}
\end{align}
 we have that the operator $\calR_C$ that maps $T_{(\Sigma_1,\Sigma_2)}\calE$ to itself, as    
 \begin{align*}
     \calR_C(W):=\Hess\ell(\Sigma_1,\Sigma_2)[W]=(X_1,X_2)
 \end{align*}
for $(X_1,X_2)$  depending on $(W_1,W_2)$ defined in (\ref{sec5.1.prop1.eq4}), is invertible as it is the Riemannian Hessian operator of a strictly geodesically convex map. In particular, if $C=I_p$, the argument above yields that 
\begin{align*}
 \calR_C(W)\equiv \text{Id}(W)=(W_1,W_2).
\end{align*}

It remains to compute $m((\Sigma_1,\Sigma_2),V)$. Note that $m((\Sigma_1,\Sigma_2),V)$ is the unique tangent vector $(V_1,V_2)\in T_{(\Sigma_1,\Sigma_2)}\calE$ such that 
\begin{align*}
\frac{d}{ds}\bigg|_{s=0}\theta'(0|\Sigma+sV)&=\tilde{g}^{\AI}((W_1,W_2),(V_1,V_2))\\
&=-\tr\parentheses{V(\Sigma_2^{-1/2}W_2\Sigma_2^{-1/2}\otimes I_{p_1})}-\tr\parentheses{V(I_{p_2}\otimes \Sigma_1^{-1/2}W_1\Sigma_1^{-1/2})}\\
&=:-(\rom{1})-(\rom{2})
\end{align*}
for any $(W_1,W_2)\in T_{(\Sigma_1,\Sigma_2)}\calE$. Here $\theta'(t|\Sigma)$ is that defined in (\ref{sec5.1.prop2.eq2}) but with $(\Omega_1,\Omega_2)=(\Sigma_1,\Sigma_2)$. By Lemma \ref{append.A.5.lemma1}, 
\begin{align*}
(\rom{1})&=p_1\tr\parentheses{\Sigma_2^{-1}W_2\Sigma_2^{-1}\left[\Sigma_2^{1/2}(\tr_2(V)/p_1)/\Sigma_2^{1/2}\right]},\\
(\rom{2})&=p_2\tr\parentheses{\Sigma_1^{-1}W_1\Sigma_1^{-1}\left[\Sigma_1^{1/2}(\tr_1(V)/p_2)\Sigma_1^{1/2}\right]}.
\end{align*}
Again by (32) of \cite{simonis2025}, we have that 
\begin{align*}
    \tilde{g}^\AI ((W_1,W_2),(V_1,V_2))&=\tilde{g}_1^\AI(W_1,V_1)+\tilde{g}_2^\AI(W_2,V_2)\\
    &=\tilde{g}_1^\AI \parentheses{-\Sigma_1^{1/2}\parentheses{\tr_1(V)/p_2-\frac{\tr(V)}{p}I_{p_1}}\Sigma_1^{1/2},W_1}\\
    &+\tilde{g}_2^\AI \parentheses{-\Sigma_2^{1/2}\tr_2(V)\Sigma_2^{1/2}/p_1,W_2},
\end{align*}
implying that 
\begin{align*}
    (V_1,V_2)=\parentheses{-\Sigma_1^{1/2}\parentheses{\tr_1(V)/p_2-\frac{\tr(V)}{p}I_{p_1}}\Sigma_1^{1/2},-\Sigma_2^{1/2}\tr_2(V)\Sigma_2^{1/2}/p_1}.
\end{align*}
Therefore, if $(U_1,U_2)=d\eta_{p_1,p_2}(\Sigma)[V]\in T_{(\Sigma_1,\Sigma_2)}\calE$, we have that 
\begin{align*}
\calR_C(U_1,U_2)=-\parentheses{V_1,V_2},
\end{align*}
which admits a unique solution $(U_1,U_2)$ as $\calR_C:T_{(\Sigma_1,\Sigma_2)}\calE\mapsto T_{(\Sigma_1,\Sigma_2)}\calE$ is a bijection. For such $(U_1,U_2)$, by (\ref{sec5.1.eq2}), we have that 
\begin{align*}
    dk(\Sigma)[V]=U_2\otimes \Sigma_1+\Sigma_2\otimes U_1. 
\end{align*}
Again, if $C=I_p$, we have that $(U_1,U_2)=-(V_1,V_2)$, concluding the proof.
\end{proof}

\subsection{Proofs of the results from Section \ref{sec5.2}}\label{append.A.6}
\vspace{0.3cm}

\begin{proof}[Proof of Lemma \ref{sec5.2.lemma1}]
We first verify that both $\tr_1(\calG(V))$ and $\tr_2(\calG(V))$ are zero matrices. 
Note that for any symmetric $U_i\in\calS_{p_i}$,
\begin{align}
    \tr_1(U_2\otimes I_{p_1})&=\tr(U_2)I_{p_1},\quad \tr_2(I_{p_2}\otimes U_1)=\tr(U_1)I_{p_2}.
\end{align}
Also, $\tr(\tr_1(V))=\tr(\tr_2(V))=\tr(V)$. Thus, by Lemma \ref{append.A.5.lemma1},
\begin{align*}
    \tr_1(\calG(V))&=\tr_1(V)-\tr_1(V)-\frac{\tr(V)}{p_1}I_{p_1}+\frac{\tr(V)}{p_1}I_{p_1}=\bfzero_{p_1\times p_1},\\
   \tr_2(\calG(V)) &=\tr_2(V)-\frac{\tr(V)}{p_2}I_{p_2}-\tr_2(V)+\frac{\tr(V)}{p_2}I_{p_2}=\bfzero_{p_2\times p_2}. 
\end{align*}
Next, we verify that for any $W\in T_C\calC_{p_1,p_2}^{++}$ and $V\in \calS_p$, $W$ and $V-\calG(V)$ are orthogonal under the metric $g^E$. Again by Lemma \ref{append.A.5.lemma1}, 
\begin{align*}
    g^E(W,V-\calG(V))&=\tr\parentheses{W(\tr_2(V)\otimes I_{p_1})}/p_1+\tr\parentheses{W(I_{p_2}\otimes \tr_1(V))}/p_2-\tr(V)\tr\parentheses{W}/p\\
    &=\tr(\tr_2(W)\tr_2(V))/p_1+\tr(\tr_1(W)\tr_1(V))/p_2-\tr(V)\tr(W)/p\\
    &=0,
\end{align*}
where the last equality holds because $\tr_i(W)=\bfzero_{p_i\times p_i}$ and $\tr(W)=0$.
\end{proof}
\\

\begin{proof}[Proof of Proposition \ref{sec5.2.prop1}]
By (3.37) of \cite{absil2008}, $\grad f(C)=\calG(\nabla f(C))$. Also, by (5.37) of \cite{absil2008},  
\begin{align*}
g^E(\Hess f(C)[V],W)=\tr\parentheses{\Hess f(C)[V] W}=\tr\parentheses{\nabla^2f(C)[V] W}.
\end{align*}
for any $W,V\in T_C\calC_{p_1,p_2}^{++}$. Again by Lemma \ref{sec5.2.lemma1},  $\Hess f(C)[V]=\calG(\nabla^2f(C)[V])$. 
\end{proof}

\subsection{Proofs of the results from Section \ref{sec5.3}}\label{append.A.7}
\vspace{0.3cm}

\begin{proof}[Proof of Lemma \ref{sec5.3.lemma1}]
From the proof of Theorem \ref{sec4.2.thm1}, one can deduce that $U\in T_A\calC_{p_1,p_2,r}$ if and only if $u:=\text{vec}(U)\in N(J(\tilde{A}))$ for the linear operator $J$ defined in (\ref{sec4.2.eq1}) and $\tilde{A}=\varphi_{p_1,p_2,r}^{-1}(A)$. Note that $I-J(\tilde{A})^\dagger J(\tilde{A})$ is an orthogonal projection onto $N(J(\tilde{A}))$. Hence, for any $V\in T_A\real_*^{p\times r}\equiv \real^{p\times r}$, with $v:=\text{vec}(V)$,
\begin{align*}
    g^E(V,U)=\tr(V^\top U)=v^\top u=v^\top (I-J(\tilde{A})^\dagger J(\tilde{A}))^\top u
\end{align*}
as $u\in N(J(\tilde{A}))$. Thus, taking $W=\textsf{mat}_{p\times r}( (I-J(\tilde{A})^\dagger J(\tilde{A}))v)$, we see that $W$ is an orthogonal projection of $V$ onto $T_A\calC_{p_1,p_2,r}$.
\end{proof}

\begin{proof}[Proof of Proposition \ref{sec5.3.prop1}]
  As an analogy to the proof of Proposition \ref{sec5.2.prop1}, using the result of Lemma \ref{sec5.3.lemma1} and (3.37) of \cite{absil2008}, the Riemannian gradient is obvious. To derive the Riemannian Hessian operator, by (5.37) of \cite{absil2008}, one can observe that $\text{vec}(\Hess f(A)[V])$ is an orthogonal projection of $ \frac{d}{dt}  \text{vec}(\grad f(A+tV))|_{t=0}$ onto $N(J(\tilde{A}))$. Suppose $\Omega$ is an open subset of $\real^{a\times b}$ and let $Q$ be a smooth function on $\Omega$ taking values in $\real^{m\times n}$ such that $Q(X)$ has a constant rank across $X \in \Omega$. Given $Z\in\Omega$, write $\calP(Z):=Z^\dagger$. By Theorem $4.3$ of \cite{golub1973}, if $Y\in \real^{a\times b}$,
  \begin{align}\label{sec5.3.prop1.eq1}
  \begin{split}
           dQ(X)^\dagger [Y]&=-Q(X)^\dagger(dQ(X)[Y])Q(X)^\dagger\\
      &+Q(X)^\dagger Q(X)(dQ(X)[Y])^\top (I-Q(X)Q(X)^\dagger)\\
      &+(I-Q(X)^\dagger Q(X))(dQ(X)[Y])^\top (Q(X)^\dagger)^\top  Q(X)^\dagger.
  \end{split}
  \end{align}
  By Theorem \ref{sec4.2.thm1}, $J$ has a constant rank over $\calD_{p_1,p_2,r}$. Also, recall that $J$ is a linear operator. Thus, by the chain rule,
\begin{align}\label{sec5.3.prop1.eq2}
\begin{split}
     \frac{d}{dt}  \text{vec}(\grad f(A+tV))\bigg|_{t=0}&=(I-J(\tilde{A})^\dagger J(\tilde{A}))\text{vec}(\nabla^2 f(A)[V])\\
 &-J(\tilde{A})^\dagger J(\tilde{V}) \text{vec}(\nabla f(A))\\
 &-d\calP(J(\tilde{A}))[J(\tilde{V})]J(\tilde{A})\text{vec}(\nabla f(A)).
\end{split}
\end{align}
Noting that $J(\tilde{A})^\dagger J(\tilde{A})J(\tilde{A})^\dagger=J(\tilde{A})^\dagger$ and $J(\tilde{A})J(\tilde{A})^\dagger J(\tilde{A})=J(\tilde{A})$, by (\ref{sec5.3.prop1.eq1}),
\begin{align}\label{sec5.3.prop1.eq3}
\begin{split}
        &(I-J(\tilde{A})^\dagger J(\tilde{A}))(d\calP(J(\tilde{A}))[J(\tilde{V})])\text{vec}(\nabla f(A))\\
    =&(I-J(\tilde{A})^\dagger J(\tilde{A}))(J(\tilde{V}))^\top (J(\tilde{A})^\dagger)^\top J(\tilde{A})^\dagger J(\tilde{A})\text{vec}(\nabla f(A)).
\end{split}
\end{align}
Also, 
\begin{align}\label{sec5.3.prop1.eq4}
    \begin{split}
        &(I-J(\tilde{A})^\dagger J(\tilde{A}))(I-J(\tilde{A})^\dagger J(\tilde{A}))\text{vec}(\nabla^2 f(A)[V])=(I-J(\tilde{A})^\dagger J(\tilde{A}))\text{vec}(\nabla^2 f(A)[V]),\\
        &(I-J(\tilde{A})^\dagger J(\tilde{A}))J(\tilde{A})^\dagger J(\tilde{V}) \text{vec}(\nabla f(A))=0.
    \end{split}
\end{align}
Combining (\ref{sec5.3.prop1.eq2})--(\ref{sec5.3.prop1.eq4}), 
\begin{align*}
    \text{vec}(\Hess f(A)[V])&=(I-J(\tilde{A})^\dagger J(\tilde{A}))\frac{d}{dt}  \text{vec}(\grad f(A+tV))\bigg|_{t=0}\\
    &=(I-J(\tilde{A})^\dagger J(\tilde{A}))\text{vec}(\nabla^2 f(A)[V])\\
    &-(I-J(\tilde{A})^\dagger J(\tilde{A}))(J(\tilde{V}))^\top (J(\tilde{A})^\dagger)^\top J(\tilde{A})^\dagger J(\tilde{A})\text{vec}(\nabla f(A)).
\end{align*}
\end{proof}

\subsection{Proofs of the results from Section \ref{sec6}}

\begin{proof}[Proof of Proposition \ref{sec6.prop1}]
Since if part is obvious, we prove the only if part. Suppose $\Omega(\tau_1)=\Omega(\tau_2)$. Then 
\begin{align*}
    k(\Omega(\tau_1))&=k(\Omega(\tau_2))=\bar{K}^1(\bar{K}^1)^\top=\bar{K}^2(\bar{K}^2)^\top,\\
    c(\Omega(\tau_1))&=c(\Omega(\tau_2))=(1-\lambda^1)A^1(A^1)^\top+\lambda^1 I_p=(1-\lambda^2)A^2(A^2)^\top+\lambda^2 I_p.
\end{align*}
Here $\bar{K}^i=\nu^i(\bar{K}_2^i\otimes \bar{K}_1^i)$. Since the square root map $h$ associated with the maps $k$ and $c$ is bijective, we have that $\bar{K}^1=\bar{K}^2$ and under the unit determinant constraint, $(\bar{K}_1^i,\bar{K}_2^i,\nu^i)$ is identifiable. Lastly, comparing the non-spiked eigenvalues of the core, we have that $\lambda^1=\lambda^2$ and so $A^1(A^1)^\top=(A^2)(A^2)^\top$, implying that $A^1=A^2O$ for some $O\in\calO_r$. From the proof of Theorem \ref{sec4.2.thm1}, the smooth action of $\calO_r$ on $\calC_{p_1,p_2,r}$ via the right matrix multiplication is well-defined, concluding the proof.
\end{proof}

\section{Formulas of Euclidean derivative and Hessian operator}\label{append.B}
We provide formulas of Euclidean derivative and Hessian operator of the negative log-likelihood $\ell$ defined in (\ref{sec6.eq2}), for each of parameters in $\set{\bar{K}_1,\bar{K}_2,A}$. For $\theta$ among these parameters, we write the derivative and Hessian operator of $\ell$ with respect to $\theta$ by $\partial_{\theta}\ell$ and $\partial_{\theta}^2\ell[V]$, where $V$ is the tangent vector in the manifold on which $\theta$ is living. We introduce the following ancillary quantities:
\begin{align}\label{append.B.eq1}
    \begin{split}
      \alpha_i&=\sigma_i^2(A)/((1-\lambda)\sigma_i^2(A)+\lambda),\quad \tilde{S}=(\bar{K}_2\otimes \bar{K}_1)^{-1}S(\bar{K}_2\otimes \bar{K}_1)^{-\top}/\nu^2,\\
    u_i&: i\text{th top left singular vector of } A\, (i\in [r]),\\
    U_i&=\textsf{mat}_{p_1\times p_2}(u_i),\quad \tilde{C}=(1-\lambda)AA^\top+\lambda I_p,\\
    W_{1,i}&=\bar{K}_1^{-\top}\bar{K}_1^{-1}Y_i \bar{K}_2^{-\top}\bar{K}_2^{-1}Y_i^\top \bar{K}_1^{-\top},\quad W_{1,i,j}=U_j\bar{K}_2^{-1}Y_i^\top \bar{K}_1^{-\top},\\
     W_{2,i}&=\bar{K}_2^{-\top}\bar{K}_2^{-1}Y_i^\top \bar{K}_1^{-\top}\bar{K}_1^{-1}Y_i\bar{K}_2^{-\top},\quad  W_{2,i,j}=U_j^\top \bar{K}_1^{-1}Y_i\bar{K}_2^{-\top}.\\
    \end{split}
\end{align}
Then the Euclidean derivative and Hessian operator of $\ell$ follows from standard facts in matrix calculus. 
\begin{prop}\label{append.B.prop1}
Recall the negative log-likelihood $\ell$ defined in (\ref{sec6.eq2}). Let $\bbF$ denote $\sym$ (resp. $\bbL$) if $\bar{K}_i\in \bbP(\calS_{p_i}^{++})$ (resp. $\bar{K}_i\in \bbP(\calL_{p_i}^{++})$). Also, suppose $V$ is the tangent vector in the manifold on which the parameter among $\set{\bar{K}_1,\bar{K}_2,A}$ is living. With the quantities defined in (\ref{append.B.eq1}), the followings are true:   

    \begin{align*}
        \partial_{\bar{K}_1}\ell&=-\frac{2}{n\lambda\nu^2}\sum_{i=1}^n \bbF(W_{1,i})+\frac{2(1-\lambda)}{n\lambda\nu^2}\sum_{i=1}^n\sum_{j=1}^r \alpha_j\tr\parentheses{W_{1,i,j}}\bbF(\bar{K}_1^{-\top}W_{1,i,j}),\\
    \partial_{\bar{K}_2}\ell&=-\frac{2}{n\lambda\nu^2}\sum_{i=1}^n \bbF(W_{2,i})+\frac{2(1-\lambda)}{n\lambda\nu^2}\sum_{i=1}^n\sum_{j=1}^r \alpha_j\tr\parentheses{W_{2,i,j}}\bbF(\bar{K}_2^{-\top}W_{2,i,j}),\\
    \partial_{A}\ell&=-2(1-\lambda)\tilde{C}^{-1}\tilde{S}\tilde{C}^{-1}A+2(1-\lambda)\tilde{C}^{-1}A,\\
\end{align*}
and
\begin{align*}
        \partial_{\bar{K}_1}^2 \ell[V]&=\frac{2}{n\lambda\nu^2}\sum_{i=1}^n \bbF(\bar{K}_1^{-\top}V^\top W_{1,i}+W_{1,i}V^\top \bar{K}_1^{-\top}+\bar{K}_1^{-\top}\bar{K}_1^{-1}V\bar{K}_1^\top W_{1,i})\\
        &-\frac{2(1-\lambda)}{n\lambda\nu^2}\sum_{i=1}^n\sum_{j=1}^r \alpha_j\tr\parentheses{W_{1,i,j}V^\top \bar{K}_1^{-\top}}\bbF(\bar{K}_1^{-\top}W_{1,i,j}),\\
         &-\frac{2(1-\lambda)}{n\lambda\nu^2}\sum_{i=1}^n\sum_{j=1}^r \alpha_j\tr\parentheses{W_{1,i,j}}\bbF(\bar{K}_1^{-\top}V^\top \bar{K}_1^{-\top} W_{1,i,j}+\bar{K}_1^{-\top} W_{1,i,j}V^\top \bar{K}_1^{-\top}),\\
       \partial_{\bar{K}_2}^2 \ell[V]&=\frac{2}{n\lambda\nu^2}\sum_{i=1}^n \bbF(\bar{K}_2^{-\top}V^\top W_{2,i}+W_{2,i}V^\top \bar{K}_2^{-\top}+\bar{K}_2^{-\top}\bar{K}_2^{-1}V\bar{K}_2^\top W_{2,i})\\
        &-\frac{2(1-\lambda)}{n\lambda\nu^2}\sum_{i=1}^n\sum_{j=1}^r \alpha_j\tr\parentheses{W_{2,i,j}V^\top \bar{K}_2^{-\top}}\bbF(\bar{K}_2^{-\top}W_{2,i,j}),\\
         &-\frac{2(1-\lambda)}{n\lambda\nu^2}\sum_{i=1}^n\sum_{j=1}^r \alpha_j\tr\parentheses{W_{2,i,j}}\bbF(\bar{K}_2^{-\top}V^\top \bar{K}_2^{-\top} W_{2,i,j}+\bar{K}_2^{-\top} W_{2,i,j}V^\top \bar{K}_2^{-\top}),\\
              \partial_{A}^2 \ell[V]&=-2(1-\lambda)\tilde{C}^{-1}\tilde{S}\tilde{C}^{-1}V+2(1-\lambda)\tilde{C}^{-1}V\\
              &+2(1-\lambda)^2\tilde{C}^{-1}(AV^\top+VA^\top)\tilde{C}^{-1}\tilde{S}\tilde{C}^{-1}A\\
              &+2(1-\lambda)^2\tilde{C}^{-1}\tilde{S}\tilde{C}^{-1}(AV^\top+VA^\top)\tilde{C}^{-1}A\\
              &-2(1-\lambda)^2\tilde{C}^{-1}(AV^\top+VA^\top)\tilde{C}^{-1}A.
\end{align*}
\end{prop}
\begin{proof}
    The results follow from some tedious algebra (see \cite{peterson2012} for example).
\end{proof}

\bibliographystyle{plain}
\bibliography{riemmCore}

\end{document}